\documentclass[12pt,a4paper]{article}
\usepackage{amssymb}
\usepackage{amsmath, amscd}
\usepackage{amsthm}
\usepackage{mathtools}
\usepackage{mathrsfs}
\usepackage{booktabs}
\usepackage{perpage}
\usepackage[all,cmtip]{xy}
\usepackage{setspace}
\usepackage{amsbsy}
\usepackage[T1]{fontenc}
\usepackage{verbatim}
\usepackage{mathdots}
\usepackage{mathrsfs}
\usepackage{ifthen}
\usepackage{blkarray}
\usepackage{multirow}
\usepackage{enumerate}
\usepackage[dvipsnames]{xcolor}
\usepackage{tikz}
\usepackage{fullpage}
\usepackage{colonequals}
\usepackage{braket}
\usepackage{lipsum}
\usepackage{float}

\usepackage{graphicx}

\makeatletter
\newcommand{\sigmaop}[1]{\mathop{\mathpalette\@sigmaop{#1}}\slimits@}
\newcommand{\@sigmaop}[2]{%
  \vphantom{\sum}%
  \sbox\z@{$\m@th#1\sum$}%
  \dimen@=\ht\z@ \advance\dimen@\dp\z@
  \dimen\tw@=\wd\z@
  \ifx#1\displaystyle\dimen@=.9\dimen@\fi
  \ooalign{%
    \hidewidth
    $\vcenter{\hbox{$\m@th#1#2$\kern.3\dimen\tw@}%
     \ifx#1\scriptstyle\kern-.25ex\fi}$\hidewidth\cr
    $\vcenter{\hbox{%
      \resizebox{!}{\dimen@}{$\m@th\boxtimes$}%
    }\ifx#1\scriptstyle\kern-.25ex\fi}$\cr
  }%
}
\makeatother

\mathtoolsset{showonlyrefs=true}

\numberwithin{equation}{subsection}

\newtheorem{theorem}{Theorem}[subsection]
\newtheorem{lemma}[theorem]{Lemma}

\newtheorem{expect}[theorem]{Expectation}

\newtheorem{conjecture}[theorem]{Conjecture}

\newtheorem{corollary}[theorem]{Corollary}
\newtheorem{definition}[theorem]{Definition}
\newtheorem{question}[theorem]{Question}

\newtheorem{proposition}[theorem]{Proposition}

\newtheorem{assumption}[theorem]{Assumption}

\newtheorem*{thm1}{Theorem 1}
\newtheorem*{thm2}{Theorem 2}
\newtheorem*{thm3}{Theorem 3}

\newtheorem*{conj1}{Conjecture 1}

\theoremstyle{remark}

\newtheorem{rmk}[theorem]{Remark}

\newtheorem{exa}[theorem]{Example}

\newcommand{\GZip}{\mathop{\text{$G$-{\tt Zip}}}\nolimits}

\newcommand{\GpZip}{\mathop{\text{$G'$-{\tt Zip}}}\nolimits}
\newcommand{\GF}{\mathop{\text{$G$-{\tt ZipFlag}}}\nolimits}

\newcommand{\VB}{\mathfrak{VB}}

\newskip\procskipamount
\newskip\interskipamount
\newskip\refskipamount
\newcommand{\procskip}{\vskip\procskipamount}
\newcommand{\interskip}{\vskip\interskipamount}
\newcommand{\refskip}{\vskip\refskipamount}

\newcommand{\procbreak}{\par
   \ifdim\lastskip<\procskipamount\removelastskip
   \penalty-100
   \procskip\fi
   \noindent\ignorespaces}

\newcommand{\titlebreak}{\par%
\ifdim\lastskip<\interskipamount\removelastskip%
\penalty10000%
\interskip\fi%
\noindent}%

\newcommand{\interbreak}{\par%
\ifdim\lastskip<\interskipamount\removelastskip%
\penalty-100%
\interskip\fi%
\noindent\ignorespaces}%

\newcommand{\refbreak}{\par%
\ifdim\lastskip<\refskipamount\removelastskip%
\penalty-100%
\refskip\fi%
\noindent\ignorespaces}%

\newcounter{listcounter}
\newcounter{deflistcounter}
\newcounter{equivcounter}

\newskip{\itemsepamount}
\newskip{\topsepamount}
\newenvironment{assertionlist}{%
  \begin{list}
    {\upshape (\arabic{listcounter})}
    {\setlength{\leftmargin}{18pt}
     \setlength{\rightmargin}{0pt}
     \setlength{\itemindent}{0pt}
     \setlength{\labelsep}{5pt}
     \setlength{\labelwidth}{13pt}
     \setlength{\listparindent}{\parindent}
     \setlength{\parsep}{0pt}
     \setlength{\itemsep}{\itemsepamount}
     \setlength{\topsep}{\topsepamount}
     \usecounter{listcounter}}}
  {\end{list}}

\newenvironment{definitionlist}{%
  \begin{list}
    {\upshape (\alph{deflistcounter})}
    {\setlength{\leftmargin}{18pt}
     \setlength{\rightmargin}{0pt}
     \setlength{\itemindent}{0pt}
     \setlength{\labelsep}{5pt}
     \setlength{\labelwidth}{13pt}
     \setlength{\listparindent}{\parindent}
     \setlength{\parsep}{0pt}
     \setlength{\itemsep}{\itemsepamount}
     \setlength{\topsep}{\topsepamount}
     \usecounter{deflistcounter}}}
  {\end{list}}

\newenvironment{Alist}{%
  \begin{list}
    {\upshape (\Alph{deflistcounter})}
    {\setlength{\leftmargin}{18pt}
     \setlength{\rightmargin}{0pt}
     \setlength{\itemindent}{0pt}
     \setlength{\labelsep}{5pt}
     \setlength{\labelwidth}{13pt}
     \setlength{\listparindent}{\parindent}
     \setlength{\parsep}{0pt}
     \setlength{\itemsep}{\itemsepamount}
     \setlength{\topsep}{\topsepamount}
     \usecounter{deflistcounter}}}
  {\end{list}}

\newenvironment{equivlist}{%
  \begin{list}
    {\upshape (\roman{equivcounter})}
    {\setlength{\leftmargin}{18pt}
     \setlength{\rightmargin}{0pt}
     \setlength{\itemindent}{0pt}
     \setlength{\labelsep}{5pt}
     \setlength{\labelwidth}{13pt}
     \setlength{\listparindent}{\parindent}
     \setlength{\parsep}{0pt}
     \setlength{\itemsep}{\itemsepamount}
     \setlength{\topsep}{\topsepamount}
     \usecounter{equivcounter}}}
  {\end{list}}

\newenvironment{bulletlist}{%
  \begin{list}
    {\upshape \textbullet}
    {\setlength{\leftmargin}{18pt}
     \setlength{\rightmargin}{0pt}
     \setlength{\itemindent}{0pt}
     \setlength{\labelsep}{6pt}
     \setlength{\labelwidth}{12pt}
     \setlength{\listparindent}{\parindent}
     \setlength{\parsep}{0pt}
     \setlength{\itemsep}{\itemsepamount}
     \setlength{\topsep}{\topsepamount}}}
  {\end{list}}

% ...

\newcommand{\Acal}{{\mathcal A}}
\newcommand{\Bcal}{{\mathcal B}}

\newcommand{\Fcal}{{\mathcal F}}
\newcommand{\Gcal}{{\mathcal G}}

\newcommand{\Lcal}{{\mathcal L}}

\newcommand{\Ocal}{{\mathcal O}}

\newcommand{\Ucal}{{\mathcal U}}
\newcommand{\Vcal}{{\mathcal V}}

\newcommand{\Xcal}{{\mathcal X}}

\newcommand{\Zcal}{{\mathcal Z}}

\newcommand{\Sfr}{{\mathfrak S}}

\renewcommand{\AA}{\mathbb{A}}

\newcommand{\CC}{\mathbb{C}}

\newcommand{\EE}{\mathbb{E}}
\newcommand{\FF}{\mathbb{F}}
\newcommand{\GG}{\mathbb{G}}

\newcommand{\LL}{\mathbb{L}}

\newcommand{\NN}{\mathbb{N}}

\newcommand{\QQ}{\mathbb{Q}}
\newcommand{\RR}{\mathbb{R}}

\newcommand{\ZZ}{\mathbb{Z}}

\newcommand{\Lscr}{{\mathscr L}}

\newcommand{\Pscr}{{\mathscr P}}

\newcommand{\Sscr}{{\mathscr S}}

\DeclareMathOperator{\ad}{ad}

\DeclareMathOperator{\Gal}{Gal}

\DeclareMathOperator{\Hom}{Hom}

\DeclareMathOperator{\Span}{Span}

\DeclareMathOperator{\orb}{orb}

\DeclareMathOperator{\pr}{pr}

\DeclareMathOperator{\Rep}{Rep}

\DeclareMathOperator{\Sbt}{Sbt}
\DeclareMathOperator{\Sh}{Sh}
\DeclareMathOperator{\spec}{Spec}

\DeclareMathOperator{\Sch}{Sbt}

\DeclareMathOperator{\zip}{zip}
\DeclareMathOperator{\GS}{GS}

\DeclareMathOperator{\GL}{GL}

\DeclareMathOperator{\GSp}{GSp}

\DeclareMathOperator{\Sp}{Sp}
\DeclareMathOperator{\U}{U}
\DeclareMathOperator{\GU}{GU}

\newcommand{\shgx}{\Sh(\mathbf G, \mathbf X)}

\newcommand{\gx}{(\mathbf G, \mathbf X)}

\newcommand{\gofaf}{\mathbf G(\mathbf A_f)}

\newcommand{\id}{{\rm Id}}

\newcommand{\loccit}{{\em loc.\ cit. }}
\newcommand{\loccitn}{{\em loc.\ cit.}}

\newcommand{\diag}{{\rm diag}}

\newcommand{\reg}{{\rm reg}}

\renewcommand{\div}{{\rm div}}

\DeclareMathOperator{\hw}{hw}

\DeclareMathOperator{\flag}{flag}

\DeclareMathOperator{\Norm}{Norm}
\DeclareMathOperator{\Ind}{Ind}

\DeclareMathOperator{\Hasse}{Hasse}
\DeclareMathOperator{\Res}{Res}
\DeclareMathOperator{\Flag}{Flag}

\DeclareMathOperator{\high}{high}
\DeclareMathOperator{\low}{low}
\DeclareMathOperator{\lw}{lw}

\DeclareMathOperator{\Ha}{Ha}

\newcommand{\relmiddle}[1]{\mathrel{}\middle#1\mathrel{}}

\begin{document}

\author{Wushi Goldring and Jean-Stefan Koskivirta}

\title{Divisibility of mod $p$ automorphic forms and the cone conjecture for certain Shimura varieties of Hodge-type} 

\date{}

\maketitle

\begin{abstract}
For several Hodge-type Shimura varieties of good reduction in characteristic $p$, we show that the cone of weights of automorphic forms is encoded by the stack of $G$-zips of Pink--Wedhorn--Ziegler. This establishes several instances of a general conjecture formulated in previous papers by the authors. Furthermore, we prove in these cases that any mod $p$ automorphic form whose weight lies in a specific region of the weight space is divisible by a partial Hasse invariant. This generalizes to other Shimura varieties previous results of Diamond--Kassaei on Hilbert modular forms.
\end{abstract}

\section*{Introduction}
In a series of papers \cite{Goldring-Koskivirta-global-sections-compositio, Goldring-Koskivirta-Strata-Hasse,   Goldring-griffiths-bundle-group-theoretic}, the authors suggested that several geometric invariants attached to Shimura varieties should be expressible in terms of group theoretical objects. This paper pursues this philosophy and presents further evidence for a conjecture proposed in \cite{Goldring-Koskivirta-global-sections-compositio} regarding weights of mod $p$ automorphic forms. In particular, we generalize the results of Diamond--Kassaei in \cite{Diamond-Kassaei} (extended in \cite{arxiv-Diamond-Kassaei-cone-minimal}) regarding Hilbert--Blumenthal Shimura varieties, to more general Hodge-type Shimura varieties. We briefly review their results. Let $\mathbf{E}/\QQ$ be a totally real extension of degree $n\geq 2$, and let $X$ be the associated Hilbert--Blumenthal Shimura variety (for a fixed level) over $\CC$. For a tuple $\mathbf{k}=(k_1, \dots , k_n)\in \ZZ^n$, there is a automorphic line bundle $\omega^{\mathbf{k}}$. We call elements of $H^0(X,\omega^{\mathbf{k}})$ Hilbert automorphic forms of weight $\mathbf{k}$. We fix a place of good reduction $p$, and consider the geometric special fiber $X_{\overline{\FF}_p}$ of the integral model of $X$ constructed by Kottwitz. Diamond--Kassaei define the minimal cone $C_{\min}\subset \ZZ^n$. Their definition of $C_{\min}$ is derived from considerations regarding the minimal weights in Serre's Conjecture. They show the following:
\begin{assertionlist}
\item The weight of any nonzero Hilbert automorphic form lies in the cone $C_{\Hasse}\subset \ZZ^n$ spanned (over $\QQ_{\geq 0}$) by the weights of all partial Hasse invariants.
\item Any Hilbert automorphic form $f$ over $\overline{\FF}_p$ whose weight lies in the complement of $C_{\min}$ is divisible by (a specific) partial Hasse invariant.
\end{assertionlist}
The notion of partial Hasse invariant was introduced by Andreatta--Goren in \cite{Goren-partial-hasse}, \cite{Andreatta-Goren-book}. They are characterized by the fact that their vanishing locus is the closure of a codimension one Ekedahl--Oort stratum of $X_{\overline{\FF}_p}$. To generalize this result to other Hodge-type Shimura varieties, we need to consider vector-valued automorphic forms. Let $(\mathbf{G},\mathbf{X})$ be a Hodge-type Shimura datum (\cite{Deligne-Shimura-varieties}). In particular, $\mathbf{G}$ is a connected reductive group over $\QQ$. There is an attached Shimura variety $\shgx_{K}$ for any compact open $K\subset \mathbf{G}(\AA_f)$. For sufficiently small $K$, it is a smooth, quasi-projective scheme over a number field $\mathbf{E}$ (called the reflex field). Fix a Borel pair $(\mathbf{B},\mathbf{T})$ such that $\mathbf{B}\subset \mathbf{P}$, where $\mathbf{P}$ is the parabolic subgroup stabilizing the Hodge filtration. Let $\mathbf{L}$ be the unique Levi subgroup of $\mathbf{P}$ containing $\mathbf{T}$. For each $\lambda\in X^*(\mathbf{T})$, there is an automorphic vector bundle $\Vcal_I(\lambda)$ on $\shgx_K$, modeled on the representation $\Ind_{\mathbf{B}}^{\mathbf{P}}(\lambda)$, whose sections are called automorphic forms of weight $\lambda$ and level $K$. Assume that $\shgx_{K}$ has good reduction at a prime $p$, i.e that $K$ can be written $K=K_p K^p$, where $K_p\subset \mathbf{G}(\QQ_p)$ is hyperspecial and $K^p\subset \mathbf{G}(\AA_f^p)$ is compact open. In this case, Kisin (\cite{Kisin-Hodge-Type-Shimura}) and Vasiu (\cite{Vasiu-Preabelian-integral-canonical-models}) have constructed a canonical model $\Sscr_K$ over $\Ocal_{\mathbf{E}_v}$ for any place $v|p$ in $\mathbf{E}$. Write $S_K\colonequals \Sscr_K\otimes_{\Ocal_{\mathbf{E}_v}}\overline{\FF}_p$. Let also $G$ be the reductive $\FF_p$-group obtained as the special fiber of a reductive $\ZZ_p$-model of $\mathbf{G}_{\QQ_p}$ (it exists since $K_p$ is hyperspecial). For any $\lambda\in X^*(\mathbf{T})$, the automorphic vector bundle $\Vcal_I(\lambda)$ extends to $\Sscr_K$.

Aiming at generalizing Diamond--Kassaei's result (1) to other Hodge-type Shimura varieties, we ask the following: Given a field $F$ which is an $\Ocal_{\mathbf{E}_v}$-algebra, for which $\lambda\in X^*(T)$ is the space $H^0(\Sscr_K\otimes_{\Ocal_{\mathbf{E}_v}}\otimes F,\Vcal_I(\lambda))\neq 0$ nonzero ? In other words, we wish to understand the set
\begin{equation}\label{intro-CK}
C_K(F) = \{\lambda\in X^*(\mathbf{T}) \ \mid \ H^0(\Sscr_K\otimes_{\Ocal_{\mathbf{E}_v}}\otimes F,\Vcal_I(\lambda))\neq 0\}.
\end{equation}
This set is a subcone of $X^*(\mathbf{T})$ (i.e an additive submonoid containing $0$). It depends on the level $K$, but its saturated cone $\langle C_K(F)\rangle$ does not. Here, the saturated cone $\langle C\rangle$ of a cone $C\subset X^*(\mathbf{T})$ is the set of $\lambda\in X^*(\mathbf{T})$ such that some positive multiple of $\lambda$ lies in $C$. For $F=\CC$, the set $\langle C_K(\CC)\rangle$ is conjectured to coincide with the Griffiths-Schmid cone
\begin{equation}
C_{\GS}=\left\{ \lambda\in X^{*}(\mathbf{T}) \ \relmiddle| \
\parbox{6cm}{
$\langle \lambda, \alpha^\vee \rangle \geq 0 \ \textrm{ for }\alpha\in \Phi_{\mathbf{L},+}, \\
\langle \lambda, \alpha^\vee \rangle \leq 0 \ \textrm{ for }\alpha\in \Phi_+ \setminus \Phi_{\mathbf{L},+}$}
\right\}.
\end{equation}
Here, $\Phi_+$ is the set of positive $\mathbf{T}$-roots with resepct to the opposite Borel subgroup of $\mathbf{B}$, and $\Phi_{\mathbf{L},+}$ is the set of positive roots in $\mathbf{L}$. The inclusion $\langle C_K(\CC)\rangle \subset C_{\GS}$ is proved in \cite{Goldring-Koskivirta-GS-cone}. In this paper, we are interested in the case $F=\overline{\FF}_p$. Zhang has shown that there is a smooth map $\zeta\colon S_K\to \GZip^\mu$, where $\GZip^\mu$ is the stack of $G$-zips of Moonen--Wedhorn and Pink--Wedhorn--Ziegler (\cite{Moonen-Wedhorn-Discrete-Invariants, Pink-Wedhorn-Ziegler-zip-data,PinkWedhornZiegler-F-Zips-additional-structure}). The fibers of $\zeta$ are called the Ekedahl--Ort strata of $S_K$. Here $\mu\colon \GG_{\textrm{m},\overline{\FF}_p}\to G_{\overline{\FF}_p}$ is a cocharacter derived from the Shimura datum. The map $\zeta$ is also surjective by \cite[Corollary 3.5.3(1)]{Shen-Yu-Zhang-EKOR}. 

The vector bundle $\Vcal_I(\lambda)$ also exists on $\GZip^\mu$, and its pullback via $\zeta$ coincides with the automorphic vector bundle $\Vcal_I(\lambda)$ on the special fiber $S_K$. Hence, similarly to $C_K(\overline{\FF}_p)$, it is natural to define a cone $C_{\zip}$ as the set of $\lambda\in X^*(T)$ such that $H^0(\GZip^\mu,\Vcal_I(\lambda))\neq 0$. Since $\zeta\colon S_K\to \GZip^\mu$ is surjective, one has by pullback an obvious inclusion $C_{\zip}\subset C_K(\overline{\FF}_p)$. We conjectured in \cite{Goldring-Koskivirta-global-sections-compositio}:
\begin{conj1}\label{conj-intro}
One has an equality $\langle C_K(\overline{\FF}_p)\rangle = \langle C_{\zip}\rangle$.
\end{conj1}
We may also consider the pair $(S^{\Sigma}_K,\zeta^{\Sigma})$ where $\zeta^{\Sigma}$ is the extension of $\zeta$ to the toroidal compactification $S^{\Sigma}_{K}$ (see \S\ref{sec-tor}). In the setting of Hilbert--Blumenthal Shimura varieties, Diamond--Kassaei's result (1) says that the set $\langle C_K(\overline{\FF}_p)\rangle$ is generated by the weights of partial Hasse invariants. For a general connected reductive $\FF_p$-group $G$, partial Hasse invariants were defined in \cite{Goldring-Koskivirta-Strata-Hasse, Goldring-Koskivirta-global-sections-compositio}, and studied in detail in \cite{Imai-Koskivirta-partial-Hasse}. We denote by $C_{\Hasse}$ the subcone of $X^*(T)$ generated by their weights (see Definition \ref{definition-CHasse} for a precise definition). Hence, we could ask if $\langle C_{K}(\overline{\FF}_p)\rangle = \langle C_{\Hasse} \rangle$ holds more generally. However, we cannot expect such an equality in general. Indeed, denote by $(B,T)$ be a Borel pair of $G$ such that $\mu$ factors through $T$, and by $L\subset G_{\overline{\FF}_p}$ the Levi subgroup centralizing the cocharacter $\mu\colon \GG_{\textrm{m},\overline{\FF}_p}\to G_{\overline{\FF}_p}$. Let $\Delta_L$ denote the set of simple roots of $L$ (with respect to the opposite Borel of $B$). Let $W_L\colonequals W(L,T)$ be the Weyl group of $L$ and write $w_{0,L}\in W_L$ for its longest element. We showed in joint work with Imai (\cite[Theorem 4.3.1]{Goldring-Imai-Koskivirta-weights}) that the equality $\langle C_{\zip} \rangle = \langle C_{\Hasse} \rangle$ holds if and only if $L$ is defined over $\FF_p$ and the Frobenius acts on $\Delta_L$ by $-w_{0,L}$. Since we always have $C_{\Hasse}\subset C_{\zip}\subset C_K(\overline{\FF}_p)$, the cone $\langle C_K(\overline{\FF}_p)\rangle$ cannot coincide with $\langle C_{\Hasse} \rangle$ unless this condition is satisfied.

\bigskip

We now explain the results of the present paper. First of all, we prove Conjecture 1 for several Shimura varieties. More precisely, we consider any $\overline{\FF}_p$-scheme $S$ endowed with a map $\zeta\colon S\to \GZip^\mu$ satisfying certain regularity assumptions (see \S\ref{sec-cone-conj} for details). For example, the toroidal compactification $S^{\Sigma}_K$ endowed with the extension of Zhang's map $\zeta^\Sigma\colon S^{\Sigma}_K\to \GZip^\mu$ satisfies these assumptions. Note also that by the Koecher principle, the global sections of $\Vcal_I(\lambda)$ over $S_K$ and $S^{\Sigma}_K$ can be identified. For any such pair $(S,\zeta)$, we may define a subset $C_S\subset X^*(T)$ as the set of $\lambda\in X^*(T)$ such that $H^0(S,\Vcal_I(\lambda))\neq 0$, similarly to $C_K(\overline{\FF}_p)$. For two reductive $\FF_p$-groups $G$ and $G'$ endowed with cocharacters $\mu\colon \GG_{\textrm{m},\overline{\FF}_p}\to G_{\overline{\FF}_p}$ and $\mu'\colon \GG_{\textrm{m},\overline{\FF}_p}\to G'_{\overline{\FF}_p}$, let us say that $(G,\mu)$ and $(G',\mu')$ are equivalent if $(G^{\rm ad},\mu^{\rm ad})=(G'^{\rm ad},\mu'^{\rm ad})$.

\begin{thm1}[Theorems \ref{thm-conjSp6}, \ref{thm-conjGL31}, \ref{thmGL22-conj}, \ref{thm-U31-conj}] Suppose $(G,\mu)$ is one of the following (up to equivalence) :
\begin{definitionlist}
    \item $G=\GSp(6)_{\FF_p}$, $\mu$ is minuscule and $p\geq 5$,
    \item $G=\GL_{4,\FF_p}$ and $\mu$ is minucule,
    \item $G=\GU(3)_{\FF_p}$ and $\mu$ is minuscule,
    \item $G=\GU(4)_{\FF_p}$, $\mu$ is minuscule, and the type of $\mu$ is not $(2,2)$.
    \end{definitionlist}
Then we have $\langle C_{\zip}\rangle = \langle C_S \rangle$.
\end{thm1}
In the case when $G=\GSp(4)$ and $G=\GL_{3,\FF_p}$, the conjecture was already proved in \cite[Theorem 5.1.1]{Goldring-Koskivirta-global-sections-compositio}. Theorem 1(a) shows Conjecture 1 for the special fiber of Siegel-type modular varieties $\Acal_{3,\overline{\FF}_p}$. Theorem 1 (b)--(d) show Conjecture 1 for the special fiber of unitary Shimura varieties attached to $\GU(r,s)$ with $r+s\leq 4$ (at a split or inert prime), with one case missing: that of the reduction at an inert prime of a Shimura variety attached to $\GU(2,2)$ (for which we could not verify the conjecture). See \S\ref{sec-unit-Shim} for more details on unitary Shimura varieties. In particular, we obtain a vanishing result for the cohomology $H^0(S_K,\Vcal_I(\lambda))$: we deduce that this space is zero for all $\lambda$ in the complement of $\langle C_{\zip} \rangle$. The cone $\langle C_{\zip} \rangle$ is given explicitly in each case in \S \ref{sec-symp}, \S\ref{sec-GLn} and \S\ref{sec-unit}.

\bigskip

Our second main result concerns the generalization of Diamond-Kassaei's divisibility result (2) by partial Hasse invariants. In each case appearing in Theorem 1, and each possible type of cocharacter $\mu$, there is a similar divisibility result. To explain it, we first need to explain the notion of divisibility. Denote by $\GF^\mu$ the stack of zip flags, defined in \cite[Part 1, \S2]{Goldring-Koskivirta-Strata-Hasse}. It classifies $G$-zips endowed with a compatible $B$-torsor. Similarly, we can define the flag space $\Flag(S)$ of any $(S,\zeta)$ as above. When $S$ is a Siegel-type Shimura variety, the flag space $\Flag(S)$ was defined and studied by Ekedahl--Van der Geer in \cite{Ekedahl-Geer-EO}. In this case, it classifies principally polarized abelian varieties $(A,\chi)$ endowed with a full symplectic flag in $H^1_{\rm dR}(A)$. In general, there is a natural projection $\pi\colon \Flag(S)\to S$ and for each $\lambda\in X^*(T)$ a natural line bundle $\Vcal_{\flag}(\lambda)$ such that $\pi_{*}(\Vcal_{\flag}(\lambda))=\Vcal_I(\lambda)$. Hence, we may identify any $f\in H^0(S,\Vcal_I(\lambda))$ with a section $f_{\flag}\in H^0(\Flag(S),\Vcal_{\flag}(\lambda))$. Then, we say that a vector-valued section $f$ is divisible by another section $g$ if $f_{\flag}$ is divisible by $g_{\flag}$ (as sections of line bundles). Furthermore, $\GF^\mu$ and $\Flag(S)$ are naturally stratified. The codimension $1$ strata in $\GF^\mu$ are of the form $(\Fcal_{w_0 s_\alpha})_{\alpha\in \Delta}$, where $s_\alpha$ is the reflection attached to $\alpha$, $\Delta$ is the set of simple roots and $w_0\in W$ is the longest element of $W$. For each $\alpha\in \Delta$, there is a partial Hasse invariant $\Ha_\alpha\in H^0(\GF^\mu,\Vcal_{I}(\lambda_\alpha))$ for some $\lambda_\alpha\in X^*(T)$, whose vanishing locus is precisely $\overline{\Fcal}_{w_0s_\alpha}$. Such sections were studied in detail in \cite{Imai-Koskivirta-partial-Hasse} by Imai and the second-named author and were called "(flag) partial Hasse invariants". Generalizing (2), we show that for certain simple roots $\alpha\in \Delta$, an automorphic form whose weight is "close" to $\lambda_\alpha$ is automatically divisible by $\Ha_\alpha$. There are two restrictions on the roots $\alpha$ which admit such divisibility results:
\begin{Alist}
    \item The weight $\lambda_\alpha$ generates an extremal ray of the cone $\langle C_{\zip} \rangle$.
    \item $\lambda_{\alpha}$ lies in the complement of $C_{\GS}$.
\end{Alist}
For any pair $(G,\mu)$ that we have checked, there always exists roots $\alpha\in \Delta$ satisfiying the above conditions, although most roots do not satisfy them in general. We explain in \S\ref{sec-sympn3} the necessity for condition (A). We impose Condition (B) for more empirical reasons. Indeed, in the case of Hilbert--Blumenthal Shimura varieties, the result of Diamond--Kassaei gives a non-trivial divisibility result only when $\lambda_\alpha$ lies outside of $C_{\GS}=(\ZZ_{\leq 0})^n$ (note that we have a sign convention different from \cite{arxiv-Diamond-Kassaei-cone-minimal}). Similarly, in all examples considered in this paper, only roots $\alpha$ satisfying (A) and (B) admit divisibility results.

\begin{thm2}[Theorems \ref{thm-divSp6}, \ref{thm-divGL31}, \ref{thm-divGU21inert}, \ref{thm-U31-inert-div}]
We make the same assumptions as in Theorem 1. Let $\alpha\in \Delta$ be a root satisfying (A) and (B). Then there exists an (explicit) subcone $V_\alpha\subset X^*(T)$ which is a neighborhood of $\lambda_\alpha$ such that any section $f\in H^0(S,\Vcal_I(\lambda))$ of weight $\lambda\in V_\alpha$ is divisible by $\Ha_\alpha$.
\end{thm2}

The result also holds in the cases $\GSp(4)_{\FF_p}$ and $\GL_{3,\FF_p}$ considered in \cite{Goldring-Koskivirta-global-sections-compositio}, as explained in Theorems \ref{thm-Sp4-div} and  \ref{thm-GL21-div}. By "neighborhood of $\lambda_\alpha$", we mean specifically that the $\RR_{\geq0}$-cone $V_{\alpha,\RR_{\geq 0}}$ of linear combinations with nonnegative real coefficients is a neighborhood of $\lambda_\alpha$ in $X^*(T)\otimes \RR$. As an example, we give a sample result in the case $G=\Sp(6)$ and $\mu$ of Siegel-type (the case $G=\GSp(6)$ is completely similar). In this case, we identify $X^*(T)=\ZZ^3$ and write the simple roots as $\alpha_1=e_1-e_2$, $\alpha_2=e_2-e_3$, $\beta=2e_3$ where $(e_1,e_2,e_3)$ is the canonical basis of $\ZZ^3$. Let $(S,\zeta)$ be any pair satisfying Assumption \ref{assume} of \S\ref{sec-cone-conj}. The only simple root satisfying (A) and (B) is $\alpha_1$. The corresponding partial Hasse invariant $\Ha_{\alpha_1}$ has weight $\lambda_{\alpha_1}=(1,0,-p)$. The neighborhood $V_{\alpha_1}$ is given explicitly in the following theorem.

\begin{thm3}[Theorem \ref{thm-divSp6}]\label{thm-divSp6-intro}
Assume $p\geq 5$. Let $f\in H^0(\Flag(S),\Vcal_{\flag}(\lambda))$ and suppose that $\lambda=(a_1,a_2,a_3)\in \ZZ^3$ satisfies $p^2a_1+pa_2+a_3>0$. Then $f$ is divisible by the partial Hasse invariant $\Ha_{\alpha_1}$.
\end{thm3}

We do not know if this result holds for $p=2,3$. We illustrate such divisibility results for each pair $(G,\mu)$ in Figures 1, 2, 4, 5, 6.

\vspace{0.8cm}

\noindent
{\bf Acknowledgements.}
This work was supported by JSPS KAKENHI Grant Number 21K13765. W.G. thanks the Knut \& Alice Wallenberg Foundation for its support under grants KAW 2018.0356 and Wallenberg Academy Fellow KAW 2019.0256, and the Swedish Research Council for its support under grant \"AR-NT-2020-04924.

\section{Automorphic forms on $\GZip^\mu$}

In this section, we recall results from \cite{Koskivirta-automforms-GZip, Imai-Koskivirta-vector-bundles, Goldring-Imai-Koskivirta-weights} regarding automorphic forms on the stack of $G$-zips. This is a purely group-theoretical setting, but it gives intuition about what one can expect for usual automorphic forms in characteristic $p$.

\subsection{Notation}\label{subsec-notation}

Throughout the paper, $p$ is a prime number, $q$ is a power of $p$ and $\FF_q$ is the finite field with $q$ elements. We write $k=\overline{\FF}_q$ for an algebraic closure of $\FF_q$. The notation $G$ will always denote a connected reductive group over $\FF_q$. For a $k$-scheme $X$, we denote by $X^{(q)}$ its $q$-th power Frobenius twist and by $\varphi \colon X\to X^{(q)}$ its relative Frobenius morphism. Write $\sigma \in \Gal(k/\FF_q)$ for the $q$-power Frobenius. We will always write $(B,T)$ for a Borel pair of $G$ defined over $\FF_q$. We do not assume that $T$ is split over $\FF_q$. Let $B^+$ be the Borel subgroup of $G$  opposite to $B$ with respect to $T$ (i.e. the unique Borel subgroup $B^+$ of $G$ such that $B^+\cap B=T$). We use the following notations:
\begin{bulletlist}
\item As usual, $X^*(T)$ (resp. $X_*(T)$) denotes the group of characters (resp. cocharacters) of $T$. The group $\Gal(k/\FF_q)$ acts naturally on these groups. Let $W=W(G_k,T)$ be the Weyl group of $G_k$. Similarly, $\Gal(k/\FF_q)$ acts on $W$. Furthermore, the actions of $\Gal(k/\FF_q)$ and $W$ on $X^*(T)$ and $X_*(T)$ are compatible in a natural sense.
\item $\Phi\subset X^*(T)$: the set of $T$-roots of $G$.
\item $\Phi_+\subset \Phi$: the system of positive roots with respect to $B^+$ (i.e. $\alpha \in \Phi_+$ when the $\alpha$-root group $U_{\alpha}$ is contained in $B^+$). This convention may differ from other authors. We use it to match the conventions of previous publications \cite{Goldring-Koskivirta-Strata-Hasse}, \cite{Koskivirta-automforms-GZip}.
\item $\Delta\subset \Phi_+$: the set of simple roots. 
\item For $\alpha \in \Phi$, let $s_\alpha \in W$ be the corresponding reflection. The system $(W,\{s_\alpha \mid \alpha \in \Delta\})$ is a Coxeter system. 
We write $\ell  \colon W\to \NN$ for the length function, and $\leq$ for the Bruhat order on $W$. Let $w_0$ denote the longest element of $W$.
\item For a subset $K\subset \Delta$, let $W_K$ denote the subgroup of $W$ generated by $\{s_\alpha \mid \alpha \in K\}$. Write $w_{0,K}$ for the longest element in $W_K$.
\item Let ${}^KW$ (resp. $W^K$) denote the subset of elements $w\in W$ which have minimal length in the coset $W_K w$ (resp. $wW_K$). Then ${}^K W$ (resp. $W^K$) is a set of representatives of $W_K\backslash W$ (resp. $W/W_K$). The map $g\mapsto g^{-1}$ induces a bijection ${}^K W\to W^K$. The longest element in the set ${}^K W$ (resp. $W^K$) is $w_{0,K} w_0$ (resp. $w_0 w_{0,K}$).
\item $X_{+}^*(T)$ denotes the set of dominant characters, i.e. characters $\lambda\in X^*(T)$ such that $\langle \lambda,\alpha^\vee \rangle \geq 0$ for all $\alpha \in \Delta$.
\item Let $P\subset G_k$ be a parabolic subgroup containing $B$ and let $L\subset P$ be the unique Levi subgroup of $P$ containing $T$. Then we define a subset $I_P\subset \Delta$ as the unique subset such that $W(L,T)=W_{I_P}$. It is the set of simple roots of $L$ with respect to $B\cap L$. For an arbitrary parabolic subgroup $P\subset G_k$ containing $T$, we define $I_P\subset \Delta$ as $I_P \colonequals I_{P'}$ where $P'$ is the unique conjugate of $P$ containing $B$.
\item For a parabolic $P\subset G_k$, write $\Delta^P \colonequals \Delta \setminus I_P$.
\item For a subset $I\subset \Delta$, let $X_{+,I}^*(T)$ denote the set of characters $\lambda\in X^*(T)$ such that $\langle \lambda,\alpha^\vee \rangle \geq 0$ for all $\alpha \in I$. We call them $I$-dominant characters. When $P=BL$ and $I=I_P$, we also call such characters $L$-dominant.
\end{bulletlist}

\subsection{$G$-zips and $G$-zip flags}

\subsubsection{Definitions}\label{sec-zip-def}

A zip datum is a tuple $(G,P,Q,L,M,\varphi)$, where $G$ is a connected reductive group over $\FF_q$, $P,Q$ are parabolic subgroups of $G_k$ ($k=\overline{\FF}_q$), with respective Levi subgroups $L\subset P$ and $M\subset Q$ satisfying $M=L^{(q)}$. Finally, $\varphi\colon L\to M$ is  the $q$th power Frobenius homomorphism. For applications to Shimura varieties, we always take $q=p$. If $G$ is a connected, reductive group over $\FF_q$ and $\mu\colon \GG_{\textrm{m},k}\to G_k$ is a cocharacter, we call $(G,\mu)$ a cocharacter datum over $\FF_q$. We can attach to $(G,\mu)$ a zip datum $\Zcal_\mu$ as explained in \cite[\S 2.2.2]{Imai-Koskivirta-vector-bundles}. We recall the construction. First, $\mu$ defines a pair of opposite parabolics $P_{\pm}(\mu)$, where  $P_+(\mu)(k)$ (resp. $P_-(\mu)(k)$) consists of the elements $g\in G(k)$ such that the map 
\begin{equation}\label{eq-Pmu}
\GG_{\mathrm{m},k} \to G_{k}; \  t\mapsto\mu(t)g\mu(t)^{-1} \quad (\textrm{resp. } t\mapsto\mu(t)^{-1}g\mu(t))    
\end{equation}
extends to a regular map $\AA_{k}^1\to G_{k}$. The centralizer of $\mu$ is $L(\mu)=P_+(\mu)\cap P_-(\mu)$. Then, define $P\colonequals P_-(\mu)$, $Q\colonequals (P_+(\mu))^{(q)}$, $L\colonequals L(\mu)$ and $M\colonequals L^{(q)}$. Let $\varphi\colon L\to M$ be the Frobenius homomorphism. The tuple $\Zcal_\mu=(G,P,L,Q,M,\varphi)$ is called the zip datum attached to $(G,\mu)$. We will only consider zip data arising in this way. 

For a zip datum $\Zcal=(G,P,Q,L,M,\varphi)$, the zip group of $\Zcal$, denoted by $E$, is the subgroup of $P\times Q$ defined by:
\begin{equation}\label{zipgroup}
E \colonequals \{(x,y)\in P\times Q \mid  \varphi(\theta^P_L(x))=\theta^Q_M(y)\}.
\end{equation}
Here, $\theta^P_L\colon P\to L$ is the map that sends $x\in P$ to the unique element $\overline{x}\in L$ such that $x=\overline{x}u$ with $u \in R_{\textrm{u}}(P)$, where $R_{\textrm{u}}(P)$ is the unipotent radical of $P$ (and similarly for $\theta^Q_M$). Moonen--Wedhorn (\cite{Moonen-Wedhorn-Discrete-Invariants}) and Pink--Wedhorn--Ziegler (\cite{PinkWedhornZiegler-F-Zips-additional-structure,Pink-Wedhorn-Ziegler-zip-data}) defined the stack of $G$-zips of type $\mu$, denoted by $\GZip^\mu$. It is the quotient stack:
\begin{equation}
    \GZip^\mu=\left[E\backslash G_k \right].
\end{equation}
where $E$ acts on $G$ by $(x,y)\cdot g\colonequals xgy^{-1}$ for all $(x,y)\in E$ and all $g\in G$.

\subsubsection{Parametrization of $E$-orbits}
Let $(G,\mu)$ be a cocharacter datum over $\FF_q$ and $\Zcal=\Zcal_\mu=(G,P,L,Q,M,\varphi)$ the attached zip datum. We assume that there exists a Borel pair $(B,T)$ such that $B,T$ are defined over $\FF_q$ and $B\subset P$. We can always change $\mu$ to a conjugate to ensure that such a Borel pair exists (see \cite[\S 2.2.3]{Imai-Koskivirta-vector-bundles}). Set $I\colonequals I_P$ and $J\colonequals I_Q$ and define \begin{equation}\label{defz}
    z\colonequals \sigma(w_{0,I})w_0=w_0 w_{0,J}.
\end{equation}
We give the parametrization of $E$-orbits in $G_k$ following \cite{Pink-Wedhorn-Ziegler-zip-data}. For $w\in W$, fix a representative $\dot{w}\in N_G(T)$, such that $(w_1w_2)^\cdot = \dot{w}_1\dot{w}_2$ whenever $\ell(w_1 w_2)=\ell(w_1)+\ell(w_2)$ (this is possible by choosing a Chevalley system, \cite[ XXIII, \S6]{SGA3}). For $w\in W$, define $G_w$ as the $E$-orbit of $\dot{w}\dot{z}^{-1}$. If no confusion occurs, we write $w$ instead of $\dot{w}$. For $w,w'\in {}^I W$, write $w'\preccurlyeq w$ if there exists $w_1\in W_I$ such that $w'\leq w_1 w \sigma(w_1)^{-1}$. This defines a partial order on ${}^I W$ (\cite[Corollary 6.3]{Pink-Wedhorn-Ziegler-zip-data}).

\begin{theorem}[{\cite[Theorem 7.5, Theorem 11.2]{Pink-Wedhorn-Ziegler-zip-data}}] \label{thm-E-orb-param}
We have two bijections:
\begin{align} \label{orbparam}
{}^I W &\longrightarrow \{ \textrm{$E$-orbits in $G_k$} \}, \quad \ w\mapsto G_w \\  
\label{dualorbparam} W^J &\longrightarrow \{ \textrm{$E$-orbits in $G_k$}\}, \quad \ w\mapsto G_w. 
\end{align}
For $w\in {}^I W$, one has $\dim(G_w)= \ell(w)+\dim(P)$ and the Zariski closure of $G_w$ is 
\begin{equation}\label{equ-closure-rel}
\overline{G}_w=\bigsqcup_{w'\in {}^IW,\  w'\preccurlyeq w} G_{w'}.
\end{equation}
There is a similar formula for $W^J$. 
\end{theorem}
For $w\in {}^I W\cup W^J$, we write $\Xcal_w\colonequals[E\backslash G_w]$. It is a smooth, locally closed substack of $\GZip^\mu$.

\subsubsection{The flag space}\label{sec-zip-flag}
The stack of $G$-zip flags introduced in \cite[Part 1, \S2]{Goldring-Koskivirta-Strata-Hasse} is the quotient stack
\begin{equation}
    \GF^\mu=\left[E'\backslash G_k \right]
\end{equation}
where $E'\colonequals E\cap (B\times G)$. There is a natural projection $\pi\colon \GF^\mu\to \GZip^\mu$ with fibers isomorphic to $P/B$. Define a group $\widehat{E}_B$ as follows. It is the subgroup $\widehat{E}_B\subset B\times {}^zB$ of pairs $(x,y)\in B\times {}^zB$ such that $\varphi(\theta^B_T(x))=\theta^{{}^zB}_T(y)$. Concretely, $\widehat{E}_B$ consists of elements $(tu,\varphi(t)u')$ with $t\in T$ and $(u,u')\in R_{\textrm{u}}(B)\times R_{\textrm{u}}({}^zB)$. Note that $\widehat{E}_B$ is the zip group of the zip datum $\Zcal^B\colonequals (G,B,{}^zB,T,T,\varphi)$. Clearly, one has inclusions
\begin{equation}
    E'\subset \widehat{E}_B \subset B\times {}^zB.
\end{equation}
Hence, we obtain natural projection maps $\Psi$ and $\gamma$ as follows:
\begin{equation}\label{mapGFcompo}
\xymatrix@1{
 \GF^\mu \ar[r]^-{\Psi} &  \left[ \widehat{E}_B \backslash G_k \right] \ar[r]^-{\gamma} & \left[(B\times{}^zB) \backslash G_k \right] .
}
\end{equation}
We call $\left[(B\times{}^zB) \backslash G_k \right]$ the untwisted Schubert stack. It is sometimes convenient to twist by the element $z$ by composing with the isomorphism $\ad(z)\colon \left[(B\times{}^zB) \backslash G_k \right]\to \left[(B\times B) \backslash G_k \right]$ induced by $G_k\to G_k$, $x\mapsto xz$. The composition $\ad(z)\circ \gamma \circ \Psi$ gives a smooth surjective map
\begin{equation}\label{psimap}
    \psi\colon \GF^\mu \to \Sbt\colonequals [(B\times B)\backslash G_k].
\end{equation}
We call $\Sbt$ the (twisted) Schubert stack. The underlying topological space of $\Sbt$ is isomorphic to $W$, endowed with the topology induced from the Bruhat--Chevalley order. For $w\in W$, define
\begin{equation}\label{Sbtw}
    \Sbt_w\colonequals \left[ (B\times B)\backslash BwB \right].
\end{equation}
The Zariski closure $\overline{BwB}$ is normal (\cite[Theorem 3]{Ramanan-Ramanathan-projective-normality}) and coincides with the union of $Bw'B$ for $w'\leq w$. Define the flag strata of $\GF^\mu$ as the fibers of the map $\psi$. More precisely, let $w\in W$ and set $F_w=BwBz^{-1}$. Then $F_w$ is the $B\times {}^zB$-orbit of $wz^{-1}$. We define
\begin{equation}\label{def-Cw}
    \Fcal_w\colonequals \left[E' \backslash F_w \right].
\end{equation} This is a smooth, locally closed substack of $\GF^\mu$. The Zariski closure $\overline{\Fcal}_w$ of $\Fcal_w$ is the union of $\Fcal_{w'}$ for $w'\leq w$.

\subsubsection{Automorphic vector bundles on $\GZip^\mu$}\label{sec-automVB-zip}
As explained in \cite[\S 2.4]{Imai-Koskivirta-vector-bundles}, we can attach to any $P$-representation $(V,\rho)$ a vector bundle $\Vcal(\rho)$ on $\GZip^\mu$, by the usual "associated sheaf construction" of \cite[\S5.8]{jantzen-representations}. We will only consider $P$-representations which are trivial on $R_{\textrm{u}}(P)$. For $\lambda \in X^*(T)$, let $(V_I(\lambda),\rho_{I,\lambda})$ be the $P$-representation $\Ind_{B}^P(\lambda)$. It has highest weight $\lambda$ and is trivial on $R_{\textrm{u}}(P)$. Concretely, $V_I(\lambda)$ consists of all regular maps $f\colon P\to \AA^1$ satisfying 
\begin{equation}\label{eqVI}
 f(xb)=\lambda(b)^{-1}f(x)   
\end{equation}
for all $x\in P$ and all $b\in B$. Denote by $\Vcal_I(\lambda)$ the vector bundle attached to $V_I(\lambda)$, and call it an automorphic vector bundle. Its global sections are given by
\begin{equation}\label{H0GZip}
 H^0(\GZip^\mu,\Vcal_I(\lambda))=\left\{f \colon G_k \to V_I(\lambda) \Bigg| \begin{array}{l}
  f(xgy^{-1})=\rho_{I,\lambda}(x) f(g), \\
  \forall (x,y)\in E, \ \forall g\in G_k    
 \end{array}   \right\}.    
\end{equation}
Similarly, we can define a line bundle $\Vcal_{\flag}(\lambda)$ on $\GF^\mu$ such that $\pi_*(\Vcal_{\flag}(\lambda))=\Vcal_I(\lambda)$, as in \cite[\S 2.5.2]{Goldring-Imai-Koskivirta-weights}. Hence, we have an identification
\begin{equation}\label{H0-ident}
    H^0(\GZip^\mu,\Vcal_I(\lambda)) =  H^0(\GF^\mu,\Vcal_{\flag}(\lambda)).
\end{equation}
Concretely, the right-hand side is the following space:
\begin{equation}\label{H0GF}
 H^0(\GF^\mu,\Vcal_{\flag}(\lambda))=\left\{f \colon G_k \to \AA^1 \Bigg| \begin{array}{l}
  f(xgy^{-1})=\lambda(x) f(g), \\
  \forall (x,y)\in E', \ \forall g\in G_k    
 \end{array}   \right\}.
\end{equation}
If $f:G_k\to V_I(\lambda)$ is as in \eqref{H0GZip}, then $g\mapsto f(g)[1]$ lies in the space \eqref{H0GF}, and this defines an isomorphism between these two spaces.

\subsection{Shimura varieties and $G$-zips} \label{subsec-Shimura}
\subsubsection{The map $\zeta$}
Let $\gx$ be a Shimura datum of Hodge-type \cite[2.1.1]{Deligne-Shimura-varieties}. In particular, $\mathbf{G}$ is a connected, reductive group over $\QQ$. Furthermore, $\mathbf{X}$ provides a well-defined $\mathbf{G}(\overline{\QQ})$-conjugacy class of cocharacters $\{\mu\}$ of $\mathbf{G}_{\overline{\QQ}}$. Write $\mathbf{E}=\mathbf{E}(\mathbf{G},\mathbf{X})$ for the reflex field of $\gx$ (i.e. the field of definition of $\{\mu\}$) and $\Ocal_\mathbf{E}$ for its ring of integers. 
Given an open compact subgroup $K \subset \gofaf$, write $\shgx_{K}$ for Deligne's canonical model at level $K$ over $\mathbf{E}$ (see \cite{Deligne-Shimura-varieties}). For $K\subset \mathbf{G}(\AA_f)$ small enough, $\shgx_K$ is a smooth, quasi-projective scheme over $\mathbf{E}$. Assume that $\gx$ is of Hodge-type. Fix a prime number $p$ of good reduction. In particular, $\mathbf{G}_{\QQ_p}$ is unramified, so there exists a reductive $\ZZ_p$-model $\Gcal$, such that $G\colonequals \Gcal\otimes_{\ZZ_p}\mathbb{F}_p$ is connected. For any place $v$ above $p$ in $\mathbf{E}$, Kisin (\cite{Kisin-Hodge-Type-Shimura}) and Vasiu (\cite{Vasiu-Preabelian-integral-canonical-models})  constructed a smooth canonical model $\Sscr_K$ over $\Ocal_{\mathbf{E}_v}$-schemes. Write $S_K\colonequals \Sscr_K\otimes_{\Ocal_{\mathbf{E}_v}} \overline{\FF}_p$. 

For $\mu \in \{\mu\}$, let $\mathbf{P}=\mathbf{P}_-(\mu)$ be the parabolic of $\mathbf{G}_\CC$ defined as in \S\ref{sec-zip-def}. As explained in \cite[\S 2.5]{Imai-Koskivirta-vector-bundles}, we can find $\mu\in \{\mu\}$ which extends to a cocharacter of $\Gcal_{\Ocal_{\mathbf{E}_v}}$. Write again $\mu$ for its special fiber. Then $(G,\mu)$ is a cocharacter datum, and yields a zip datum $(G,P,L,Q,M,\varphi)$ (we always take $q=p$ in the context of Shimura varieties). Zhang (\cite[4.1]{Zhang-EO-Hodge}) constructed a smooth morphism
\begin{equation}\label{zeta-Shimura}
\zeta \colon S_K\to \GZip^\mu. 
\end{equation}
This map is also surjective by \cite[Corollary 3.5.3(1)]{Shen-Yu-Zhang-EKOR}.

\subsubsection{Automorphic vector bundles}
Let $\Pscr$ be the unique parabolic of $\Gcal_{\Ocal_{\mathbf{E}_v}}$ which extends $\mathbf{P}$. Then, we have a commutative diagram of functors
$$\xymatrix{
\Rep_{\overline{\ZZ}_p}(\mathscr{P}) \ar[r]^{\Vcal} \ar[d] & \VB(\Sscr_{K}) \ar[d] \\
\Rep_{\overline{\FF}_p}(P) \ar[r]^{\Vcal} & \VB(S_K).
}$$
The vector bundles of this form on $\Sscr_K$ and $S_K$ are called \emph{automorphic vector bundles} in \cite[III. Remark 2.3]{Milne-ann-arbor}. In particular, let $(\mathbf{B},\mathbf{T})$ be a Borel pair such that $\mathbf{B}\subset \mathbf{P}$ and $\lambda\in X^*(\mathbf{T})$ an $\mathbf{L}$-dominant character. Let $\mathbf{V}_{\mathbf{L}}(\lambda)=H^0(\mathbf{P}/\mathbf{B},\Lcal_\lambda)$ denote the unique irreducible representation of $\mathbf{P}$ over $\overline{\QQ}_p$ of highest weight $\lambda$, where $\Lcal_\lambda$ is the line bundle attached to 
$\lambda$. It admits a natural model over $\overline{\ZZ}_p$, namely $\mathbf{V}_{\mathbf{L}}(\lambda)_{\overline{\ZZ}_p} \colonequals H^0(\mathscr{P}/\mathscr{B},\Lcal_\lambda)$, where $\mathscr{B}$ is the unique Borel of $\Gcal_{\Ocal_{\mathbf{E}_v}}$ extending $\mathbf{B}$. Its reduction modulo $p$ is the $P$-representation $V_I(\lambda)=H^0(P/B,\Lcal_\lambda)$ over $k=\overline{\FF}_p$ as in \S\ref{sec-automVB-zip}. We denote by $\Vcal_I(\lambda)$ the vector bundle on $\Sscr_K$ (resp. $S_K$) attached to $\mathbf{V}_{\mathbf{L}}(\lambda)_{\overline{\ZZ}_p}$ (resp. $V_I(\lambda)$).

\subsubsection{Toroidal compactification} \label{sec-tor}

By \cite[Theorem 1]{Madapusi-Hodge-Tor}, there is a sufficiently fine cone decomposition $\Sigma$ and a toroidal compactification $\Sscr_K^\Sigma$ of $\Sscr_K$ over $\Ocal_{\mathbf{E}_v}$. We fix such a toroidal compactification, and we denote by $S_K^\Sigma$ its special fiber. By \cite[Theorem 6.2.1]{Goldring-Koskivirta-Strata-Hasse}, the map $\zeta\colon S_K\to \GZip^\mu$ extends naturally to a map
\begin{equation}
    \zeta^\Sigma\colon S_K^{\Sigma}\to \GZip^\mu.
\end{equation}
Furthermore, by \cite[Theorem 1.2]{Andreatta-modp-period-maps}, the map $\zeta^\Sigma$ is smooth. Since $\zeta$ is surjective, $\zeta^\Sigma$ is also surjective. For $\lambda\in X^*(T)$, denote by $\Vcal^{\Sigma}_I(\lambda)$ the vector bundle $\zeta^{\Sigma,*}(\Vcal_I(\lambda))$. By construction, $\Vcal^{\Sigma}_I(\lambda)$ coincides with the canonical extension of $\Vcal_I(\lambda)$ to $S_K^{\Sigma}$. We have the following Koecher principle:

\begin{theorem}[{\cite[Theorem 2.5.11]{Lan-Stroh-stratifications-compactifications}}] \label{thm-koecher}
The natural map
\begin{equation}
    H^0(S^{\Sigma}_K,\Vcal^{\Sigma}_I(\lambda)) \to H^0(S_K,\Vcal_I(\lambda)) 
\end{equation}
is an isomorphism, except when $\dim(S_K)=1$ and $S^{\Sigma}_K \setminus S_K\neq \emptyset$.
\end{theorem}

We will only consider Shimura varieties satisfying the condition $\dim(S_K)>1$ or $S^{\Sigma}_K \setminus S_K\neq \emptyset$. We are interested in the set of $\lambda\in X^*(T)$ such that $H^0(S_K,\Vcal_I(\lambda))\neq 0$. Equivalently, we may replace the pair $(S_K,\Vcal_I(\lambda))$ by the pair $(S^{\Sigma}_K,\Vcal^{\Sigma}_I(\lambda))$ by Theorem \ref{thm-koecher}. For each field $F$ which is a $\Ocal_{\mathbf{E}_v}$-algebra, define
\begin{equation}
    C_K(F)\colonequals \{ \lambda\in X^*(\mathbf{T}) \mid H^0(\Sscr_K\otimes_{\Ocal_{\mathbf{E}_v}} F,\Vcal_I(\lambda)) \neq 0 \}.
\end{equation}
If $F\subset F'$, one has $C_K(F)=C_K(F')$ by flat base change along the map $\spec(F')\to \spec(F)$. The main goal of this paper is to study $C_K(\overline{\FF}_p)$ for certain simple Shimura varieties of Hodge-type with good reduction at $p$. The cone $\langle C_K(\CC) \rangle$ is less mysterious, it coincides conjecturally with the Griffiths--Schmid cone (see \S\ref{subsec-subcones}). See also \cite{Goldring-Koskivirta-GS-cone} for details.

\subsection{The zip cone}

\subsubsection{Definition} \label{sec-zipcone-def}
For a cocharacter datum $(G,\mu)$ over $\FF_q$, we defined the zip cone of $(G,\mu)$ in \cite[\S 1.2]{Koskivirta-automforms-GZip} and \cite[\S3]{Goldring-Imai-Koskivirta-weights} as
\begin{equation}
    C_{\zip}\colonequals \{\lambda\in X^*(T) \mid H^0(\GZip^\mu, \Vcal_I(\lambda))\neq 0\}.
\end{equation}
This can be seen as a group-theoretical version of the set $C_K(\overline{\FF}_p)$ in the case of Shimura varieties. Since $V_I(\lambda)=0$ when $\lambda$ is not $I$-dominant, we clearly have $C_{\zip} \subset X_{+,I}^*(T)$. One can see that $C_{\zip}$ is an additive submonoid of $X^*(T)$ containing $0$ (\cite[Lemma 1.4.1]{Koskivirta-automforms-GZip}). An additive monoid containing $0$ will be called a cone. For a cone $C\subset X^*(T)$, define the saturated cone $\langle C \rangle$ as:

\begin{equation}
 \langle C \rangle \colonequals \{\lambda\in X^*(T) \mid \exists N\geq 1, N\lambda \in C\}.   
\end{equation}
We say that $C$ is saturated in $X^*(T)$ if $\langle C\rangle =C$. We denote by $C_{\RR_{\geq 0}}$ the subset of $X^*(T)_\RR$ consisting of all linear combinations of elements of $C$ with nonnegative real coefficients. We define $C_{\QQ_{\geq 0}}$ similarly. Note that $X^*(T)\cap C_{\QQ_{\geq 0}}=\langle C_{\zip} \rangle$.

\subsubsection{Motivation}
Before we explain further properties of $C_{\zip}$, we mention the main conjecture that motivates this article. We consider the special fiber $S_K$ of a Hodge-type Shimura variety, and its associated map $\zeta\colon S_K\to \GZip^\mu$. Since $\zeta$ is surjective, it is in particular dominant, which yields an injection
\begin{equation}
    H^0(\GZip^\mu,\Vcal_I(\lambda)) \subset H^0(S_K,\Vcal_I(\lambda)).
\end{equation}
Therefore, we obtain an inclusion $C_{\zip}\subset C_K(\overline{\FF}_p)$.

\begin{conjecture}\label{conj-Sk}
For any Hodge-type Shimura variety, one has:
\begin{equation}
    \langle C_K(\overline{\FF}_p) \rangle = \langle C_{\zip} \rangle.
\end{equation}
\end{conjecture}
Note that $C_K(\overline{\FF}_p)$ highly depends on the choice of the level $K$. Thus, we cannot expect the equality $C_K(\overline{\FF}_p)=C_{\zip}$ on the nose. However, the saturated cone $\langle C_K(\overline{\FF}_p) \rangle$ is independent of the level $K$ by \cite[Corollary 1.5.3]{Koskivirta-automforms-GZip}. For this reason, the above conjecture is not unreasonable. We explain a more general form of the conjecture in section \ref{sec-cone-conj}.

\subsubsection{First properties of $C_{\zip}$}
Define the set of anti-dominant regular characters of $L$ by
\begin{equation}\label{char-L-reg}
X^*_{-}(L)_{\reg} = \{\lambda\in X^*(L) \mid \langle \lambda,\alpha^\vee \rangle<0, \ \forall \alpha\in \Delta^P \}.
\end{equation}
Endow $X_{+,I}^*(T)_{\RR_{\geq 0}}$ with the subspace topology of $X^*(T)_\RR$. Then, $C_{\zip,\RR_{\geq 0}}$ is a neighborhood of $X_{-}^*(L)_{\reg}$ in $X_{+,I}^*(T)_{\RR_{\geq 0}}$ endowed with the real topology induced from $X^*(T)\otimes \RR$ (see \cite[Lemma 4.1.3]{Goldring-Imai-Koskivirta-weights}).

There is an interpretation of $C_{\zip}$ in terms of representation theory. Assume $P$ is defined over $\FF_q$ for simplicity. The Lang torsor morphism $\wp \colon T \to T$, $g\mapsto g\varphi(g)^{-1}$ induces an isomorphism 
\begin{equation}
\wp_* \colon X_*(T)_{\RR} \stackrel{\sim}{\longrightarrow} X_*(T)_{\RR};\ \delta \mapsto \wp \circ \delta  = \delta -q\sigma(\delta).  \label{equ-Plowstar}
\end{equation}
For $\alpha\in \Delta$, define a cocharacter $\delta_\alpha$ by $\delta_\alpha\colonequals \wp_{*}^{-1}(\alpha^\vee)$. For an $L$-representation $V$, define $V^{\Delta^P}_{\geq 0}$ as the direct sum of $T$-weight spaces $V_\nu$ (where $\nu\in X^*(T)$) for those $\nu$ such that $\langle \nu,\delta_\alpha \rangle \geq 0$ for all $\alpha\in \Delta^P$. For example, if $T$ is split over $\FF_q$, then 
$\delta_{\alpha}=-\alpha^\vee /(q-1)$, 
and $V_{\geq 0}^{\Delta^P}$ is simply the direct sum of the weight spaces $V_\nu$ for those $\nu\in X^*(T)$ satisfying $\langle \nu,\alpha^\vee \rangle \leq 0$ for all $\alpha \in \Delta^P$. By \cite[Corollary 3.4.3]{Imai-Koskivirta-vector-bundles}, one has
\begin{equation}\label{sectionsVlambda}
    H^0(\GZip^\mu,\Vcal_I(\lambda))=V_I(\lambda)^{L(\FF_q)}\cap V_I(\lambda)_{\geq 0}^{\Delta^P}
\end{equation}
where $V_I(\lambda)^{L(\FF_q)}$ is the $L(\FF_q)$-invariant subspace of $V_I(\lambda)$. There is also a description in the general case (when $P$ is not necessarily defined over $\FF_q$) involving the Brylinski--Kostant filtration of $V_I(\lambda)$ (see \cite[Theorem 3.4.1]{Imai-Koskivirta-vector-bundles}). Consequently, the cone $C_{\zip}$ is determined by the behaviour of the representation $V_I(\lambda)$ viewed simultaneously as a $L(\FF_q)$-representation and as a $T$-representation.

\subsubsection{Subcones of $C_{\zip}$}\label{subsec-subcones}

In general, it is difficult to determine $C_{\zip}$ or even $\langle C_{\zip} \rangle$. Therefore, it is useful to seek approximations of $C_{\zip}$ by subcones. We defined in \cite{Goldring-Imai-Koskivirta-weights} and \cite{Goldring-Koskivirta-GS-cone} several cones, that we represent in the diagram below.

\begin{equation}\label{conediag}
\xymatrix@1@M=5pt{
&\langle C_{\Hasse} \rangle \ar@{^{(}->}[rd] & & & \\
X^*_{-}(L) \ar@{^{(}->}[r]  \ar@{^{(}->}[rd]  & C_{\rm hw} \ar@{^{(}->}[r] &  \langle C_{\zip} \rangle \ar@{^{(}->}[r] & C^{+,I}_{\rm unip} \ar@{^{(}->}[r] &  X_{+,I}^*(T)\\
& C_{\GS} \ar@{^{(}->}[ru] \ar@{^{(}->}[r] & C_{\lw} \ar@{^{(}->}[u]_-{\textrm{if $P$ is defined over $\FF_{q^2}$}}   & &}
\end{equation}
All arrows of this diagram are inclusions. We briefly recall the definitions of these cones and their interpretation.

\paragraph{The Griffiths-Schmid cone $C_{\GS}$:} It is defined as the set of $\lambda\in X^*(T)$ satisfying
\begin{align}
\langle \lambda, \alpha^\vee \rangle &\geq 0 \ \textrm{ for }\alpha\in I, \\
\langle \lambda, \alpha^\vee \rangle &\leq 0 \ \textrm{ for }\alpha\in \Phi_+ \setminus \Phi_{L,+}. 
\end{align}
One sees easily that $\lambda\in C_{\GS}$ if and only if $-w_{0,I}\lambda$ is $G$-dominant. Clearly $C_{\GS}$ is a saturated subcone of $X^*(T)$ and contains $X_{-}^*(L)$. Asume that $(G,\mu)$ is attached to a Hodge-type Shimura variety with good reduction at $p$, as in \S\ref{subsec-Shimura}. In this case, $C_{\GS}$ has the following interpretation. Recall that we defined
\begin{equation}\label{eq-CKC}
C_K(\CC)\colonequals\{\lambda\in X^*(T) \mid H^0(\Sscr_K\otimes \CC,\Vcal_I(\lambda))\neq 0\}.
\end{equation}
Based on the results of \cite{Griffiths-Schmid-homogeneous-complex-manifolds}, it is expected (but not proved in full generality) that $\langle C_K(\CC) \rangle=C_{\GS}$ (the inclusion $\langle C_K(\CC) \rangle \subset C_{\GS}$ is proved in \cite[Theorem 1]{Goldring-Koskivirta-GS-cone}). In the context of Shimura varieties, it is easy to show by a reduction modulo $p$ argument that $C_K(\CC)\subset C_K(\overline{\FF}_p)$ (see \cite[Proposition 1.8.3]{Koskivirta-automforms-GZip}). Therefore, if Conjecture \ref{conj-Sk} is correct, one should expect an inclusion $C_{\GS}\subset \langle C_{\zip} \rangle$. This was indeed showed in \cite{Goldring-Imai-Koskivirta-weights}, which gives some evidence for Conjecture \ref{conj-Sk}:

\begin{theorem}[{\cite[Theorem 6.4.2]{Goldring-Imai-Koskivirta-weights}}] \label{CGScontained}
We have $C_{\GS}\subset \langle C_{\zip} \rangle$.
\end{theorem}
When $L$ is defined over $\FF_q$, this inclusion was already showed in \cite[Corollary 3.5.6]{Koskivirta-automforms-GZip}. However, for general $P$ it requires much more work.

\paragraph{The Hasse cone $C_{\Hasse}$:}
The Hasse cone $C_{\Hasse}$ is related to the flag stratification on $\GF^{\mu}$ (section \ref{sec-zip-flag}). The flag strata of codimension one are $(\Fcal_{w_0 s_{\alpha}})_{\alpha\in \Delta}$. For each $\alpha\in \Delta$, there exists a partial Hasse invariant $\Ha_{\alpha}$ by \cite[Proposition 5.2.7]{Imai-Koskivirta-partial-Hasse}. By definition, this is a section of $\Vcal_{\flag}(\lambda_{\alpha})$ (for some $\lambda_{\alpha}\in X^*(T)$) over $\GF^{\mu}$ such that the vanishing locus of $\Ha_{\alpha}$ is the Zariski closure of $\Fcal_{w_0 s_{\alpha}}$. The cone $\langle C_{\Hasse} \rangle$ can be defined as the saturated cone generated by all the $(\lambda_{\alpha})_{\alpha\in \Delta}$ and by $X^*(G)$ (which corresponds to the torsion line bundles on $\GZip^{\mu}$). The cone $\langle C_{\Hasse} \rangle$ is independent of the choice of $\Ha_{\alpha}$. This definition is similar to the one used by Diamond--Kassaei in \cite{arxiv-Diamond-Kassaei-cone-minimal}. However, to avoid the slight ambiguity in the choice of partial Hasse invariants, we prefer to use the following, more precise definition.

\begin{definition}[{\cite[Definition 1.7.1]{Koskivirta-automforms-GZip}}] \label{definition-CHasse}
Define $C_{\Hasse}$ as the image of $X^*_+(T)$ by
\begin{equation}
h_{\Zcal}\colon X^*(T)\to X^*(T), \quad \lambda \mapsto \lambda-q w_{0,I} (\sigma^{-1} \lambda).
\end{equation}\label{equ-maph}
\end{definition}
As usual, we let $\langle C_{\Hasse} \rangle$ denote its saturated cone. The cone $C_{\Hasse}$ can be interpreted as the set of weights of nonzero automorphic forms on $\GF^{\mu}$ which arise by pullback from the stack $\Sbt$ via the map \eqref{psimap}. Using the identification \eqref{H0-ident}, we have immediately $C_{\Hasse}\subset C_{\zip}$. When $G$ is split over $\FF_q$, the saturated cone $\langle C_{\Hasse} \rangle$ has a simple form: By inverting the map $h_\Zcal$ of Definition \ref{definition-CHasse}, we see that $\langle C_{\Hasse} \rangle$ is the set of $\lambda\in X^*(T)$ such that $ \lambda + q w_{0,I}\lambda \in X_-^*(T)$.

\paragraph{The highest weight cone:} Let $L_0$ be the largest subgroup of $L$ defined over $\FF_q$, i.e. $L_0=\bigcap_{n\in \ZZ}\sigma^n(L)$, it is a Levi subgroup containing $T$. For $\alpha\in \Delta$, let $r_\alpha$ be the smallest integer $r\geq 1$ such that $\sigma^r(\alpha)=\alpha$. We define the cone $C_{\hw}$ as the set of $\lambda\in X_{+,I}^*(T)$ such that for all $\alpha \in \Delta^P$, one has
\begin{equation}\label{formula-norm}
\sum_{w\in W_{L_0}(\FF_q)} \sum_{i=0}^{r_\alpha-1} q^{i+\ell(w)} \ \langle w\lambda, \sigma^i(\alpha^\vee) \rangle\leq 0.
\end{equation}
We explain the interpretation of this cone. Denote by $L_{\varphi}\subset E$ the stabilizer of $1\in G_k$ with respect to the action of $E$ on $G_k$. The first projection $\pr_1\colon E\to P$ induces a closed immersion $L_{\varphi}\to P$. Furthermore, its image is contained in $L$. Identifying $L_{\varphi}$ with a subgroup of $L$, it can be written as $L_{\varphi}=L_0(\FF_q)\rtimes L_{\varphi}^{\circ}$, where the connected component $L_{\varphi}^{\circ}$ is a finite unipotent subgroup (\cite[Theorem 8.1]{Pink-Wedhorn-Ziegler-zip-data}). Let $m\geq 1$ such that $L_{\varphi}^{\circ}$ is annihilated by $\varphi^m$. For $f\in V_I(\lambda)$, we defined in \cite{Goldring-Imai-Koskivirta-weights} the $L_\varphi$-norm of $f$
\begin{equation}
    \Norm_{L_\varphi}(f) \in H^0(\Ucal_\mu, \Vcal_I(d\lambda))
\end{equation}
where $d=|L_0(\FF_q)|q^m$ and $\Ucal_\mu\subset \GZip^{\mu}$ is the open, $\mu$-ordinary stratum in $\GZip^{\mu}$. In particular, consider the case when $f=f_{\lambda, \high}$, where $f_{\lambda, \high}$ is the highest weight vector of $V_I(\lambda)$. Then, we showed (\cite{Goldring-Imai-Koskivirta-weights}) that $\Norm_{L_{\varphi}}(f_{\lambda,\high})$ extends from $\Ucal_\mu$ to $\GZip^\mu$ if and only if $\lambda\in C_{\hw}$. Thus, for each $\lambda\in C_{\hw}$, we obtain a nonzero automorphic form of weight $d\lambda$. In particular, we have $C_{\hw}\subset \frac{1}{d} C_{\zip}\subset \langle C_{\zip} \rangle$. We will see that the highest weight cone $C_{\hw}$ plays an important role in some cases, for example in the case of Siegel-type Shimura varieties associated to the group $\GSp(6)$.

\paragraph{The lowest weight cone}
Denote by $P_0$ the largest subgroup of $P$ defined over $\FF_q$. It is a parabolic subgroup with Levi subgroup $L_0$. Denote its type by $I_0\subset I$. For $\lambda\in X^*(T)$, write $\lambda_0\colonequals w_{0,I_0}w_{0,I}\lambda$. We define $C_{\rm lw}$ as the set of $\lambda\in X_{+,I}^*(T)$ such that for all $\alpha\in \Delta^{P_0}$,
\begin{equation}\label{formula-norm-low}
\sum_{w\in W_{L_0}(\FF_q)} \sum_{i=0}^{r_\alpha-1} q^{i+\ell(w)} \ \langle w \lambda_0, \sigma^i(\alpha^\vee) \rangle\leq 0
\end{equation}
where $r_\alpha$ is again an integer such that $\sigma^{r_\alpha}(\alpha)=\alpha$. Note that when $P$ is defined over $\FF_q$, we have $P_0=P$ and hence $C_{\rm lw}=C_{\hw}$. In general, we do not know if $C_{\rm lw}$ is contained in $\langle C_{\zip} \rangle$. However, we showed in \cite[Theorem 5.2.2]{Goldring-Imai-Koskivirta-weights} that under Condition 5.1.1 of \loccitn, one has $C_{\rm lw}\subset \langle C_{\zip} \rangle$. For example, this condition is satisfied when $P$ is defined over $\FF_{q^2}$. This will be the case for all cases considered in this paper. The terminology "lowest weight cone" stems from the fact that if $\lambda\in C_{\rm lw}$, then $\Norm_{L_\varphi}(f_{\lambda, \textrm{low}})$ extends to $\GZip^\mu$ (at least under the aforementioned condition), where $f_{\lambda, \textrm{low}}$ is the lowest weight vector of $V_I(\lambda)$. The lowest weight cone always satisfies $C_{\GS}\subset C_{\lw}$, contrary to the highest weight cone $C_{\hw}$, which does not always contain $C_{\GS}$.

\paragraph{The unipotent-invariance cone}
In \cite{Goldring-Koskivirta-GS-cone}, we determined an upper bound $C_{\rm unip}$ for the cone $C_{\zip}$. The definition of $C_{\rm unip}$ is not enlightening, so we only explain a slightly larger cone $C_{\orb}$ under the assumption that $G$ is split over $\FF_{q}$. Let $W_L=W(L,T)$ be the Weyl group of $L$. Note that $W_L$ acts naturally on the set $\Phi_+\setminus \Phi_{L,+}$. Let $\Ocal\subset \Phi_+\setminus \Phi_{L,+}$ be a $W_L$-orbit and let $S\subset \Ocal$ be any subset. Set
\begin{equation}\label{ineq-subset-intro}
\Gamma_{\Ocal,S}(\lambda)\colonequals \sum_{\alpha \in \Ocal\setminus S} \langle \lambda,\alpha^\vee \rangle \ + \ \frac{1}{q} \sum_{\substack{\alpha\in S}} \langle  \lambda, \alpha^\vee \rangle
\end{equation}
Then, the cone $C_{\orb}$ is defined by
\begin{equation}\label{L-min-cone}
C_{\rm orb} = \{ \lambda\in X^*(T) \ \mid \ \Gamma_{\Ocal,S}(\lambda)\leq 0 \ \textrm{for all orbits $\Ocal$ and all subsets} \ S\subset \Ocal \}.   
\end{equation}
In the diagram at the beginning of this section, the cone $C_{\rm unip}^{+,I}$ denotes the intersection $C_{\rm unip}\cap X_{+,I}^*(T)$. When $G$ is split over $\FF_{q}$, it is contained in $C^{+,I}_{\orb}\colonequals C_{\orb}\cap X_{+,I}^*(T)$. For example, in the case $G=\Sp(2n)_{\FF_q}$, this cone is explicitly given in section \ref{sec-symp}.

\subsubsection{Hasse-type cocharacter datum}\label{sec-hassetype}

We describe a family of cocharacter data $(G,\mu)$ where the (saturated) zip cone $\langle C_{\zip} \rangle$ is entirely determined.

\begin{theorem}[{\cite[Theorem 4.3.1]{Goldring-Imai-Koskivirta-weights}}] \label{thm-Hasse-type}
The following are equivalent:
\begin{equivlist}
\item One has $\langle C_{\Hasse} \rangle = \langle C_{\zip} \rangle$.
\item One has $C_{\GS}\subset \langle C_{\Hasse} \rangle$.
\item $L$ is defined over $\FF_q$ and $\sigma$ acts on $\Delta_L$ by $-w_{0,L}$.
\end{equivlist}
\end{theorem}

When Condition (iii) of Theorem \ref{thm-Hasse-type} is satisfied, we say that $(G,\mu)$ is of Hasse-type. We give a family of examples satisfying this condition. Let $J$ be the symmetric matrix of size $2n+1$ ($n\geq 1$) defined by
\begin{equation}
    J\colonequals\begin{pmatrix}
    &&1\\ &\iddots& \\ 1&&
    \end{pmatrix}.
\end{equation}
Let $G$ be the special orthogonal group over $\FF_q$ attached to $J$. Let $T$ be the maximal, diagonal torus of $G$, consisting of matrices of the form $t=\diag(t_1, \dots ,t_n,1,t_n^{-1}, \dots ,t_1^{-1})$. We identify $X^*(T)\simeq \ZZ^n$ such that $(a_1, \dots ,a_n)\in \ZZ^n$ corresponds to $t\mapsto t_1^{a_1} \dots t_n^{a_n}$. Define a cocharacter $\mu\colon \GG_{\textrm{m}}\to G$ by $z\mapsto \diag(z,1, \dots ,1,z^{-1})$. Then the zip datum attached to $(G,\mu)$ satisfies the conditions of Theorem \ref{thm-Hasse-type} (see \cite[\S7.2]{Goldring-Imai-Koskivirta-weights}). This example corresponds to Shimura varieties associated to spinor groups. Concretely, we have in this case:
\begin{equation}
    \langle C_{\zip} \rangle =\langle C_{\Hasse} \rangle = \{ (a_1,\dots ,a_n)\in X^*_{+,I}(T) \mid \   (q+1)a_1+(q-1)a_2\leq 0\} 
\end{equation}
Other examples of Hasse-type cocharacter data are: Siegel-type Shimura varieties attached to $\GSp(4)$, unitary Shimura varieties attached to $\GU(2,1)$ at a split prime, Hilbert--Blumenthal Shimura varieties.

\section{The cone conjecture}
In this section, we explain the main conjecture and the strategy of proof. We do not restrict ourselves to Shimura varieties, we consider instead more general schemes which admit a "nice" map to the stack of $G$-zips.

\subsection{Set-up}
\subsubsection{Stratification of $S$}
Recall that $k$ denotes an algebraic closure of $\FF_q$. Let $(G,\mu)$ be a cocharacter datum over $\FF_q$. Let $S$ be a $k$-scheme endowed with a smooth, surjective $k$-morphism $\zeta\colon S\to \GZip^\mu$. For $w\in {}^IW$ (or $w\in W^J$), we define 
 \begin{equation}\label{Sw-def}
     S_w\colonequals \zeta^{-1}(\Xcal_w).
 \end{equation}
It is a locally closed subset of $S$, and we endow $S_w$ with the reduced subscheme structure. Since $\zeta$ is smooth, the Zariski closure $\overline{S}_w$ coincides with $\zeta^{-1}(\overline{\Xcal}_w)$. We obtain a stratification on $S$ by smooth, locally closed subschemes. For $\lambda\in X^*(T)$, we denote again by $\Vcal_I(\lambda)$ the pullback via $\zeta$ of $\Vcal_I(\lambda)$ on $\GZip^\mu$.

\subsubsection{Flag space of $S$}\label{sec-flagspace}
Define the flag space of $S$ as the fiber product
$$\xymatrix@1@M=5pt{
\Flag(S)\ar[r]^-{\zeta_{\flag}} \ar[d]_{\pi_S} & \GF^\mu \ar[d]^{\pi} \\
S \ar[r]_-{\zeta} & \GZip^\mu.
}$$
For $w\in W$, define $\Flag(S)_w\colonequals \zeta_{\flag}^{-1}(\Fcal_w)$. Again, we obtain on $\Flag(S)$ a stratification by locally closed, smooth subschemes.
If $\ell(w)=n$, we call $\Flag(S)_w$ a stratum of length $n$. For $\lambda\in X^*(T)$, we denote again by $\Vcal_{\flag}(\lambda)$ the pullback of the line bundle $\Vcal_{\flag}(\lambda)$ via $\zeta_{\flag}$. Similarly to $\GZip^\mu$, we have the formula $\pi_{S,*}(\Vcal_{\flag}(\lambda))=\Vcal_I(\lambda)$. In particular, we have an identification
\begin{equation} \label{identif-lambda}
    H^0(S,\Vcal_I(\lambda))=H^0(\Flag(S),\Vcal_{\flag}(\lambda)).
\end{equation}

\subsection{The cone conjecture}\label{sec-cone-conj}
Let $S$ be a $k$-scheme endowed with a morphism $\zeta \colon S\to \GZip^\mu$. We make the following assumption:
\begin{assumption}\label{assume} \ 
\begin{definitionlist}
\item $\zeta$ is smooth.
\item The restriction of $\zeta$ to every connected component of $S$ is surjective.
\item For all $w\in W$ such that $\ell(w)=1$, $\overline{\Flag(S)}_w$ is 
pseudo-complete.
\end{definitionlist}
\end{assumption}
Recall that a $k$-scheme $X$ is called pseudo-complete if any section of $\Ocal_X(X)$ is Zariski locally constant on $X$. In particular, Assumption (c) is satisfied if $S$ is a proper $k$-scheme. We define
\begin{equation}\label{CS-def}
C_S\colonequals\{\lambda\in X^*(T) \mid H^0(S,\Vcal_I(\lambda))\neq 0\}.
\end{equation}
Since $\zeta$ is surjective, the pullback via $\zeta$ of a nonzero section of $\Vcal_I(\lambda)$ is again nonzero. Hence $H^0(\GZip^\mu,\Vcal_I(\lambda))\subset H^0(S,\Vcal_I(\lambda))$. In particular, we have $C_{\zip}\subset C_S$. By analogy to the case of Shimura varieties, we sometimes call elements of $H^0(S,\Vcal_I(\lambda))$ automorphic forms of weight $\lambda$ on $S$.

\begin{conjecture}\label{conj-S}
Under Assumption \ref{assume}, we have $\langle C_{S} \rangle=\langle C_{\zip} \rangle$.
\end{conjecture}

In \cite[Conjecture 2.1.6]{Goldring-Koskivirta-global-sections-compositio}, we formulated Conjecture \ref{conj-S} with the additional assumption that the pair $(G,\mu)$ is of "connected Hodge-type". Furthermore, we stated that the conjecture does not hold in general without this assumption and gave in \cite[Proposition 5.1.3]{Goldring-Koskivirta-global-sections-compositio} an example of a map $S\to \GZip^\Zcal$ which does not satisfy the conjecture. However, the zip datum $\Zcal$ considered in this counter-example is not of cocharacter-type (i.e does not arise from a cocharacter $\mu\colon \GG_{\textrm{m},k}\to G_k$). Hence, we do not know whether "connected Hodge-type" is a necessary assumption in the setting considered here, so we removed this condition from Assumption \ref{assume}.

We now discuss the example of Shimura varieties. As we noted in \S \ref{sec-tor}, the map $\zeta^\Sigma\colon S^{\Sigma}_K \to \GZip^\mu$ is smooth and surjective. Moreover, \cite[Proposition 6.20]{Wedhorn-Ziegler-tautological} shows that any connected component $S^\circ\subset S^{\Sigma}_{K}$ intersects the unique zero-dimensional stratum. Since the map $\zeta^\Sigma \colon S^\circ \to \GZip^\mu$ is smooth, its image is open, hence surjective. Therefore, $\zeta^\Sigma$ satisfies Condition (b) of Assumption \ref{assume}. Furthermore, since $S_K^{\Sigma}$ is proper, Assumption \ref{assume}(c) is also satisfied. Hence Conjecture \ref{conj-S} applies to Shimura varieties of Hodge-type. In the case of $S=S^{\Sigma}_K$, the set $C_S$ coincides with $C_K(\overline{\FF}_p)$ by Theorem \ref{thm-koecher}. Therefore, Conjecture \ref{conj-S} is a generalization of Conjecture \ref{conj-Sk}.

\begin{rmk}\label{rmk-GeqP}
Assume that $G=P$. This is equivalent to $\mu\colon \GG_{\textrm{m},k}\to G_k$ being a central cocharacter. In this case, Conjecture \ref{conj-S} holds for any scheme $S$ endowed with a map $\zeta\colon S\to \GZip^{\mu}$ (without any assumption). Note that in this case ${}^IW=\{e\}$, so the underlying topological space of $\GZip^{\mu}$ is a single point, hence $\zeta$ is obviously surjective. Moreover, in this case we have $C_{\GS}=X_{+,I}^*(T)$. Since we always have $C_{\GS}\subset \langle C_{\zip}\rangle\subset \langle C_{S}\rangle \subset X_{+,I}^*(T)$, the result follows.
\end{rmk}

\subsection{Strategy}
\subsubsection{Hasse cones $C_{\Hasse,w}$}\label{sec-Hasse-cone}
\label{subsec-sum-inter}
For $w\in W$, denote by $E_w$ the set of positive roots $\alpha$ such that $ws_\alpha<w$ and $\ell(ws_\alpha)=\ell(w)-1$. We call $ws_\alpha$ (for $\alpha\in E_w$) a lower neighbor of $w$. We recall Chevalley's formula for the strata $\Sbt_w$ of $\Sbt$ defined in \eqref{Sbtw}. For $(\lambda,\nu)\in X^*(T)\times X^*(T)$, one attaches a line bundle $\Vcal_{\Sbt}(\lambda,\nu)$ on $\Sbt$ (in \cite[\S 2.2]{Goldring-Koskivirta-Strata-Hasse}, this line bundles was denoted by $\Lscr_{\Sbt}(\lambda,\nu)$). A section of $\Vcal_{\Sbt}(\lambda,\nu)$ over $\Sbt_w$ can be viewed as a regular map $f\colon BwB\to \AA^1$ satisfying $f(agb^{-1})=\lambda(a)\nu(b)f(g)$ for all $a,b\in B$ and all $g\in G$.

\begin{theorem}[{\cite[Theorem 2.2.1]{Goldring-Koskivirta-Strata-Hasse}}]\label{brion}
Let $w \in W$. One has the following:
\begin{assertionlist}
\item \label{brion1} $H^0\left(\Sch_w,\Vcal_{\Sch}(\lambda,\mu)\right)\neq 0\Longleftrightarrow \mu = -w^{-1} \lambda$.
\item \label{brion2} $\dim_k H^0\left(\Sch_w,\Vcal_{\Sch}(\lambda,-w^{-1} \lambda) \right)=1$.
\item \label{brion3} For any nonzero $f\in H^0\left(\Sch_w,\Vcal_{\Sch}(\lambda,-w^{-1} \lambda) \right)$, one has
\begin{equation}\label{briondiv}
\div(f)=-\sum_{\alpha \in E_w} \langle \lambda , w\alpha^\vee \rangle \overline{\Sbt}_{w s_\alpha}.
\end{equation}
\end{assertionlist}
\end{theorem}

For each $w\in W$ and $\lambda\in X^*(T)$, denote by $f_{w,\lambda}$ a nonzero element of the one-dimensional space $H^0(\Sch_w,\Vcal_{\Sch}(\lambda,-w^{-1} \lambda))$. Define $X_{+,w}^*(T)\subset X^*(T)$ as the subset of $\chi\in X^*(T)$ such that $\langle \chi,\alpha^\vee \rangle \geq 0$ for all $\alpha\in E_w$. For $\chi\in X_{+,w}^*(T)$, write $\lambda=-w\chi$. By Theorem \ref{brion}(2)-(3), we have
\begin{equation}
  f_{w,\lambda}\in H^0(\overline{\Sbt}_w,\Vcal_{\Sbt}(\lambda,-w^{-1}\lambda)).
\end{equation}
For $\lambda,\nu \in X^*(T)$, one has the formula
\begin{equation}
\psi^*(\Vcal_{\Sbt}(\lambda,\nu))=\Vcal_{\flag}(\lambda + q w_{0,I}w_0\sigma^{-1}(\nu))    
\end{equation}
by \cite[Lemma 3.1.1 (b)]{Goldring-Koskivirta-Strata-Hasse} (note that \loccit contains a typo; it should be $\sigma^{-1}$ instead of $\sigma$). In particular, the pullback $\psi^*(\Vcal_{\Sbt}(\lambda,-w^{-1}\lambda))$ coincides with $\Vcal_{\flag}(\lambda - q w_{0,I}w_0\sigma^{-1}(w^{-1}\lambda))$. Define a map 
\begin{equation}
    h_w\colon X^*(T)\to X^*(T), \quad \chi\mapsto -w\chi + q w_{0,I}w_0\sigma^{-1}(\chi).
\end{equation}
Define the Hasse cone $C_{\Hasse,w}$ by
\begin{equation} \label{CHassew-def}
    C_{\Hasse,w}\colonequals h_w(X_{+,w}^*(T)).
\end{equation}
Concretely, $C_{\Hasse,w}$ is the set of all possible weights $\lambda$ of nonzero sections over $\overline{\Fcal}_{w}$ of $\Vcal_{\flag}(\lambda)$ which arise by pullback from $\overline{\Sbt}_w$. For each $\chi\in X^*(T)$, define
\begin{equation}
    \Ha_{w,\chi} \colonequals \psi^*(f_{w,-w\chi}).
\end{equation}
Then $\Ha_{w,\chi}$ a section over the stratum $\Fcal_w$ of the line bundle $\Vcal_{\flag}(h_w(\chi))$. Furthermore, by the above discussion, $\Ha_{w,\chi}$ extends to $\overline{\Fcal}_w$ if and only if $\chi\in X_{+,w}^*(T)$. The multiplicity of $\div(\Ha_{w,\chi})$ along $\overline{\Fcal}_{ws_\alpha}$ is $\langle \chi,\alpha^\vee \rangle$.

\subsubsection{Global partial Hasse invariants}
Consider the case $w=w_0$ of the longest element of $W$. In this case, we have $E_w=\Delta$ and $X_{+,w}^*(T)=X_{+}^*(T)$. Since $X_{+}^*(T)$ is invariant by $-w_0$, the set $C_{\Hasse,w_0}$ coincides with the set $C_{\Hasse}$ of Definition \ref{definition-CHasse}. In the case $w=w_0$, we simply write for $\chi\in X^*(T)$:
\begin{equation}\label{Hachi-def}
\Ha_\chi \colonequals \Ha_{w_0,\chi}.    
\end{equation}
Let $\chi_\alpha$ be a character such that $\langle \chi_\alpha,\alpha^\vee \rangle >0$ and $\langle \chi_\alpha,\beta^\vee \rangle=0$ for all $\beta\in \Delta\setminus \{\alpha\}$. For any such $\chi_\alpha$, the section $\Ha_{\chi_\alpha}$ is a global section over $\GF^\mu$ whose vanishing locus is exactly the Zariski closure $\overline{\Fcal}_{w_0s_\alpha}$ of the codimension one stratum $\Fcal_{w_0s_\alpha}$. Such sections are studied in detail by Imai and the second-named author in \cite{Imai-Koskivirta-partial-Hasse}. Instead of $\Ha_{\chi_\alpha}$, we often simply write $\Ha_\alpha$. Note that $\chi_\alpha$ is well defined up to $X^*(G)$ and up to positive multiple. Hence, the weight of $\Ha_\alpha$, given by
\begin{equation} \label{lamalph-eq}
    \lambda_\alpha\colonequals h_{w_0}(\chi_\alpha)
\end{equation}
is also well defined up to the same ambiguity.

\begin{definition}\label{part-Hasse-glob}
We call $\Ha_\alpha$ a partial Hasse invariant for $\alpha\in \Delta$.
\end{definition}

Partial Hasse invariants seem to play an important role in the theory of mod $p$ automorphic forms. As an illustration, we explain the main result of Diamond--Kasaei in \cite{Diamond-Kassaei},  extended in \cite{arxiv-Diamond-Kassaei-cone-minimal}. The authors study Hilbert--Blumenthal Shimura varieties attached to $\mathbf{G}=\Res_{F/\QQ}(\GL_{2,F})$ (where $F/\QQ$ is a totally real extension of degree $d=[F:\QQ]$). They prove results about Hilbert automorphic forms in characteristic $p$. We give a short explanation of \cite[Corollary 5.4]{Diamond-Kassaei}. To simplify, let $p$ be a prime number unramified in $F$ (in \cite{arxiv-Diamond-Kassaei-cone-minimal}, $p$ is allowed to be ramified in $F$). Fix a small enough level $K^p\subset \mathbf{G}(\AA_f^p)$ outside $p$. Let $X$ be the Pappas--Rapoport integral model over $\ZZ_p$ of the associated Hilbert modular Shimura variety defined in \cite[\S 2]{Diamond-Kassaei} (denoted by $X^{\rm PR}$ in \cite{arxiv-Diamond-Kassaei-cone-minimal}). Since $p$ is assumed unramified, it is the same as the Deligne--Pappas model $X^{\rm DP}$ (see \cite[\S 3]{arxiv-Diamond-Kassaei-cone-minimal}). The scheme $X_{\overline{\FF}_p}$ is smooth of dimension $d$ over $\overline{\FF}_p$. It parametrizes tuples $(A,\lambda,\iota,\overline{\eta})$ of abelian schemes of dimension $d$ endowed with a principal polarization $\lambda$, an action $\iota$ of $\Ocal_{F}$ on $A$ and a $K^p$-level structure $\overline{\eta}$.

Let $\Sigma\colonequals \Hom(F,\overline{\QQ}_p)$ be the set of field embeddings $F\to \overline{\QQ}_p$. Write $(\mathbf{e}_\tau)_\tau$ for the canonical basis of $\ZZ^\Sigma$. Let $\sigma$ denote the action of Frobenius on $\Sigma$. For each $\tau\in \Sigma$, there is an associated line bundle $\omega_\tau$ on $X_{\overline{\FF}_p}$. For $\mathbf{k}=\sum_\tau k_\tau \mathbf{e}_\tau\in \ZZ^\Sigma$, let $\omega^\mathbf{k}\colonequals \bigotimes_{\tau\in \Sigma}\omega_{\tau}^{k_\tau}$. Elements of $H^0(X_{\overline{\FF}_p},\omega^\mathbf{k})$ are modulo $p$ Hilbert modular forms of weight $\mathbf{k}$. Andreatta--Goren (\cite{Andreatta-Goren-book}) constructed partial Hasse invariants $\Ha_\tau$ for each $\tau\in \Sigma$. The weight of $\Ha_\tau$ is given by
\begin{equation}
    \mathbf{h}_\tau \colonequals e_{\tau}-pe_{\sigma^{-1} \tau}.
\end{equation}
Note that the sign of $\mathbf{h}_\tau$ is different in \cite{Andreatta-Goren-book}, due to a different convention of positivity. The main property of $\Ha_\tau$ is that it cuts out a single Ekedahl--Oort stratum of $X_{\overline{\FF}_p}$. In this example, the Ekedahl--Oort stratification are given by the isomorphism class of the $p$-torsion $A[p]$ (with its additional structure given by $\lambda$ and $\iota$). There is a unique open stratum (on which $A$ is an ordinary abelian variety). The codimension 1 strata can be labeled $(S_\tau)_{\tau\in \Sigma}$ and the vanishing locus of $\Ha_\tau$ coincides with the Zariski closure of $S_\tau$ in $X_{\overline{\FF}_p}$. Diamond--Kassaei define the Hasse cone as the subset of $\QQ_{\geq 0}^\Sigma$ spanned over $\QQ_{\geq 0}$ by the weights $(\mathbf{h}_\tau)_{\tau}$. With the notation explained in section \ref{sec-zipcone-def}, this corresponds for our notation to $C_{\Hasse,\QQ_{\geq 0}}$. Diamond--Kassaei prove divisibility results by partial Hasse invariants:

\begin{theorem}[Diamond--Kassaei, {\cite[Theorem 5.1, Corollary 5.4]{Diamond-Kassaei}}] \label{thmDK} \ 
\begin{assertionlist}
    \item\label{thm-DK-item1} Let $f\in H^0(X_{\overline{\FF}_p},\omega^\mathbf{k})$  and assume that $pk_\tau > k_{\sigma^{-1}\tau}$. Then $f$ is divisible by $\Ha_\tau$.
    \item\label{thm-DK-item2} If $H^0(X_{\overline{\FF}_p},\omega^\mathbf{k})\neq 0$, then $\mathbf{k}\in \langle C_{\Hasse}\rangle$.
\end{assertionlist}
\end{theorem}
Note that when $\tau = \sigma^{-1}\tau$, then by \eqref{thm-DK-item2} there does not exist any nonzero form with $pk_\tau > k_{\sigma^{-1}\tau}$, hence \eqref{thm-DK-item1} does not provide any useful information. The authors define a minimal cone $C_{\rm min}\subset C_{\Hasse,\QQ_{\geq 0}}$ as follows:
\begin{equation}
    C_{\rm min}=\{\mathbf{k}\in \QQ_{\geq 0}^\Sigma \ | \ pk_\tau \leq k_{\sigma^{-1}\tau}\}.
\end{equation}
Theorem \ref{thmDK}\eqref{thm-DK-item1} shows that any Hilbert modular form $f$ of weight $\mathbf{k}$ can be written as a product $f=f_{\min}H$, where $f_{\min}$ has weight $\mathbf{k}_{\min}\in C_{\min}$ and $H$ is a product of partial Hasse invariants. One sees easily that \eqref{thm-DK-item2} is a direct consequence of \eqref{thm-DK-item1}. In particular, \eqref{thm-DK-item2} says exactly that $C_{X}\subset \langle C_{\Hasse} \rangle$. Since $C_{\Hasse}\subset C_X$ always holds, one obtains $\langle C_{X}\rangle =\langle C_{\Hasse} \rangle$. As we already mentioned, we cannot expect $\langle C_{S}\rangle =\langle C_{\Hasse} \rangle$ for general Shimura varieties $S$. However, when $(G,\mu)$ is of Hasse-type (see \S\ref{sec-hassetype}), we do expect such an equality by Theorem \ref{thm-Hasse-type}.

\subsubsection{Separating systems}

We now consider arbitrary strata of $\GF^\mu$.
\begin{definition}\label{def-Q-sep}
Let $w\in W$. We say that $w$ admits a full separating system of partial Hasse invariants if for each $\alpha\in E_w$, there exists $\chi_\alpha\in X^*(T)$ such that
\begin{definitionlist}
\item $\langle \chi_{\alpha},\alpha^\vee \rangle >0$
\item $\langle \chi_{\alpha},\beta^\vee \rangle =0$ for all $\beta\in E_{w}\setminus \{\alpha\}$.
\end{definitionlist}
\end{definition}
In particular, in this case one has $\chi_\alpha\in X^*_{+,w}(T)$. Note that $w$ admits a full separating system of partial Hasse invariants if and only if the linear forms $\{\alpha^\vee\}_{\alpha\in E_w}$ are linearly independent over $\QQ$. Let $\chi_{\alpha}$ be a character satisfying the conditions (a) and (b) of Definition \ref{def-Q-sep}. By Theorem \ref{brion}(3), the section $f_{w,-w\chi_\alpha}$ over $\overline{\Sbt}_w$ vanishes exactly on the closed subscheme $\overline{\Sbt}_{ws_\alpha}$. Similarly, $\Ha_{w,\chi_\alpha}$ vanishes exactly on $\overline{\Fcal}_{ws_\alpha}$. This explains the terminology "full separating system of partial Hasse invariants" used in Definition \ref{def-Q-sep}.

There are examples where the cardinality of $E_w$ exceeds the rank of $X^*(T)$, which prevents $w$ from admitting such a separating system. We now define a more flexible notion of separating system. For each $w\in W$, let $\EE_w\subset E_w$ be a subset. Furthermore, for each $\alpha\in \EE_w$, let $\chi_{\alpha}\in X^*(T)$ be a character. For $w_1,w_2\in \EE_w$, say that $w_1$, $w_2$ are connected if $w_1$ and $w_2$ have a common lower neighbor (i.e if $E_{w_1}\cap E_{w_2}\neq \emptyset$). 

\begin{definition}\label{def-sep-syst}
We say that the family $\EE=(\EE_w, \{\chi_\alpha \}_{\alpha\in \EE_w})_{w\in W}$ is a separating system if the following condition holds: For each $w\in W$ and each $\alpha\in \EE_w$, one has
\begin{definitionlist}
\item $\langle \chi_{\alpha},\alpha^\vee \rangle >0$,
\item $\langle \chi_{\alpha},\beta^\vee \rangle =0$ for all $\beta\in \EE_{w}\setminus \{\alpha\}$,
\item for all $\beta\in E_w\setminus \EE_w$ which is connected to an element of $\EE_w$, we have $\langle \chi_{\alpha},\beta^\vee \rangle = 0$.
\end{definitionlist}
\end{definition}
Let $\EE$ be a separating system and $w\in W$. Let $\Gamma_{w}\colonequals\sum_{\alpha\in \EE_w}\ZZ_{\geq 0}\chi_{\alpha}$ be the cone generated by $(\chi_\alpha)_{\alpha\in \EE_w}$. Define the Hasse cone of $\EE$ at $w$ as
\begin{equation}
    C_{\Hasse,w}^{\EE}\colonequals h_w(\Gamma_w).
\end{equation}

\subsubsection{Intersection-sum cones}
Let $(S,\zeta)$ be a pair satisfying Assumption \ref{assume}(a)-(c). We explain a strategy to prove Conjecture \ref{conj-S} for $(S,\zeta)$. We consider the stratification $(\Flag(S)_w)_{w\in W}$ on the flag space of $S$. Let $\EE=(\EE_w, \{\chi_\alpha \}_{\alpha\in \EE_w})_{w\in W}$ be a separating system. Furthermore, write simply $f_{w,\alpha}$ for $f_{w,-w\chi_{\alpha}}$. Put $m_{w,\alpha}\colonequals \langle \chi_{\alpha},\alpha^\vee \rangle$. It is the multiplicity of $f_{w,\alpha}$ along $\overline{\Sbt}_{ws_\alpha}$. Recall that we have $\Ha_{w,\alpha}\colonequals \psi^*(f_{w,\alpha})$ and that it is a section of $\Vcal_{\flag}(h_w(\chi_\alpha))$ over $\Fcal_w$. Write $\lambda_{w,\alpha}\colonequals h_w(\chi_\alpha)$. By slight abuse of notation, we write again $\Ha_{w,\alpha}$ for the pullback of this section by $\zeta$, which is a section over $\Flag(S)_w$. Since $\psi$ and $\zeta$ are smooth, the multiplicity of $\Ha_{w,\alpha}$ along $\overline{\Flag(S)}_{ws_\alpha}$ does not change, i.e. it is $m_{w,\alpha}$. For a given $w\in W$, fix an integer $N\geq 1$ divisible by all the $m_{w,\alpha}$ for $\alpha\in \EE_w$, and let $N_{\alpha}$ be the integer such that $N=N_{\alpha} m_{w,\alpha}$.

\begin{lemma}\label{lem-g0}
Let $w\in W$ and $g\in H^0(\overline{\Flag(S)}_w,\Vcal_{\flag}(\lambda))$. For $\alpha\in \EE_w$, let $m_\alpha(g)\geq 0$ be the multiplicity of $\div(g)$ along $\overline{\Flag(S)}_{ws_\alpha}$. Define
\begin{equation}
    g_0\colonequals g^N \prod_{\alpha\in \EE_w}\Ha_{w,\alpha}^{-N_{\alpha} m_\alpha(g)}.
\end{equation}
Then the divisor $\div(g_0)$ has multiplicity $0$ along $\overline{\Flag(S)}_{ws_\alpha}$ for all $\alpha \in \EE_w$. Furthermore, for all $\alpha \in \EE_w$, the restriction of $g_0$ to $\overline{\Flag(S)}_{ws_\alpha}$ is a nonzero element in the space $H^0(\overline{\Flag(S)}_{ws_\alpha}, \Vcal_{\flag}(\lambda_0))$, where $\lambda_0\colonequals N\lambda-\sum_{\beta\in \EE_w} N_{\alpha} m_{\alpha}(g)\lambda_{w,\alpha}$.
\end{lemma}
\begin{proof}
By construction, it is clear that $g_0$ does not have any poles along $\overline{\Flag(S)}_{ws_\beta}$ for all $\beta\in E_w$ which is connected to an element of $\EE_w$. Moreover, $g_0$ does not have a zero along $\overline{\Flag(S)}_{ws_\alpha}$ for all $\alpha \in \EE_w$. Therefore, the restriction of $g_0$ to $\overline{\Flag(S)}_{ws_\alpha}$ is regular on an open subset of codimension $\geq 2$ in $\overline{\Flag(S)}_{ws_\alpha}$. Since $\overline{\Flag(S)}_{ws_\alpha}$ is normal, the section extends to $\overline{\Flag(S)}_{ws_\alpha}$. The result follows.
\end{proof}

Define the cone $C_{S,w}$ by
\begin{equation}
    C_{S,w}\colonequals \{\lambda\in X^*(T) \mid H^0(\overline{\Flag(S)}_{w},\Vcal_{\flag}(\lambda))\neq 0 \}.
\end{equation}
For the longest element $w_0$, note that we have an equality $C_{S,w_0}=C_S$ (where $C_S$ was defined in equation \eqref{CS-def}). For $w\in W$, define the intersection-sum cone of $w$ (with respect to $\EE$) as follows:

\begin{definition}\label{def-intersumcone}
For $\ell(w)=1$, set $C^{+,\EE}_{w}\colonequals C_{\Hasse,w}$. For $\ell(w)\geq 2$, define inductively
\begin{equation}
 C^{+,\EE}_{w}\colonequals C^{\EE}_{\Hasse,w}+ \bigcap_{\alpha\in \EE_w} C^{+,\EE}_{ws_\alpha}.
\end{equation}
\end{definition}
In the case $\EE_w=\emptyset$, we define by convention $\bigcap_{\alpha\in \EE_w} C^{+,\EE}_{ws_\alpha}=X^*(T)$.
\begin{theorem}\label{thm-sep-syst}
Let $\EE$ be a separating system. For each $w\in W$, we have
\begin{equation}
 C_{S,w} \subset \langle C^{+,\EE}_{w} \rangle.
\end{equation}
\end{theorem}

\begin{proof}
We prove the result by induction on $\ell(w)$. When $\ell(w)=1$, the statement is \cite[Proposition 3.2.1]{Goldring-Koskivirta-global-sections-compositio}. Assume $\ell(w)\geq 2$. Let $g\in H^0(\overline{\Flag(S)}_w,\Vcal_{\flag}(\lambda))$ be a nonzero section. By Lemma \ref{lem-g0}, the character $N\lambda-\sum_{\beta\in \EE_w} N_\alpha m_{\alpha}(g)\lambda_{w,\alpha}$ lies in $\bigcap_{\alpha\in \EE_w} C_{S,ws_\alpha}$. By induction, we have $C_{S,ws_\alpha}\subset \langle C^{+,\EE}_{ws_\alpha}\rangle$. Hence $N\lambda\in C^{+,\EE}_{w}$, which implies $\lambda\in \langle C^{+,\EE}_{w}\rangle$.
\end{proof}

Assumption \ref{assume}(d) is only used in the proof of Theorem \ref{thm-sep-syst} for $\ell(w)=1$. This is why we do not make any assumption regarding pseudo-completeness of strata of higher length. Taking $w=w_0$ in Theorem \ref{thm-sep-syst}, we deduce that $ C_{S} \subset \langle C_{w_0}^{+,\EE}\rangle$. Hence,
\begin{equation}
    C_{S} \subset \bigcap_{\EE}\langle C_{w_0}^{+,\EE}\rangle,
\end{equation}
where the intersection is taken over all separating systems. Recall also that $\langle C_{\zip}\rangle \subset \langle C_S \rangle$. In particular, if we exhibit a separating system $\EE$ such that $\langle C_{w_0}^{+,\EE}\rangle=\langle C_{\zip}\rangle$, then we deduce that Conjecture \ref{conj-S} holds for all schemes $S$ satisfying Assumption \ref{assume}(a)-(c). Note that the cones $\langle C_{w_0}^{+,\EE}\rangle$ and $\langle C_{\zip}\rangle$ are entirely objects in the realm of group theory, a priori unrelated to the scheme $S$. This turns Conjecture \ref{conj-S} into a group theory problem.
\begin{question}
Does there always exist a separating system $\EE$ such that $\langle C_{w_0}^{+,\EE}\rangle=\langle C_{\zip}\rangle$?
\end{question}
We exhibit suitable separating systems $\EE$ for several pairs $(G,\mu)$ which arise from Shimura varieties.

\subsection{Divisibility}\label{sec-divis}
We retain all notations from the previous section. For $\lambda\in X^*(T)$ and $f\in H^0(S,\Vcal_I(\lambda))$, we denote by $f_{\flag}$ the corresponding section
\begin{equation}
    f_{\flag}\in H^0(\Flag(S),\Vcal_{\flag}(\lambda))
\end{equation}
via the identification \eqref{identif-lambda}. Since $\Vcal_{\flag}(\lambda)$ is a line bundle on $\Flag(S)$, we can make sense of divisibility of sections.

\begin{definition}\label{divis-def}
Let $f\in H^0(S,\Vcal_I(\lambda))$ and $f'\in H^0(S,\Vcal_I(\lambda'))$ where $\lambda,\lambda'\in X^*(T)$. We say that $f'$ is divisible by $f$ if $f'_{\flag}$ is divisible by $f_{\flag}$.
\end{definition}

For $\lambda\in X^*(T)$, we say that a subset $V\subset X^*(T)_{\RR}$ is a neighborhood cone of $\lambda$ if $V$ is a $\RR_{\geq 0}$-subcone of $X^*(T)_{\RR}$ which is a neighborhood of $\lambda$ for the real topology. By slight abuse of terminology, we also call $V\cap X^*(T)$ a neighborhood cone of $\lambda$ in $X^*(T)$. 

\begin{definition}\label{extremal-def}
Let $f\in H^0(\GZip^\mu,\Vcal_I(\lambda))$ be nonzero. We say that $f$ is an isolated form (for $S$) if there exists a neighborhood cone $V$ of $\lambda$ such that any form $H^0(S,\Vcal_I(\lambda'))$ with $\lambda'\in V$ is divisible by $f$. We call $V$ a neighborhood of divisibility.
\end{definition}
\begin{exa}[Hilbert modular varieties]
Consider the case of Hilbert modular varieties. Theorem \ref{thmDK} of Diamond--Kassaei shows that the partial Hasse invariants $\Ha_\tau$ are isolated sections, when the condition $pk_\tau > k_{\sigma^{-1}\tau}$ defines a neighborhood of $\mathbf{h}_\tau$, which is the case exactly when $\tau \neq \sigma^{-1}\tau$. 

\end{exa}

If $f\in H^0(\GZip^\mu,\Vcal_I(\lambda))$ is isolated, we expect its weight $\lambda$ to generate an extremal ray of the cone $\langle C_{\zip} \rangle$. This is indeed the case in all our examples. However, not all forms whose weight generates an extremal ray of the cone are isolated. In this paper, we only consider the case when $f$ is a partial Hasse invariant $\Ha_\alpha$ for some root $\alpha\in \Delta$. In this case, all the examples of this article show evidence for the following:

\begin{expect}\label{expect}
Let $\alpha\in \Delta$ be a simple root, $\Ha_\alpha$ a corresponding partial Hasse invariant, and $\lambda_\alpha$ the weight of $\Ha_\alpha$ (see \eqref{lamalph-eq}). Assume the following:
\begin{Alist}
    \item The weight $\lambda_\alpha$ generates an extremal ray of the cone $\langle C_{\zip} \rangle$.
    \item $\lambda_{\alpha}$ lies in the complement of $C_{\GS}$.
\end{Alist}
Then $\Ha_\alpha$ is an isolated section for any scheme $(S,\zeta)$ satisfying Assumption \ref{assume}.
\end{expect}
The necessity of condition (A) can be seen in the case of $\Sp(6)$, we discuss it at the end of \S \ref{sec-sympn3} (see equation \eqref{non-extr-phi}). The necessity for condition (B) can be seen in the case of Hilbert--Blumenthal Shimura varieties: As we noted above, the partial Hasse invariant $\Ha_\tau$ is isolated if and only if $\tau\neq \sigma^{-1}\tau$, which is indeed equivalent to $\Ha_\tau \notin C_{\GS}$. Similarly, in all other cases treated in this paper, Condition (B) is required.

More generally than the Hilbert--Blumenthal case, assume that $(G,\mu)$ is of Hasse-type (see \S \ref{sec-hassetype}). Then Condition (A) is always satisfied for any root $\alpha\in \Delta$, because $\langle C_{\Hasse} \rangle = \langle C_{\zip} \rangle$. As for Condition (B), it is satisfied for roots $\alpha\in \Delta\setminus I$ if and only if $\sigma^{-1}(\alpha)\neq \alpha$, similarly to the Hilbert--Blumenthal case. For roots $\alpha \in I$, Condition (B) is almost always satisfied, except in some trivial cases when $\alpha$ is a root of a direct factor of $G^{\ad}$ contained in $L^{\ad}$. Hence, if the above expectation is correct, one should have several divisibility results for orthogonal Shimura varieties of type $B_n$, namely one for each of the $n-1$ simple roots of the maximal Levi $L$.

\section{Symplectic groups}\label{sec-symp}

In this section, we consider a symplectic group $\Sp(2n)$, endowed with its usual Siegel-type zip datum (see below). Note that the Siegel-type Shimura variety $\Acal_n$ is associated to the reductive group $\GSp(2n)$ rather than $\Sp(2n)$. However, the stacks of $G$-zips for both groups are closely related, and all results of this section hold in both cases.

\subsection{Group theory}\label{zipcone-sp6}

We first give some notations for an arbitrary symplectic group. Let $(V_0,\psi)$ be a non-degenerate symplectic space over $\FF_q$ of dimension $2n$, for some integer $n\geq 1$. After choosing an appropriate basis $\Bcal$ for $V_0$, we assume that $\psi$ is given by the matrix
\begin{equation}
\begin{pmatrix}
& -J \\
J&\end{pmatrix}
\quad \textrm{ where } \quad
J\colonequals \begin{psmallmatrix}
&&1 \\
&\iddots& \\
1&&
\end{psmallmatrix}.
\end{equation} 
Define $G$ as follows:
\begin{equation}\label{group}
G(R) = \{f\in \GL_{\FF_{q}}(V_0\otimes_{\FF} R) \mid  \psi_R(f(x),f(y))=\psi_R(x,y), \ \forall x,y\in V_0\otimes_{\FF_q} R \}
\end{equation}
for all $\FF_q$-algebras $R$. Identify $V_0=\FF_q^{2n}$ via $\Bcal$ and view $G$ as a subgroup of $\GL_{2n,\FF_q}$. Fix the $\FF_q$-split maximal torus $T$ given by diagonal matrices in $G$, i.e.
\begin{equation}
T(R)\colonequals \{ \diag_{2n}(x_1,\ldots ,x_n,x^{-1}_n,\ldots,x^{-1}_1) \mid x_1, \ldots ,x_n\in R^\times \}.
\end{equation}
Define $B$ as the Borel subgroup of $G$ consisting of the lower-triangular matrices in $G$. For a tuple $(a_1,\dots,a_n)\in \ZZ^n$, define a character of $T$ by mapping $\diag_{2n}(x_1, \dots ,x_n,x^{-1}_n, \dots ,x^{-1}_1)$ to $x_1^{a_1} \cdots x_n^{a_n}$. From this, we obtain an identification $X^*(T) = \ZZ^n$. Denoting by $(e_1,\dots ,e_n)$ the standard basis of $\ZZ^n$, the $T$-roots of $G$ and the $B$-positive roots are respectively
\begin{align}
\Phi&\colonequals \{\pm e_i \pm e_j \mid 1\leq i \neq j \leq n\} \cup \{ \pm 2e_i \mid 1\leq i \leq n \}, \\
\Phi_+&\colonequals \{e_i \pm e_j \mid 1\leq i< j \leq n\} \cup \{ 2e_i \mid 1\leq i \leq n\} 
\end{align}
and the $B$-simple roots are $\Delta\colonequals \{\alpha_1,\dots , \alpha_{n-1},\beta\}$ where 
\begin{align*}
\alpha_i&\colonequals e_{i}-e_{i+1} \textrm{ for } i=1,...,n-1 ,\\ \beta&\colonequals 2e_n.
\end{align*}
The Weyl group $W\colonequals W(G,T)$ can be identified with the group of permutations $\sigma \in \Sfr_{2n}$ satisfying $\sigma(i)+\sigma(2n+1-i)=2n+1$ for all $1\leq i \leq 2n$. In particular, $\sigma$ is completely determined by the values $\sigma(i)$ for $1\leq i \leq n$. If $\sigma(i)=a_i$ for all $1\leq i \leq n$, we write 
\begin{equation}
\sigma=[a_1 \cdots a_n].    
\end{equation}
Define a cocharacter $\mu \colon \GG_{\mathrm{m},\FF_q}\to G$ by $z\mapsto \diag(zI_n,z^{-1}I_n)$. Write $\Zcal\colonequals (G,P,L,Q,M,\varphi)$ for the associated zip datum (since $\mu$ is defined over $\FF_q$, we have $M=L$). Concretely, if we denote by $(u_i)_{i=1}^{2n}$ the canonical basis of $k^{2n}$, then $P$ is the stabilizer of $V_{0,P}=\Span_k(u_{n+1},...,u_{2n})$ and $Q$ is the stabilizer of $V_{0,Q}=\Span_k(u_{1},...,u_{n})$. The intersection $L\colonequals P\cap Q$ is a common Levi subgroup and there is an isomorphism $\GL_{n,\FF_q}\to L$, $A\mapsto \delta(A)$, where:
\begin{equation}
\delta(A):=\left(\begin{matrix}\label{deltadef}
A& \\ & J {}^t\! A^{-1} J
\end{matrix} \right).
\end{equation}
Let $\{\chi_\alpha\}_{\alpha\in \Delta}$ be the set of fundamental weights. They lie in $X^*(T)$ and satisfy $\langle \chi_\alpha,\alpha^\vee \rangle=1$ and $\langle \chi_\alpha,\beta^\vee \rangle=0$ for all $\beta\in \Delta\setminus \{\alpha\}$. Denote by $\Ha_\alpha$ the partial Hasse invariant attached to $\chi_\alpha$ as in Definition \ref{part-Hasse-glob}. It is a section over $\GF^\mu$ of $\Vcal_{\flag}(\lambda_\alpha)$ for $\lambda_\alpha=h_{w_0}(\chi_{\alpha})$, with notation as in \S\ref{sec-Hasse-cone}. The vanishing locus of $\Ha_{\alpha}$ is $\overline{\Fcal}_{w_0s_\alpha}$ and the multiplicity of $\Ha_\alpha$ along $\overline{\Fcal}_{w_0s_\alpha}$ is $1$ by Theorem \ref{brion}(3). Concretely, $\lambda_\alpha$ is given as follows. For $\alpha_i$ ($1\leq i\leq n-1$),
\begin{equation}
    \lambda_{\alpha_i} = (\underbrace{1, \dots ,1}_{\textrm{$i$ times}},\underbrace{0, \dots ,0}_{\textrm{$n-i$ times}}) + (\underbrace{0, \dots , 0}_{\textrm{$n-i$ times}}, \underbrace{-q, \dots ,-q}_{\textrm{$i$ times}}),
\end{equation}
and $\lambda_{\beta}=(1-q,\dots , 1-q)$. Note that $\Ha_{\beta}$ is the classical ordinary Hasse invariant. The weights $\{\lambda_\alpha\}_{\alpha\in \Delta}$ generate the Hasse cone $C_{\Hasse}$. Other cones defined in \S\ref{subsec-subcones} are given as follows:
\begin{align*}
    X^*_{+,I}(T) & = \{ (a_1,\dots ,a_n)\in \ZZ^n \mid \  a_1\geq \dots \geq a_n\} \\
    X_{-}^*(L) &= \NN(-1,\dots ,-1) \\
    C_{\GS} &= \{ (a_1,\dots ,a_n)\in X^*_{+,I}(T) \mid \ a_1\leq 0 \} \\
    C_{\rm hw} &= \{ (a_1,\dots ,a_n)\in X^*_{+,I}(T) \mid \  \sum_{i=1}^n q^{n-i}a_i\leq 0 \} \\
    C^{+,I}_{\orb} &= \{ (k_1,\dots,k_n)\in \ZZ^n \ \mid \  \sum_{i=1}^j k_i + \frac{1}{q} \sum_{i=j+1}^{n} k_i \leq 0, \quad j=1,\dots, n-1. \}
\end{align*}
We are not able to determine $C_{\zip}$ or even $\langle C_{\zip}\rangle$ for general $n$.

\subsection{The case $n=2$}
In this case, Conjecture \ref{conj-S} was proved in \cite[Theorem 5.1.1]{Goldring-Koskivirta-global-sections-compositio} (note that in \cite{Goldring-Koskivirta-global-sections-compositio}, we consider the group $G=\Sp(4)_{\FF_p}$, but the same proof applies if we replace $p$ by $q$). We explain here that the proof of \loccit also yields a divisibility result by partial Hasse invariants. First, in this case one has
\begin{equation}
    \langle C_{\zip} \rangle = \langle C_{\Hasse} \rangle = C_{\hw} = \{(a_1,a_2)\in X_{+,I}^*(T) \mid qa_1+a_2 \leq 0\}.
\end{equation}
Let $(S,\zeta)$ be a pair satisfying Assumption \ref{assume}. There are two strata of codimension $1$ in $\Flag(S)$, namely $\Flag(S)_{w}$ for $w$ in the set
\begin{equation}
    \{w_0 s_{\alpha_1}, w_0 s_{\beta}\}=\{[34],[42]\}.
\end{equation}

\begin{proposition}\label{prop-Sp4-van} Assume that $\lambda=(a_1,a_2)\in \ZZ^2$ satisfies $a_1 > 0$ or $a_2>0$. Then 
\begin{equation}
   H^0(\overline{\Flag(S)}_{[34]},\Vcal_{\flag}(\lambda))=0.
\end{equation}
\end{proposition}

\begin{proof}
By the proof of \cite[Theorem 5.1.1]{Goldring-Koskivirta-global-sections-compositio}, one has $C_{S,[34]}\subset \langle C_{\Hasse,[34]} \rangle$ (note that the cone $\langle C_{\Hasse,w} \rangle$ is denoted by $C_{\Sbt,w}$ in \loccitn). By \cite[Figure 1]{Goldring-Koskivirta-global-sections-compositio}, this cone is the subset of $\ZZ^2$ defined by $a_1\leq 0$ and $a_2\leq 0$. The result follows.
\end{proof}

Recall that $\Ha_{\alpha}$ (for $\alpha \in \Delta$) denotes the partial Hasse invariant (with multiplicity $1$) with respect to $\alpha$. The weight of $\Ha_{\alpha_1}$ is $\lambda_{\alpha_1}= (1,-q)$ and the weight of $\Ha_{\beta}$ is $\lambda_{\beta}=(1-q,1-q)$.

\begin{theorem}\label{thm-Sp4-div}
Let $f\in H^0(S,\Vcal_I(\lambda))$ for $\lambda=(a_1,a_2)\in \ZZ^3$. If $a_1 > 0$, then $f$ is divisible by the partial Hasse invariant $\Ha_{\alpha_1}$.
\end{theorem}

\begin{proof}
Assume that $a_1>0$. By Proposition \ref{prop-Sp4-van}, the restriction of $f$ to the stratum $\overline{\Flag(S)}_{[34]} = \overline{\Flag(S)}_{w_0s_{\alpha_1}}$ is zero. Since $\Ha_{\alpha_1}$ cuts out $\overline{\Flag(S)}_{w_0 s_{\alpha_1}}$ with multiplicity one, we deduce that $f$ is divisible by $\Ha_{\alpha_1}$.
\end{proof}

We could also state a similar result when $a_2>0$, but there exist no nonzero global sections of weight $(a_1,a_2)$ with $a_2>0$ (see Figure \ref{fig:Sp4} below), hence this result would be empty. We illustrate Theorem \ref{thm-GL21-div} graphically. The area colored in grey in the figure below corresponds to the subset of $\langle C_{\zip} \rangle$ where $a_1>0$. By Theorem \ref{thm-GL21-div}, any section whose weight lies in this area is divisible by $\Ha_{\alpha_{1}}$.

\begin{figure}[H]
    \centering
    \includegraphics[width=13cm]{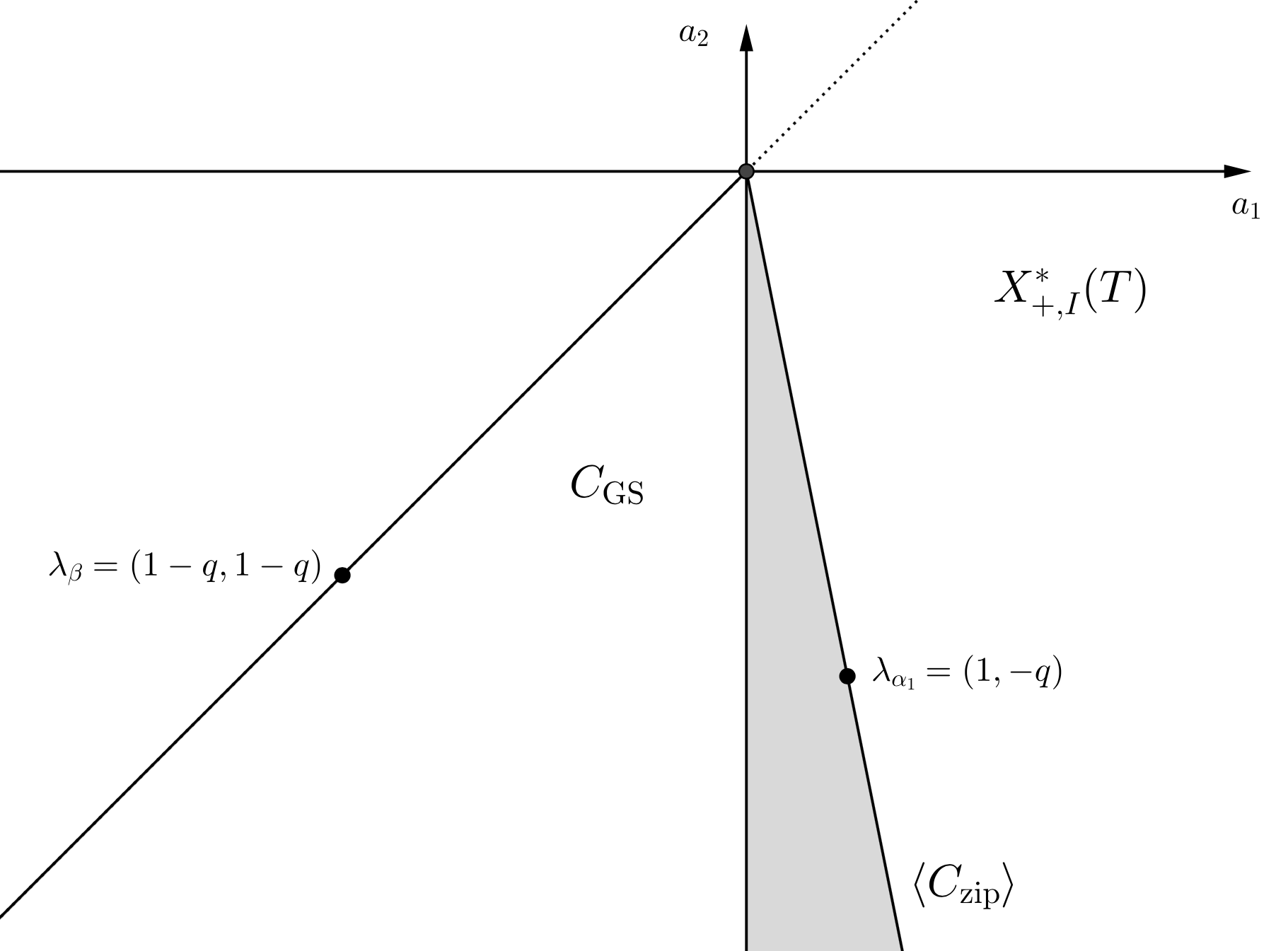}
    \caption{The case of $\Sp(4)_{\FF_q}$}
    \label{fig:Sp4}
\end{figure}

This shows that $\Ha_{\alpha_1}$ is an isolated form (Definition \ref{extremal-def}) and $a_1>0$ defines a neighborhood of divisibility.

\subsection{The case $n=3$} \label{sec-sympn3}

We now assume that $n=3$. We recall previous results from \cite[\S 5.5]{Koskivirta-automforms-GZip}.

\begin{proposition}[{\cite[\S 5.5]{Koskivirta-automforms-GZip}}]\label{prop-Sp6}
We have
$$ \langle C_{\zip} \rangle = \{ (a_1,a_2,a_3)\in X^*_{+,I}(T) \mid \ q^2 a_1+a_2+qa_3\leq 0 \ \textnormal{ and } \  qa_1+q^2 a_2+a_3\leq 0 \}.$$
\end{proposition}

The Hasse cone $C_{\Hasse}$ is generated by the three weights $\lambda_{\alpha_1}=(1,0,-q)$, $\lambda_{\alpha_2}=(1,1-q,-q)$ and $\lambda_{\beta}=(1-q,1-q,1-q)$. As in the case of $\Sp(4)$, we will focus our attention on $\lambda_{\alpha_1}$. On the other hand, the cone $C_{\hw}$ is defined inside $X_{+,I}^*(T)$ by the equation $q^2a_1+qa_2+a_3\leq 0$. Contrary to the case $n=2$, the cone $C_{\hw}$ is strictly contained in $\langle C_{\zip} \rangle$. Note that $\lambda_{\alpha_1}$ lies in the complement of $C_{\hw}$. The cone $C_{\hw}$ is generated (over $\QQ_{\geq 0}$) by the three weights $\lambda_{\beta}=(1-q,1-q,1-q)$, $\eta_1=(1,1,-(q^2+q))$ and $\eta_2=(q+1,-q^2,-q^2)$. The last two are the weights of the forms $h_1=\Norm_{L_{\varphi}}(f_{\eta_1,\high})$ and $h_2=\Norm_{L_\varphi}(f_{\eta_2,\high})$, where the notation $\Norm_{L_{\varphi}}$ was explained in \S\ref{subsec-subcones}, and where $f_{\eta,\high}$ denotes the highest weight vector of the $L$-representation $V_I(\eta)$.

It is helpful to visualize the different cones on a diagram. We represent a two-dimensional generic "slice" of the three-dimensional subcones of $\ZZ^3$. Therefore, a line passing through the origin appears as a point. In Figure \ref{fig:sp6} below, the two enclosing half-lines correspond respectively to the hyperplanes $a_1=a_2$ and $a_2=a_3$, which form the boundary of $X_{+,I}^*(T)$. The cones $C_{\GS}\subset C_{\rm hw}\subset \langle C_{\zip}\rangle$ are represented on the figure. We colored in grey the complement of $C_{\hw}$ inside $\langle C_{\zip} \rangle$. We explain the significance of this subset below the figure

\begin{figure}[H]
    \centering
    \includegraphics[width=13cm]{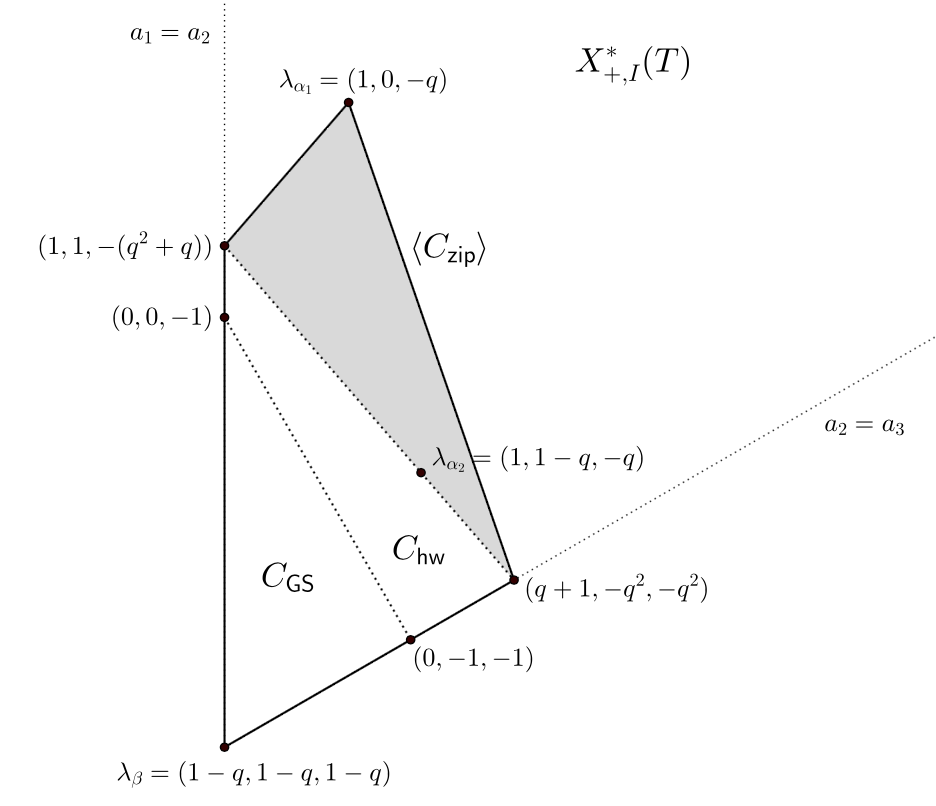}
    \caption{The case of $\Sp(6)_{\FF_q}$}
    \label{fig:sp6}
\end{figure}

One sees immediately on the figure that there are four extremal rays, generated by the weights $\lambda_{\alpha_1}=(1,0,-q), \lambda_\beta=(1-q,1-q,1-q), \eta_1$ and $\eta_2$. In particular, $\langle C_{\zip} \rangle$ is spanned (over $\QQ_{\geq 0}$) by $C_{\hw}$ and $\lambda_{\alpha_1}=(1,0,-q)$. This observation will be crucial in the proof of our main theorem.

We briefly explain the importance of the area colored in grey, foreshadowing the main result of this section (Theorem \ref{thm-divSp6}). By definition, the grey area is the set of $\lambda=(a_1,a_2,a_3)\in \langle C_{\zip} \rangle$ such that $q^2a_1+qa_2+a_3>0$. We will prove that if $(S,\zeta)$ is a scheme satisfying Assumption \ref{assume}, and $f$ is a section over $S$ whose weight $\lambda$ lies in the grey area, then $f$ is divisible by the partial Hasse invariant $\Ha_{\alpha_1}$ of weight $(1,0,-q)$, in the sense of Definition \ref{divis-def}. This applies for example to the Siegel-type Shimura variety $\Acal_3$ (as mentioned earlier, we need to change the group to $\GSp(6)$, but this change does not affect the result). As explained in Expectation \ref{expect}, the general philosophy seems to be that forms whose weight lies "far away" from the cone $C_{\GS}$ tend to be divisible by appropriate partial Hasse invariants.

In the case of $\Sp(6)$, this divisibility result only holds for $\Ha_{\alpha_1}$. We cannot expect a similar divisibility result for the partial Hasse invariant $\Ha_{\alpha_2}$ of weight $\lambda_{\alpha_2}=(1,1-q,-q)$, because the weight $(1,1-q,-q)$ lies in the interior of the cone $\langle C_{\zip} \rangle$. Indeed, since the weights of $\Ha_{\alpha_1}$, $\Ha_{\beta}$, $h_1$ and $h_2$ generate the cone $C_{\zip}$ over $\QQ_{\geq 0}$, we can choose integers $a,b,c,d\geq 0$ appropriately so that the weight of the section 
\begin{equation} \label{non-extr-phi}
    f=\Ha_{\alpha_1}^a \Ha_{\beta}^b h_1^c h_2^d
\end{equation}
is a positive multiple of $\lambda_{\alpha_2}=(1,1-q,-q)$. However, $f$ is not divisible by $\Ha_{\alpha_2}$, which shows that $\Ha_{\alpha_2}$ is not an isolated section in the sense of Definition \ref{extremal-def}. This shows why Condition (A) of Expectation \ref{expect} is necessary. On the other hand, other candidates for isolated sections are the forms $h_1$ and $h_2$, since their weights $\eta_1$ and $\eta_2$ generate extremal rays. Their may exist neighborhoods $V(\eta_1)$ and $V(\eta_2)$ where we have divisibility by these sections. It would be interesting to investigate such divisibility results by more general sections, beyond the case of partial Hasse invariants. One can show that $h_1$ and $h_2$ are indeed isolated for $\GZip^\mu$. However, we do not know whether they stay isolated for any scheme $S\to \GZip^\mu$ satisfying Assumption \ref{assume}, in particular for $S=\Acal_3$.

\subsection{Main theorem}\label{sec-mainSp6}
We continue to assume $n=3$, i.e $G=\Sp(6)_{\FF_q}$. We consider a pair $(S,\zeta)$ satisfying Assumption \ref{assume}, and we let $(\Flag(S),\zeta_{\flag})$ be the flag space of $(S,\zeta)$, as defined in \S\ref{sec-flagspace}. By the parametrization $w\mapsto \Flag(S)_w$, there are three codimension one strata, corresponding to the elements 
\begin{equation}
    \{w_0 s_{\alpha_1}, w_0 s_{\alpha_2}, w_0 s_{\beta}\}=\{[564],[645],[653]\}.
\end{equation}
We show a result regarding the codimension one stratum $\Flag(S)_{[564]}$, where $[564]=w_0s_{\alpha_1}$. Recall also that $\overline{\Flag(S)}_{[564]}$ is the vanishing locus of the partial Hasse invariant $\Ha_{\alpha_1}$ (pulled back to $\Flag(S)$ via $\zeta_{\flag}$).

\begin{proposition}\label{prop-564}
Assume $q\geq 5$. Let $\lambda=(a_1,a_2,a_3)$ such that $q^2a_1+qa_2+a_3>0$. Then one has $ H^0(\overline{\Flag(S)}_{[564]}, \Vcal_{\flag}(\lambda))=0$.
\end{proposition}
To prove this result, we implement the strategy explained in \S\ref{subsec-sum-inter}. We will exhibit a suitable separating system $\EE=(\EE_w,\{\chi_\alpha\}_{\alpha\in \EE_w})_{w\in W}$. Only certain strata $w\in W$ will be relevant in the proof. To simplify the notation, we write $C_w^{+}$ for $C_w^{\EE,+}$ (since $\EE$ will be fixed once and for all). For starter, we show in the diagram below the relevant strata that will appear in the proof.

\bigskip

\begin{figure}[H]

\hspace{-1cm} $\xymatrix@R=10pt{
&&&*++[F]{[145]}\ar@{-}[r]&*++[F]{[154]}\ar@{-}[rd]&&&& \\
&   *++[F]{[132]}\ar@{-}[r]\ar@{-}[rd] &*++[F]{[135]} \ar@{-}[r]\ar@{-}[rd]\ar@{-}[ru]&*++[F]{[153]}\ar@{-}[rd]\ar@{-}[ru]&*++[F]{[246]}\ar@{-}[r]&*++[F]{[264]}\ar@{-}[rd]&*++[F]{[541]}\ar@{-}[r]&*++[F]{[546]}\ar@{-}[rd]&\\
  *++[F]{[123]}\ar@{-}[r]\ar@{-}[ru]\ar@{-}[rd] & *++[F]{[124]}\ar@{-}[rd]\ar@{-}[r]\ar@{-}[ru] &*++[F]{[142]}\ar@{-}[ru]&*++[F]{[236]}\ar@{-}[ru]&*++[F]{[263]}\ar@{-}[rd]&*++[F]{[531]}\ar@{-}[ru]&*++[F]{[365]}\ar@{-}[r]&*++[F]{[465]}\ar@{-}[r]&*++[F]{[564]}\\
&   *++[F]{[213]}\ar@{-}[r] &*++[F]{[214]} \ar@{-}[r]\ar@{-}[ru]\ar@{-}[rd]&*++[F]{[412]}\ar@{-}[r]&*++[F]{[421]}\ar@{-}[ru]&*++[F]{[362]}\ar@{-}[ru]&&& \\
&&&*++[F]{[315]}\ar@{-}[r]&*++[F]{[326]}\ar@{-}[ru]&&&& \\
}$
    \caption{The strata appearing in the proof for $G=\Sp(6)$}
    \label{fig:sepsyst}
\end{figure}
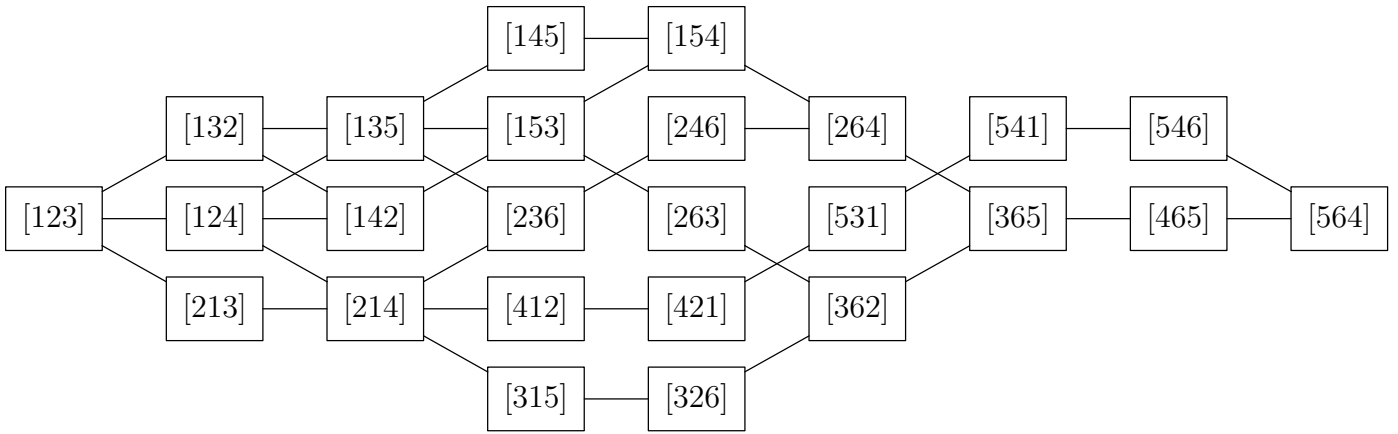
For a stratum $w$ appearing in the diagram above, we will define a subset $\EE_w\subset E_w$ and characters $\{\chi_\alpha\}_{\alpha\in \EE_w}$ satisfying Definition \ref{def-sep-syst}. We will denote by $\LL_w$ the set $\{ws_\alpha\}_{\alpha\in \EE_w}$ of lower neighbors of $w$ corresponding to $\EE_w$. When $w'\in \LL_w$, we have joined by a segment the strata $w$ and $w'$ in the above diagram (note that $\EE_w$ may by strictly smaller than $E_w$). For strata not appearing in the diagram, we set $\EE_w=\emptyset$. In the case $G=\Sp(6)_{\FF_q}$, there are strata which do not admit a full separating system of partial Hasse invariants (Definition \ref{def-Q-sep}). However, all strata in the above diagram do admit such a system. We prove Proposition \ref{prop-564} in \S\ref{Sp6-proof}. Let us here only make the strategy explicit and explain that it suffices to show the Lemma below:

\begin{lemma}\label{lem-564}
For $q\geq 5$, one has $C_{[564]}^{+}\subset C_{\hw}$.
\end{lemma}
Indeed, assume that Lemma \ref{lem-564} holds. Then, we deduce from Theorem \ref{thm-sep-syst} that $C_{S,[564]}\subset \langle C_{[564]}^{+}\rangle \subset C_{\hw}$. Since $C_{\hw}$ is precisely the set of $(a_1,a_2,a_3)\in X_{+,I}^*(T)$ such that $q^2a_1+qa_2+a_3\leq 0$, we deduce Proposition \ref{prop-564}. The proof of Lemma \ref{lem-564} is entirely computational, and is based on the recursive determination of the cones $C_w^{+}$ for all $w$ appearing in the above diagram, starting at elements of length $1$ and ending at the element $w=[564]$. More precisely, it is sufficient to give a suitable upper bound for each $C_{w}^+$ rather than determining it explicitly. This is the strategy we implement in \S\ref{Sp6-proof}.

We derive some immediate consequences of Proposition \ref{prop-564}. Let $\Ha_{\alpha_1}$ be a partial Hasse invariant for $\alpha_1$ on $\GF^\mu$ which has multiplicity one along $\overline{\Fcal}_{w_0s_{\alpha_1}}$. Since $\zeta_{\flag}$ is smooth, the pull back of $\Ha_{\alpha_1}$ via $\zeta_{\flag}\colon \Flag(S)\to \GF^\mu$ has also multiplicity $1$ along $\overline{\Flag(S)}_{w_0s_{\alpha_1}}$.

\begin{theorem}\label{thm-divSp6}
Assume $q\geq 5$. Let $f\in H^0(\Flag(S),\Vcal_{\flag}(\lambda))$ and assume that $\lambda=(a_1,a_2,a_3)\in \ZZ^3$ satisfies $q^2a_1+qa_2+a_3>0$. Then $f$ is divisible by the partial Hasse invariant $\Ha_{\alpha_1}$.
\end{theorem}

\begin{proof}
By Proposition \ref{prop-564}, the restriction of $f$ to the stratum $\overline{\Flag(S)}_{[564]} = \overline{\Flag(S)}_{w_0s_{\alpha_1}}$ is zero. Since the partial Hasse invariant $\Ha_{\alpha_1}$ cuts out $\overline{\Flag(S)}_{w_0s_{\alpha_1}}$ with multiplicity $1$ and $S$ is smooth, we deduce that $f$ is divisible by $\Ha_{\alpha_1}$.
\end{proof}

\begin{theorem}\label{thm-conjSp6}
Assume $q\geq 5$. Conjecture \ref{conj-S} holds in the case $G=\Sp(6)_{\FF_q}$ and $\mu\colon \GG_{\textrm{m},\FF_q}\to G$ defined as in \S\ref{zipcone-sp6}.
\end{theorem}

\begin{proof}
By Lemma \ref{lem-564}, we have inclusions: 
\begin{equation}
C^{+}_{w_0}\subset C_{\Hasse}+C^{+}_{[564]}\subset C_{\Hasse}+C_{\hw}\subset \langle C_{\zip}\rangle.    
\end{equation}
By Theorem \ref{thm-sep-syst}, we have $ C_S \subset \langle C^{+}_{w_0} \rangle$, hence $C_S \subset  \langle C_{\zip}\rangle$, and therefore also $\langle C_S \rangle \subset \langle C_{\zip}\rangle$. Since the converse inclusion is always satisfied, the result follows.
\end{proof}

\section{Groups of type $A_n$}
We prove conjecture \ref{conj-S} for several Shimura varieties attached to unitary groups $\mathbf{G}:=\GU(r,s)$ associated with a totally imaginary quadratic field $\mathbf{E}/\QQ$, and where $n\colonequals r+s\leq 4$. We also obtain divisibility results.

\subsection{Unitary Shimura varieties} \label{sec-unit-Shim}
We consider Shimura varieties attached to unitary groups. Let $\mathbf{E}/\QQ$ be a totally imaginary quadratic extension, and $(\mathbf{V},\psi)$ be a hermitian space over $\mathbf{E}$. We assume that there is a basis $\Bcal$ in which $\psi$ is given by the matrix:
\begin{equation}
    \begin{pmatrix}
    &&1\\&\iddots&\\1&&
    \end{pmatrix}
\end{equation}
Let $\mathbf{G}=\GU(\mathbf{V},\psi)$ be the general unitary group of $(\mathbf{V},\psi)$. Furthermore, assume that $\psi_\RR$ has signature $(r,s)$ where $r,s$ are nonnegative integers such that $r+s=n$. We let $\Lambda\subset \mathbf{V}\otimes_{\QQ} \QQ_p$ be the $\Ocal_{\mathbf{E}}$-invariant $\ZZ_p$-lattice generated by the elements of $\Bcal$. This yields a reductive $\ZZ_p$-model $\Gcal_{\ZZ_p}=\GU(\Lambda,\psi)$ of $\mathbf{G}_{\QQ_p}$. In particular, the group $K_p=\Gcal_{\ZZ_p}(\ZZ_p)$ is hyperspecial.

By \cite{Kottwitz-points-shimura-varieties}, for each open compact subgroup $K^p\subset \mathbf{G}(\AA^p_f)$, there is a PEL-type Shimura variety $\Sscr_{K}$ over $\Ocal_{\mathbf{E}_v}$ attached to this group, where $K=K^pK_p$. We are interested in the special fiber $S_K=\Sscr_K\otimes_{\Ocal_{\mathbf{E}_v}} k$. Write $G$ for the special fiber of $\Gcal_{\ZZ_p}$. We have a map $\zeta\colon S_K\to \GZip^{\mu}$, where $\mu$ is naturally attached to the Shimura datum. We are naturally led to consider separately the following two cases:
\begin{enumerate}[(1)]
    \item If $p$ is split in $\mathbf{E}$, then $G$ is isomorphic to $\GL_{n,\FF_p}\times \GG_{\textrm{m},\FF_p}$. For simplicity, we will instead work with the group $G=\GL_n$.
    \item If $p$ is inert in $\mathbf{E}$, then $G$ is a general unitary group $\GU(n)$ over $\FF_p$. For simplicity, we will instead work with the group $G=\U(n)$.
\end{enumerate}

\subsection{The case $G=\GL_{n,\FF_q}$} \label{sec-GLn}

\subsubsection{Group theory}\label{sec-cones-GLn}

Set $G=\GL_{n,\FF_q}$ (as usual, we take $q=p$ in the context of Shimura varieties). Define a cocharacter $\mu \colon \GG_{\mathrm{m},k}\to G_k$ by $\mu(x)=\diag(xI_r,I_s)$ with $r+s=n$. Write $\Zcal_\mu=(G,P,L,Q,M,\varphi)$ for the attached zip datum. If $(u_1,\dots ,u_n)$ denotes the canonical basis of $k^n$, then $P$ is the stabilizer of $V_P\colonequals \Span_k(u_{r+1},\dots , u_n)$ and $Q$ is the stabilizer of $V_Q\colonequals \Span_k(u_{1},\dots , u_r)$. Let $B$ denote the lower-triangular Borel and $T$ the diagonal torus. The Levi subgroup $L=P\cap Q$ is isomorphic to $\GL_{r,\FF_q}\times \GL_{s,\FF_q}$. Identify $X^*(T)=\ZZ^n$ such that $(a_1,\dots ,a_n)\in \ZZ^n$ corresponds to the character $\diag(x_1,\dots ,x_n)\mapsto \prod_{i=1}^n x_i^{a_i}$. The simple roots with respect to $B$ are $\{\alpha_i\}_{1\leq i \leq n-1}$ where
\begin{equation}
    \alpha_i= e_i-e_{i+1}
\end{equation}
and $(e_i)_{1\leq i \leq n}$ denotes the canonical basis of $\ZZ^n$.
For general $(r,s)$, we do not know a description of $C_{\zip}$ or even $\langle C_{\zip} \rangle$. However, one can easily compute the cones $C_{\Hasse}$ and $C_{\hw}$ as first approximations (see \S \ref{subsec-subcones}). First, the cones $X_{+,I}^*(T)$ and $C_{\GS}$ are
\begin{align*}
     X_{+,I}^*(T) &= \{(a_1,\dots,a_n)\in \ZZ^n \mid a_1\geq \dots \geq a_r \ \textrm{and} \ a_{r+1}\geq \dots \geq a_n \} \\
     C_{\GS} &= \{(a_1,\dots,a_n)\in X_{+,I}^*(T) \mid a_1\leq a_n \}.
\end{align*}
Next, we determine $C_{\Hasse}$. Write $\det\colon \GL_n\to \GG_{\textrm{m}}$ for the determinant. We may view it as a section in $H^0(\GZip^\mu,\Vcal_I(\lambda_{\det}))$ with $\lambda_{\det}=(1-q,\dots ,1-q)\in \ZZ^n$, which is everywhere non-vanishing. For each $\alpha\in \Delta$, let $\chi_\alpha$ be a fundamental weight of $\alpha$ (it is well-defined up to $\ZZ (1,\dots,1)$). Write $\Ha_{\alpha}\colonequals \Ha_{\chi_\alpha}$ for the attached partial Hasse invariant. The section $\Ha_{\alpha}$ vanishes exactly on the codimension $1$ stratum $\overline{\Fcal}_{w_0s_\alpha}$ and its divisor has multiplicity one. Denote by $\lambda_{\alpha}=h_{w_0}(\chi_\alpha)$ the weight of $\Ha_\alpha$. The weights $\{\lambda_{\alpha_d}\}_{1\leq d \leq n-1}$ were calculated in \cite[\S 8.3]{Imai-Koskivirta-partial-Hasse}. Up to $\ZZ \lambda_{\det}$, they are as follows. For $1\leq d\leq s$, we have
\begin{equation}
    \lambda_{\alpha_d}  = (\underbrace{1, \dots ,1}_{\textrm{$n-d$ times}},\underbrace{0, \dots ,0}_{\textrm{$d$ times}}) + (\underbrace{-q, \dots , -q}_{\textrm{$r$ times}}, \underbrace{0, \dots ,0}_{\textrm{$d$ times}}, \underbrace{-q, \dots ,-q}_{\textrm{$s-d$ times}}). 
\end{equation}
Similarly, for $s<d\leq n-1$, we have
\begin{equation}
    \lambda_{\alpha_d} = (\underbrace{1, \dots ,1}_{\textrm{$n-d$ times}},\underbrace{0, \dots ,0}_{\textrm{$d$ times}}) + (\underbrace{0, \dots , 0}_{\textrm{$d-s$ times}}, \underbrace{-q, \dots ,-q}_{\textrm{$n-d$ times}}, \underbrace{0, \dots ,0}_{\textrm{$s$ times}}). 
\end{equation}
The cone $C_{\Hasse}$ is the cone generated by the weights $\{\lambda_{\alpha_d}\}_{1\leq d \leq n-1}$ together with $\ZZ\lambda_{\det}$.

Next, we explicit the highest weight cone $C_{\rm hw}$. Since $G$ is $\FF_q$-split, we may use \cite[\S 3.6]{Koskivirta-automforms-GZip}. Define $\alpha\colonequals \alpha_r=e_r-e_{r+1}$. Note that we have $\Delta^P=\{\alpha\}$. Let $L_\alpha\subset L$ be the centralizer of $\alpha^\vee$ in $L$, and let $I_\alpha\subset I$ be the set of simple roots in $L_\alpha$. Then, by \loccitn, $C_{\hw}$ is the set of $\lambda\in X_{+,I}^*(T)$ such that $\sum_{w\in {}^{I_\alpha}W_I} q^{\ell(w)}\langle w\lambda, \alpha^\vee \rangle \leq 0 $. Here, ${}^{I_\alpha}W_I$ is the set of permutations $(\sigma,\sigma')\in \Sfr_r\times \Sfr_s$ such that
\begin{equation}
\sigma(1) > \dots > \sigma(r-1) \ \textrm{and} \ \sigma'(2) > \dots > \sigma'(s).
\end{equation}
Hence, $(\sigma,\sigma')$ is entirely determined by $i=\sigma(r)$ and $j=\sigma'(1)$, and $i,j$ can take any value such that $1\leq i \leq r$ and $1\leq j \leq s$. We deduce that $C_{\hw}$ is:
\begin{equation}\label{hw-GL}
C_{\hw}=\left\{(a_1,\dots ,a_n)\in X_{+,I}^*(T) \Bigg| \  \sum_{i=1}^r\sum_{j=1}^s q^{r+j-i-1} (a_i-a_{r+j}) \leq 0 \right\}.
\end{equation}

\subsubsection{Correspondence between automorphic forms}\label{sec-corresp}
In this section, we continue to assume that $G=\GL_{n,\FF_q}$, but we take $r=n-1$ and $s=1$. We explain a correspondence between automorphic forms on the stack of $G$-zips and automorphic forms on the stack of $G'$-zips, where
\begin{equation}
G'\colonequals \Sp(2(n-1))_{\FF_q}.   
\end{equation}
Endow $G'$ with the usual Siegel-type cocharacter $\mu'$ as explained in \S\ref{zipcone-sp6}. Let $\Zcal'=(G',P',L',Q',M',\varphi)$ by the zip datum attached to $\mu'$ (since $\mu'$ is defined over $\FF_q$, we have $M'=L'$). Let $B'$ be the lower-triangular Borel subgroup of $G'$ and $T'$ the diagonal torus. Let $\delta'\colon \GL_{n-1}\to L'$ be the isomorphism defined in \eqref{deltadef}. For $\lambda=(\lambda_1,\dots ,\lambda_n)$, write $\lambda'=(\lambda_1,\dots ,\lambda_{n-1})$. Hence, the representation $V_I(\lambda)$ of $L\simeq \GL_{n-1}\times \GG_{\textrm{m}}$ decomposes as
\begin{equation}
    V_I(\lambda)=V_{I'}(\lambda')\boxtimes \chi_{\lambda_n}
\end{equation}
where $I'$ denotes the simple roots of $L'\simeq \GL_{n-1}$, and $\chi_{\lambda_n}$ is the character $\GG_{\textrm{m}}\to \GG_{\textrm{m}}$, $x\mapsto x^{\lambda_n}$. By \eqref{sectionsVlambda}, we have
\begin{align*}
    H^0(\GZip^{\mu},\Vcal_I(\lambda)) & = V_{I}(\lambda)^{L(\FF_q)} \cap \bigoplus_{\substack{\eta\in \ZZ^{n-1} \\ \eta_{n-1}\leq \lambda_{n}}} V_{I'}(\lambda')_{\eta} \\
    H^0(\GpZip^{\mu'},\Vcal_{I'}(\lambda')) & = V_{I'}(\lambda')^{L'(\FF_q)} \cap \bigoplus_{\substack{\eta\in \ZZ^{n-1} \\ \eta_{n-1}\leq 0}} V_{I'}(\lambda')_{\eta}
\end{align*}
where $\eta=(\eta_1,\dots , \eta_{n-1})\in \ZZ^{n-1}$. If $q-1$ does not divide $\lambda_n$, we have $V_{I}(\lambda)^{L(\FF_q)}=0$ since $\chi_{\lambda_n}$ does not have $\GG_{\textrm{m}}(\FF_q)$-invariants. If $q-1$ divides $\lambda_n$, we can identify $V_{I}(\lambda)^{L(\FF_q)}=V_{I'}(\lambda)^{L'(\FF_q)}$. In particular, if $\lambda_n=0$ we have an identification $H^0(\GZip^{\mu},\Vcal_I(\lambda)) = H^0(\GpZip^{\mu'},\Vcal_{I'}(\lambda'))$. Recall that we view the determinant function as an element of $H^0(\GZip^\mu,\Vcal_I(\lambda_{\det}))$ with $\lambda_{\det}=(1-q,\dots ,1-q)\in \ZZ^n$. By twisting with powers of $\det$, we obtain immediately:

\begin{proposition} \label{prop-corresp} \ 
Let $\lambda=(\lambda_1, \dots , \lambda_n)\in \ZZ^n$ and assume that $\lambda_n=(q-1)m$ for some $m\in \ZZ$. Define $\overline{\lambda}=(\lambda_1-\lambda_n, \dots , \lambda_{n-1}-\lambda_n)\in \ZZ^{n-1}$. Then, there is an identification
\begin{equation}
   H^0(\GZip^{\mu},\Vcal_I(\lambda)) = H^0(\GpZip^{\mu'},\Vcal_{I'}(\overline{\lambda})). 
\end{equation}
\end{proposition}

\begin{corollary}\label{coro-zipcone}
Let $C'_{\zip}$ and $C_{\zip}$ be the zip cones of $(G',\mu')$ and $(G,\mu)$ respectively. Then, we have:
\begin{align}
    C_{\zip} & =(C'_{\zip} \times \{0\})+\ZZ\lambda_{\det}\\
    \langle C_{\zip} \rangle & =(\langle C'_{\zip} \rangle \times \{0\})+\ZZ(1, \dots , 1).
\end{align}
\end{corollary}

\begin{proof}
The first equality follows immediately from Proposition \ref{prop-corresp}. To show the second equality, it suffices to check that $(\langle C'_{\zip} \rangle \times \{0\})+\ZZ(1, \dots , 1)$ is a saturated subcone, which is easy.
\end{proof}

We do not know if we can expect a similar, or some form of correspondence between usual mod $p$ automorphic forms on Shimura varieties for the groups $\GU(n,1)$ and $\GSp(2(n-1))$.

\subsubsection{The case $(r,s)=(2,1)$} 

In this case, Conjecture \ref{conj-S} was proved in \cite[Theorem 5.1.1]{Goldring-Koskivirta-global-sections-compositio}. We explain here that the proof of \loccit also yields a divisibility result. First, in this case one has
\begin{equation}
    \langle C_{\zip} \rangle = \langle C_{\Hasse} \rangle = C_{\hw} = \{(a_1,a_2,a_3)\in X_{+,I}^*(T) \mid q(a_1-a_3)+(a_2-a_3) \leq 0\}
\end{equation}
where $X_{+,I}^*(T)$ is given by the condition $a_1\geq a_2$. The Griffiths--Schmid cone $C_{\GS}$ is defined inside $X_{+,I}^*(T)$ by the inequality $a_1-a_3\leq 0$. Let $(S,\zeta)$ satisfy Assumption \ref{assume}. There are two strata of codimension $1$ in $\Flag(S)$, namely $\Flag(S)_{w}$ for $w$ in the set
\begin{equation}
    \{w_0 s_{\alpha_1}, w_0 s_{\alpha_2}\}=\{[231],[312]\}.
\end{equation}

\begin{proposition}\label{propGL21} Assume that $\lambda=(a_1,a_2,a_3)\in \ZZ^3$ satisfies $a_1-a_3 > 0$. Then $H^0(\overline{\Flag}(S)_{[312]},\Vcal_{\flag}(\lambda))=0$.
\end{proposition}

\begin{proof}
The proof is similar to Proposition \ref{prop-Sp4-van} and relies on \cite[Figure 1]{Goldring-Koskivirta-global-sections-compositio} (replacing $p$ by $q$).
\end{proof}

Recall that $\Ha_{\alpha}$ denotes the partial Hasse invariant (with multiplicity $1$) with respect to $\alpha\in \Delta$. The weight of $\Ha_{\alpha_1}$ is $\lambda_{\alpha_1}\colonequals (1-q,1-q,0)$ and the weight of $\Ha_{\alpha_{2}}$ is $\lambda_{\alpha_2}\colonequals (1,-q,0)$. Similarly to Theorem \ref{thm-Sp4-div}, we deduce:

\begin{theorem}\label{thm-GL21-div}
Let $f\in H^0(S,\Vcal_I(\lambda))$ for $\lambda=(a_1,a_2,a_3)\in \ZZ^3$.  If $a_1-a_3 > 0$, then $f$ is divisible by $\Ha_{\alpha_2}$.
\end{theorem}

This result shows again the analogy between the cases $\Sp(4)$ and $\GL_3$ (for $r=2$, $s=1$). Proposition \ref{prop-corresp} is a correspondence between sections on the corresponding stacks of $G$-zips. Theorems \ref{thm-GL21-div} and \ref{thm-Sp4-div} suggest that this correspondence may extend in some way to mod $p$ automorphic forms, since the divisibility results in each case are completely similar: The relevant partial Hasse invariants and the neighborhoods of divisibility for $\Sp(4)$ and $\GL(3)$ correspond via the map $\lambda \mapsto \overline{\lambda}$.

\subsubsection{The case $(r,s)=(3,1)$}\label{sec-GL31}

We now take $G=\GL_{4,\FF_q}$ and $r=3$, $s=1$. By \S\ref{sec-corresp}, we expect similarities with the case of $\Sp(6)_{\FF_q}$. In particular, we know by Corollary \ref{coro-zipcone} that $C_{\zip}$ is generated by $C'_{\zip}\times \{0\}$ and $\lambda_{\det}$, where $C'_{\zip}$ is the zip cone of $\Sp(6)_{\FF_q}$ (endowed with the usual Siegel-type cocharacter, see \S\ref{zipcone-sp6}). For $\lambda=(a_1,\dots ,a_n)\in \ZZ^n$, write again
\begin{equation}
    \overline{\lambda}=(a_1-a_n, a_2-a_n, \dots , a_{n-1}-a_n).
\end{equation}
Similarly, for a subset $X\subset \ZZ^n$, define
\begin{equation}\label{Xbar}
\overline{X}\colonequals\{\overline{\lambda} \mid \lambda\in X\} \subset \ZZ^{n-1}.    
\end{equation}
With this notation, Corollary \ref{coro-zipcone} shows that $\overline{C}_{\zip}=C'_{\zip}$. We will show that analogues of Theorems \ref{thm-divSp6} and \ref{thm-conjSp6} hold also for the pair $(G,\mu)$ considered in this section. We refer to \S\ref{proof-GL31} for the proofs of all results mentioned below. Furthermore, all results can be visualized graphically on Figure \ref{fig:sp6}, since the case $G=\Sp(6)_{\FF_q}$ is entirely similar. Therefore, we do not reproduce this figure here.

Let $(S,\zeta)$ be a pair satisfying Assumption \ref{assume}. First, note that there are three codimension one strata in $\Flag(S)$, given by the elements
\begin{equation}
\{w_0s_{\alpha_1}, w_0s_{\alpha_2}, w_0s_{\alpha_3}\}=\{[3421],[4231],[4312]\}.    
\end{equation}
We show a result similar to Proposition \ref{prop-564} for the codimension one stratum $\Flag(S)_{[4312]}$, specifically:

\begin{proposition}\label{prop-GL31-van}
Assume that $\lambda=(a_1,a_2,a_3,a_4)\in \ZZ^4$ satisfies 
\begin{equation}
    q^2(a_1-a_4)+q(a_2-a_4)+(a_3-a_4)>0.
\end{equation}
Then one has $H^0(\overline{\Flag(S)}_{[4312]},\Vcal_{\flag}(\lambda))=0$.
\end{proposition}

We obtain a divisibility result with respect to the partial Hasse invariant $\Ha_{\alpha_3}$ which cuts out the codimension one stratum $\Flag(S)_{[4312]}=\Flag(S)_{w_0s_{\alpha_3}}$. The weight of $\Ha_{\alpha_3}$ is $(1,0,-q,0)$. Via the map $\lambda\mapsto \overline{\lambda}$, this weight maps to $(1,0,-q)$, the weight of the partial Hasse invariant appearing in Theorem \ref{thm-divSp6} in the case $\Sp(6)$. This illustrates again the surprising analogy between these two groups. The only difference is that we had to suppose $q\geq 5$ in Proposition \ref{prop-564} and Theorems \ref{thm-divSp6}, \ref{thm-conjSp6}, whereas here this assumption is superfluous.

\begin{theorem}\label{thm-divGL31}
Let $f\in H^0(\Flag(S),\Vcal_{\flag}(\lambda))$ and assume that $\lambda=(a_1,a_2,a_3,a_4)\in \ZZ^4$ satisfies $q^2(a_1-a_4)+q(a_2-a_4)+(a_3-a_4)>0$. Then $f$ is divisible by the partial Hasse invariant $\Ha_{\alpha_3}$.
\end{theorem}

\begin{proof}
By Proposition \ref{prop-GL31-van}, the restriction of $f$ to the stratum $\overline{\Flag(S)}_{[4312]} = \overline{\Flag(S)}_{w_0s_{\alpha_3}}$ is zero. Since the partial Hasse invariant $\Ha_{\alpha_3}$ cuts out $\overline{\Flag(S)}_{w_0s_{\alpha_3}}$ with multiplicity one, we deduce that $f$ is divisible by $\Ha_{\alpha_3}$.
\end{proof}

\begin{theorem}\label{thm-conjGL31}
We have $\langle C_{S} \rangle = \langle C_{\zip} \rangle$. Hence, Conjecture \ref{conj-S} holds in the case $G=\GL_{4,\FF_q}$, $r=3$, $s=1$.
\end{theorem}

\begin{proof}
The set of $\lambda=(a_1,a_2,a_3,a_4)\in X_{+,I}^*(T)$ satisfying $q^2(a_1-a_4)+q(a_2-a_4)+(a_3-a_4)\leq0$ coincides with $C_{\hw}$. Hence, we have $C^{+}_{w_0}\subset C_{\Hasse}+C_{\hw}\subset \langle C_{\zip}\rangle$. By Theorem \ref{thm-sep-syst}, we have $ C_S \subset \langle C^{+}_{w_0} \rangle \subset  \langle C_{\zip}\rangle$, hence also $\langle C_S \rangle \subset \langle C_{\zip}\rangle$. Since the converse inclusion is always satisfied, the result follows.
\end{proof}

\subsubsection{The case $(r,s)=(2,2)$}\label{sec-GL22}
In this section, we continue to assume $G=\GL_{4,\FF_q}$ but we take $r=s=2$. In this case, $(G,\mu)$ is of Hasse-type (see \S\ref{sec-hassetype}). By Theorem \ref{thm-Hasse-type}, we deduce that $\langle C_{\zip}\rangle = \langle C_{\Hasse} \rangle$. One checks easily that this cone is given by:
\begin{equation}
    \langle C_{\Hasse}\rangle=\{(a_1,a_2,a_3,a_4)\in X_{+,I}^*(T) \mid q(a_1-a_4)+(a_2-a_3)\leq 0\}.
\end{equation}
On the other hand, using \eqref{hw-GL} we find that the cone $C_{\hw}$ is given by the inequality $q(a_1-a_3)+q^2(a_1-a_4)+(a_2-a_3)+q(a_2-a_4)\leq 0$, which can be rearranged as $q(a_1-a_4)+(a_2-a_3)\leq 0$. Hence, we obtain
\begin{equation}\label{GL22-cones}
    C_{\hw}=\langle C_{\Hasse}\rangle=\langle C_{\zip} \rangle.
\end{equation}
Let $(S,\zeta)$ be a pair satisfying Assumption \ref{assume}. We show in \S \ref{proof-GL22}:
\begin{theorem}\label{thmGL22-conj}
We have $\langle C_{S} \rangle = \langle C_{\zip} \rangle$. Hence, Conjecture \ref{conj-S} holds in the case $G=\GL_{4,\FF_q}$, $r=2$, $s=2$. In particular, if $\lambda=(a_1,a_2,a_3,a_4)\in \ZZ^4$ satisfies $q(a_1-a_4)+(a_2-a_3)> 0$ then $H^0(S,\Vcal_I(\lambda))=0$.
\end{theorem}

We also have divisibility results. As we noted at the end of \S \ref{sec-divis}, when $(G,\mu)$ is of Hasse-type, we expect divisibility results for all roots $\alpha\in I$ (except in trivial cases), and roots of $\Delta\setminus I$ which are not fixed by the Galois action. Hence, in the case $G=\GL_{4,\FF_q}$, $(r,s)=(2,2)$, we expect a result for the two roots in $I$, namely $[3421]=ws_{\alpha_1}$ and $[4312]=w_0s_{\alpha_3}$, with notation as in \S\ref{sec-cones-GLn}. We indeed have such divisibility theorems. We first state the vanishing result on these codimension one strata:

\begin{proposition}\label{propGL22} Let $\lambda=(a_1,a_2,a_3,a_4)\in \ZZ^4$.
\begin{assertionlist}
    \item \label{propGL22-1} Assume that $\lambda$ satisfies either $(a_1-a_4)+\epsilon(q)(a_2-a_4) > 0$ where $\epsilon(q)\colonequals \frac{q^2+2q+1}{q^3+2q^2+1}$, or $a_2-a_4>0$. Then $H^0(\overline{\Flag}(S)_{[3421]},\Vcal_{\flag}(\lambda))=0$.
    \item \label{propGL22-2} Assume that $\lambda$ satisfies $a_1-a_3 > 0$. Then $H^0(\overline{\Flag}(S)_{[4312]},\Vcal_{\flag}(\lambda))=0$.
\end{assertionlist}
\end{proposition}

For each $\alpha\in \Delta$, let again $\Ha_{\alpha}$ denote a partial Hasse invariant which cuts out the stratum $\Fcal_{w_0s_{\alpha}}$ with multiplicity $1$. We deduce as in previously mentioned cases:
\begin{theorem}\label{thm-GL22}
Let $f\in H^0(S,\Vcal_I(\lambda))$ for $\lambda=(a_1,a_2,a_3,a_4)\in \ZZ^4$.
\begin{assertionlist}
    \item \label{thmGL22-1} If $(a_1-a_4)+\epsilon(q)(a_2-a_4) > 0$, then $f$ is divisible by $\Ha_{\alpha_1}$.
    \item \label{thmGL22-2} If $a_1-a_3 > 0$, then $f$ is divisible by $\Ha_{\alpha_3}$.
\end{assertionlist}
\end{theorem}
For the stratum parametrized by $w_0s_{\alpha_1}=[3421]$, we could use Proposition \ref{propGL22}\eqref{propGL22-1} to deduce that if $a_2-a_4>0$, then $f$ is divisible by $\Ha_{\alpha_1}$. However, one can show that this result is already contained in Theorem \ref{thm-GL22}\eqref{thmGL22-1}. We illustrate Theorem \ref{thm-GL22} in the figure below. Since $\langle C_{\zip} \rangle$ is four-dimensional, we reduce the dimension as follows: First, we map all cones to $\ZZ^3$ by the map $\lambda\mapsto \overline{\lambda}$. Then, we represent a generic slice of these subcones of $\ZZ^3$. Applying $\lambda\mapsto \overline{\lambda}$, the weight of $\Ha_{\alpha_1}$ becomes $\overline{\lambda}_{\alpha_1}=(1,1,q+1)$ and the weight of $\Ha_{\alpha_3}$ becomes $\overline{\lambda}_{\alpha_3}=(1,-q,0)$. Lastly, the weight of the third partial Hasse invariant $\Ha_{\alpha_2}$ lies on the half-line $X_{+}^*(L)$. It maps to $\overline{\lambda}_{\alpha_2}=(1-q,1-q,0)$ and $X_{+}^*(L)$ maps to $\NN(-1,-1,0)$. The weights $\{\lambda_{\alpha_1},\lambda_{\alpha_2},\lambda_{\alpha_3},\lambda_{\det}\}$  generate $\langle C_S \rangle=\langle C_{\zip} \rangle=\langle C_{\Hasse} \rangle=C_{\hw}$. Denote by $V_1$ and $V_3$ the subsets of $\ZZ^3$ defined by
\begin{align*}
    V_1 & \colonequals \{ (a_1,a_2,a_3)\in \ZZ^3  \mid a_1+\epsilon(q)a_2 > 0\} \\
    V_3 & \colonequals \{ (a_1,a_2,a_3)\in \ZZ^3  \mid a_1-a_3 > 0\}.
\end{align*}
Then $V_1$ is a cone neighborhood of $\overline{\lambda}_{\alpha_1}$ and $V_3$ is a cone neighborhood of $\overline{\lambda}_{\alpha_3}$. Theorem \ref{thm-GL22} asserts that if (the image by $\lambda\mapsto \overline{\lambda}$ of) the weight of $f$ lies in $V_i$, then $f$ is divisible by $\Ha_{\alpha_{i}}$ for $i=1,3$. Since $\epsilon(q)$ is very close to $\frac{1}{q}$, the weight $\overline{\lambda}_{\alpha_3}$ lies "just outside" of $V_1$.

\begin{figure}[H]
    \centering
    \includegraphics[width=14cm]{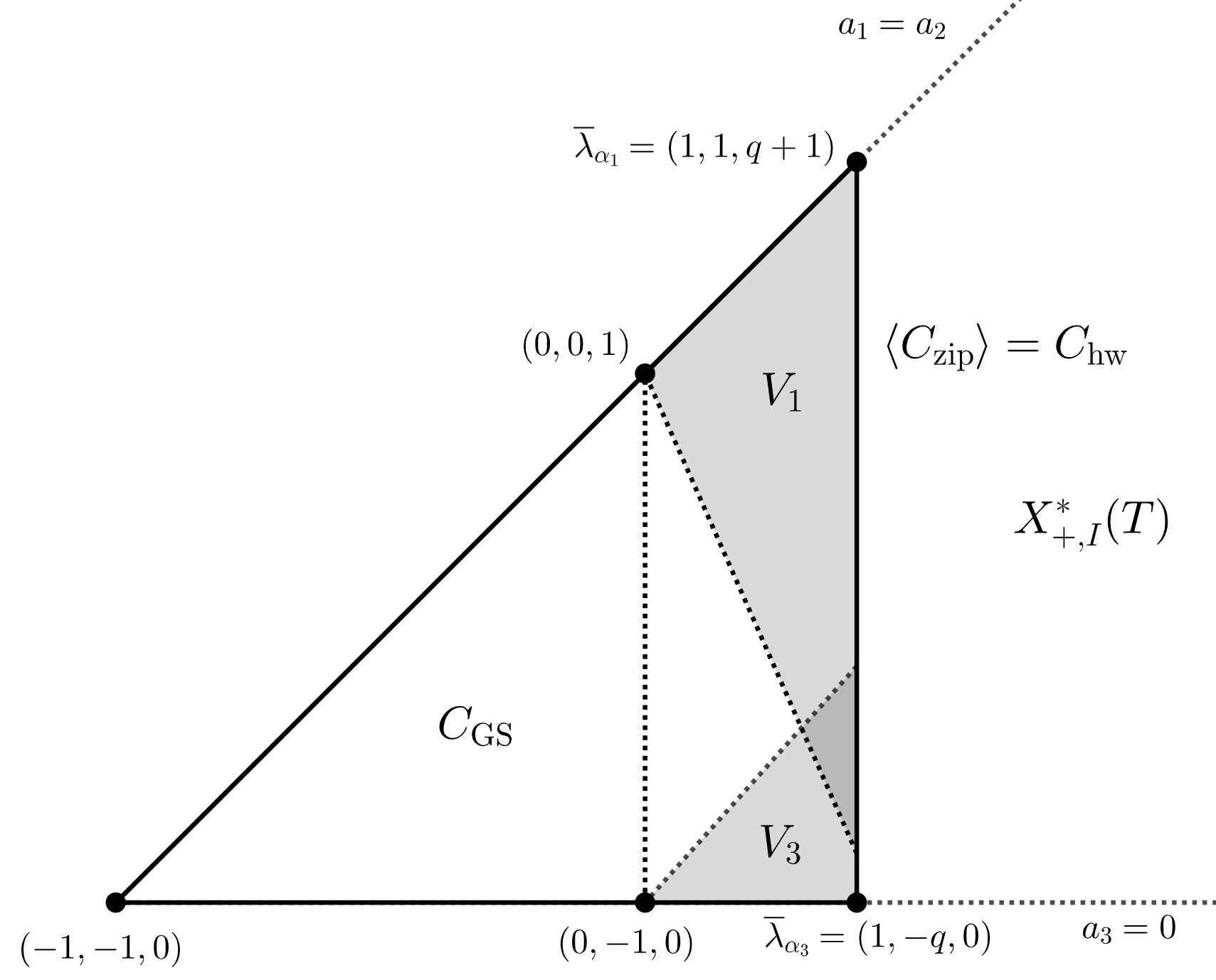}
    \caption{The case of $\GL_{4,\FF_q}$ for $r=s=2$}
    \label{fig:GL22}
\end{figure}

\subsubsection{The remaining cases}\label{sec-remainGL}
In the case $(r,s)=(1,3)$, the cocharacter $\mu$ is conjugated to the one for $(r,s)=(3,1)$. Therefore, the stacks of $G$-zips attached to both cocharacters are isomorphic. Therefore, all results from \S\ref{sec-GL31} transpose to the case $(r,s)=(1,3)$. Finally, for $(r,s)=(4,0)$ or $(0,4)$ we have $P=G$. Hence, Conjecture \ref{conj-S} follows from Remark \ref{rmk-GeqP}.

\subsection{Unitary groups} \label{sec-unit}

\subsubsection{Group theory}\label{sec-cones-Un}

Let $(V,\psi)$ be an $n$-dimensional $\FF_{q^2}$-vector space endowed with a non-degenerate hermitian form $\psi \colon V\times V\to \FF_{q^2}$ (for Shimura varieties, take $q=p$). We assume that there is a basis $\Bcal$ of $V$ where $\psi$ is given by the matrix 
\[
 J= \begin{pmatrix}
&&1\\& \iddots &\\1&&\end{pmatrix}. 
\]
Let $G=\U(V,\psi)$ be the associated unitary group. There is an isomorphism $G_{\FF_{q^2}}\simeq \GL(V_0)$. It is induced by the $\FF_{q^2}$-algebra isomorphism $\FF_{q^2}\otimes_{\FF_q} R\to R\times R$, $a\otimes x\mapsto (ax,\sigma(a)x)$ (where $\Gal(\FF_{q^2}/\FF_q)=\{\id, \sigma\}$). The action of $\sigma$ on the set $\GL_n(k)$ is given by $\sigma\cdot A = J \sigma({}^t \!A)^{-1}J$. Let $T$ denote the maximal diagonal torus and $B$ the lower-triangular Borel subgroup of $G_k$. By our choice of the basis $\Bcal$, the groups $B$ and $T$ are defined over $\FF_q$. Identify $X^*(T)=\ZZ^n$ as in \S\ref{sec-cones-GLn} and retain the notation for the simple roots $\{\alpha_i\}_{1\leq i \leq n-1}$.

Choose non-negative integers $(r,s)$ such that $n=r+s$ and define $\mu  \colon  \GG_{\mathrm{m},k}\to G_{k}$ by $x\mapsto \diag(xI_r,I_s)$ via the identification $G_{k}\simeq \GL_{n,k}$. Let $\Zcal_{\mu}=(G,P,L,Q,M,\varphi)$ be the associated zip datum. Note that $P$ is not defined over $\FF_q$ unless $r=s$. We may also identify $L=\GL_{r}\times \GL_{s}$. One has $\Delta^P=\{\alpha\}$ with $\alpha=e_r-e_{r+1}$. The determinant function $\det  \colon  G_k\to \GG_{\mathrm{m}}$ is a section of weight $\lambda_{\det}=(q+1, \dots , q+1)$. Again, we have partial Hasse invariants $\{\Ha_{\alpha}\}_{\alpha\in \Delta}$ with multiplicity $1$. For $1\leq d\leq r$, we have (up to $\ZZ \lambda_{\det}$)
\begin{equation}
    \lambda_{\alpha_d}  = (\underbrace{1, \dots ,1}_{\textrm{$n-d$ times}},\underbrace{0, \dots ,0}_{\textrm{$d$ times}}) + (\underbrace{q, \dots , q}_{\textrm{$r-d$ times}}, \underbrace{0, \dots ,0}_{\textrm{$d$ times}}, \underbrace{q, \dots ,q}_{\textrm{$s$ times}}). 
\end{equation}
Similarly, for $r<d\leq n-1$, we have
\begin{equation}
    \lambda_{\alpha_d} = (\underbrace{1, \dots ,1}_{\textrm{$n-d$ times}},\underbrace{0, \dots ,0}_{\textrm{$d$ times}}) + (\underbrace{0, \dots , 0}_{\textrm{$r$ times}}, \underbrace{q, \dots ,q}_{\textrm{$n-d$ times}}, \underbrace{0, \dots ,0}_{\textrm{$d-r$ times}}). 
\end{equation}
The cone $C_{\Hasse}$ is the cone generated by the weights $\{\lambda_{\alpha_d}\}_{1\leq d \leq n-1}$ together with $\ZZ\lambda_{\det}$. The Griffiths--Schmid cone $C_{\GS}$ is independant of the Galois action, therefore is the same as in the case of $\GL_{n,\FF_q}$. Contrary to the case $\GL_{n,\FF_q}$, the highest weight cone $C_{\hw}$ and lowest-weight cone $C_{\rm lw}$ (see \S\ref{subsec-subcones}) do not coincide (because $P$ is not defined over $\FF_q$). By \cite[Corollary 5.2.5]{Goldring-Imai-Koskivirta-weights}, one always has $C_{\GS}\subset C_{\rm lw}$. However, in the case of $\U(n)$, the cone $C_{\hw}$ is much smaller than $C_{\rm lw}$.

\subsubsection{The case $(r,s)=(2,1)$}
In this case, $\langle C_{\zip} \rangle$ was determined in \cite[Corollary 6.3.3]{Imai-Koskivirta-vector-bundles}. Moreover, by \cite[Proposition 7.1.1]{Goldring-Imai-Koskivirta-weights}, we have
\begin{align*}
\langle C_{\zip} \rangle  =C_{\lw}&= \{(a_1,a_2,a_3)\in X_{+,I}^*(T) \mid (q-1)a_1+a_2-qa_3\leq 0\}\\
C_{\hw} & = \{(a_1,a_2,a_3)\in X_{+,I}^*(T) \mid q a_1-(q-1)a_2-a_3\leq 0\}.
\end{align*}
One sees from these equations that the cone $C_{\lw}$ shrinks and tends towards $C_{\GS}$ as $q$ goes to infinity. On the other hand, $C_{\hw}$ tends towards $X_{-}^*(L)$. Let $(S,\zeta)$ be a pair satisfying Assumption \ref{assume}. As in the case of $\GL_3$, there are two codimension one strata in $\Flag(S)$, parametrized by
\begin{equation}
    \{w_0 s_{\alpha_1}, w_0 s_{\alpha_2}\}=\{[231],[312]\}.
\end{equation}
We have the following vanishing result on the stratum $\overline{\Flag(S)}_{[231]}$. We refer to \S\ref{sec-U21-proof} for the proof.

\begin{proposition}\label{propU21inert}
Assume that $\lambda=(a_1,a_2,a_3)\in \ZZ^4$ satisfies $a_1-a_3>0$. Then
$H^0(\overline{\Flag(S)}_{[231]},\Vcal_{\flag}(\lambda))=0$.
\end{proposition}

The stratum $\overline{\Flag(S)}_{[231]}$ is cut out by the partial Hasse invariant $\Ha_{\alpha_1}$ which has weight $(q+1,1,q)$. As in previous cases, we obtain immediately:

\begin{theorem}\label{thm-divGU21inert}
Let $f\in H^0(\Flag(S),\Vcal_{\flag}(\lambda))$ and assume that $\lambda=(a_1,a_2,a_3)\in \ZZ^4$ satisfies $a_1-a_3>0$. Then $f$ is divisible by the partial Hasse invariant $\Ha_{\alpha_1}$.
\end{theorem}

We also deduce:

\begin{theorem}\label{thm-U21inert}
We have $\langle C_S \rangle=\langle C_{\zip}\rangle$, thus Conjecture \ref{conj-S} holds in the case $G=\U(3)_{\FF_q}$, $r=2$ and $s=1$.
\end{theorem}

\begin{proof}
By Theorem \ref{thm-sep-syst}, we have $ C_S \subset \langle C^{+}_{w_0} \rangle \subset  \langle C_{\zip}\rangle$, hence also $\langle C_S \rangle \subset \langle C_{\zip}\rangle$.
\end{proof}
Besides the partial Hasse invariant $\Ha_{\alpha_1}$ and $\Ha_{\alpha_2}$, we also have the $\mu$-ordinary Hasse invariant, that we denote abusively by $\Ha_{\mu}$ (although it is not a section of the form $\Ha_{\chi}$ for some character $\chi$). It was first constructed in  \cite{Goldring-Nicole-mu-Hasse} and \cite{Koskivirta-Wedhorn-Hasse}. By definition, $\Ha_\mu$ is a section over $\GZip^\mu$ of $\Vcal_I(\lambda_\mu)$ for some $\lambda_\mu\in X^*(L)$ whose non-vanishing locus is the unique open stratum $\Ucal_\mu$ in $\GZip^\mu$. Note that since $\lambda_\mu\in X^*(L)$, the vector bundle $\Vcal_I(\lambda_\mu)$ is a line bundle on $\GZip^{\mu}$. For a cocharacter datum $(G,\mu)$ with $\mu$ defined over $\FF_q$ (as in the case of $\GL_3$), the pull-back of $\Ha_\mu$ via the map $\GF^\mu\to \GZip^\mu$ is a partial Hasse invariant or a product thereof, by \cite[Lemma 5.2.8]{Imai-Koskivirta-partial-Hasse}. In contrast, in cases like $G=\U(3)$ when $\mu$ is not defined over $\FF_q$, $\Ha_\mu$ is completely unrelated to partial Hasse invariants. See \cite[Lemma 6.3.1]{Imai-Koskivirta-vector-bundles} for more details on $\Ha_\mu$ and an explicit formula.

We illustrate Theorem \ref{thm-divGU21inert} below. We represent the images of the cones by the map $\lambda\mapsto \overline{\lambda}$, as in \S\ref{sec-GL31}. By this map, the weights of the partial Hasse invariants become $\overline{\lambda}_{\alpha_1}=(1,1-q)$ and $\overline{\lambda}_{\alpha_2}=(1-q,-q)$. The weight of $\Ha_{\mu}$ becomes $\overline{\lambda}_\mu=(1-q^2,1-q^2)$. By Theorem \ref{thm-divGU21inert}, any form whose weight lies in the grey area is divisible by $\Ha_{\alpha_1}$.

\begin{figure}[H]
    \centering
    \includegraphics[width=15cm]{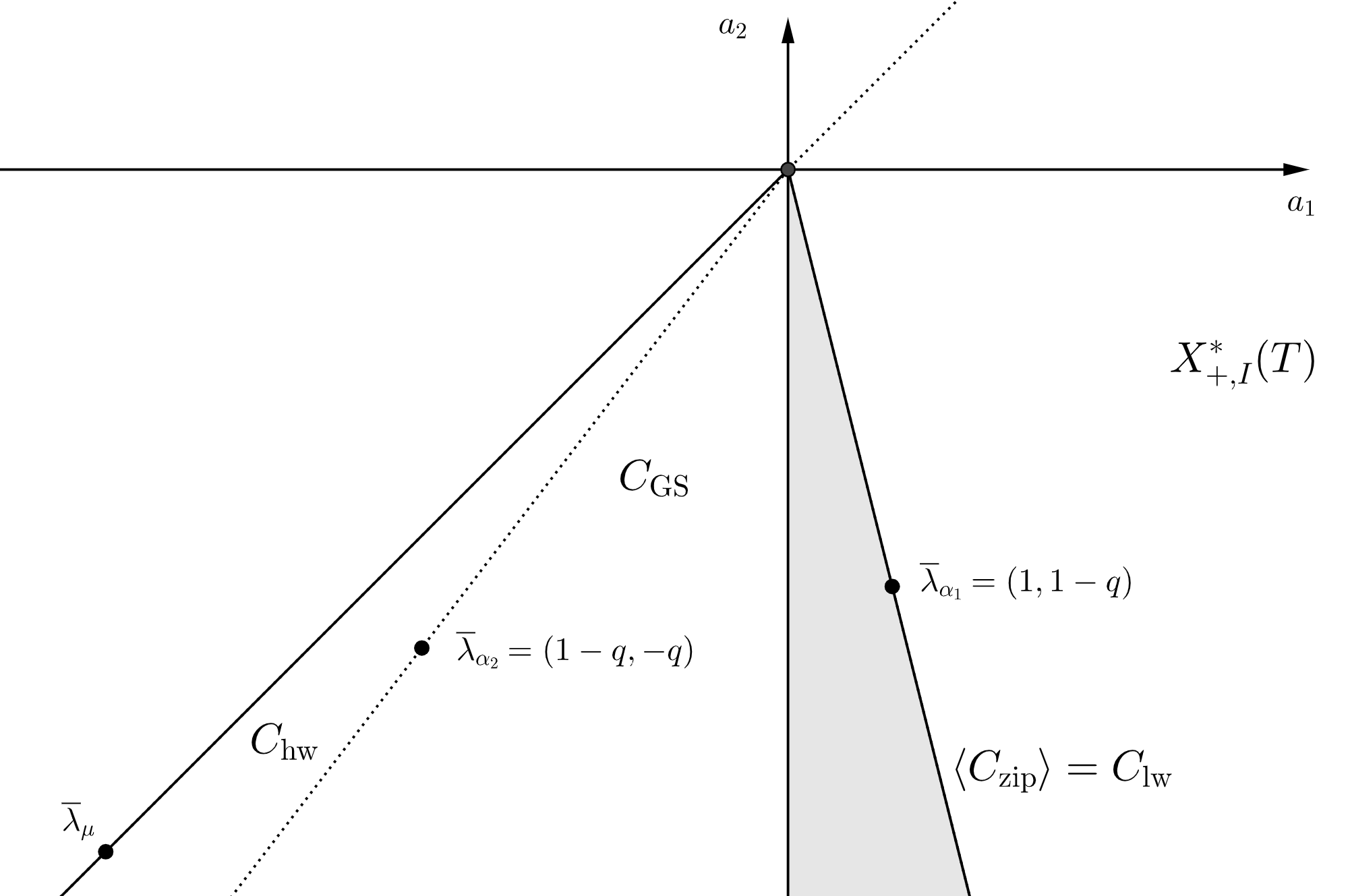}
    \caption{The case of $\U(3)_{\FF_q}$ for $r=2$, $s=1$}
    \label{fig:U21diagram}
\end{figure}

\subsubsection{The case $(r,s)=(3,1)$}
We now take $r=3$, $s=1$ for $G=\U(4)_{\FF_q}$. We start by determining the lowest weight cone $C_{\lw}$ and the highest weight cone $C_{\hw}$.

\begin{lemma}\label{lemma31inert}
One has
\begin{align*}
    C_{\lw}=\{(a_1,a_2,a_3,a_4)\in X_{+,I}^*(T) \mid (q-1)(a_1-a_4)+(a_3-a_4)\leq 0 \}\\
    C_{\hw}=\{(a_1,a_2,a_3,a_4)\in X_{+,I}^*(T) \mid (q-1)(a_1-a_3)+(a_3-a_4)\leq 0 \}\\
\end{align*}
\end{lemma}

\begin{proof}
We only prove the result for $C_{\lw}$, the case of $C_{\hw}$ is similar (but easier). We use equation \eqref{formula-norm-low}. Denote by $P_0$ the largest subgroup of $P$ defined over $\FF_q$. It is the parabolic subgroup of type $I_0=\{e_2-e_3\}$. Hence $\Delta^{P_0}=\{e_1-e_2,e_3-e_4\}$. For all $\alpha\in \Delta^{P_0}$, we may take $r_\alpha=2$. For $\lambda=(a_1,a_2,a_3,a_4)\in \ZZ^4$, we have $\lambda_0=w_{0,I_0}w_{0,I}\lambda=(a_3,a_1,a_2,a_4)$. Then, $C_{\rm lw}$ is the set of $\lambda=(a_1,a_2,a_3,a_4)\in X_{+,I}^*(T)$ satisfying
\begin{equation}\label{formula-norm-low-U31}
\begin{cases}
(q-1)(a_1-a_4)+(a_3-a_4)\leq 0  \\
-(q-1)(a_2-a_4) +q(a_3-a_4)\leq 0.
\end{cases}
\end{equation}
Hence it suffices to show that if $\lambda=(a_1,a_2,a_3,a_4)\in X_{+,I}^*(T)$ satifies the first of the above inequalities, then it also satisfies the second one. This follows from:
\begin{equation}
-(q-1)(a_2-a_4) +q(a_3-a_4) = \frac{1}{q} ((q-1)(a_1-a_4)+(a_3-a_4))+ \frac{q-1}{q}(a_2-a_1) + \frac{q^2-1}{q}(a_3-a_2).
\end{equation}
\end{proof}

As in the case $G=\U(3)$, $(r,s)=(2,1)$, one sees easily that $C_{\hw}$ is very small, and converges towards $X_{-}^*(L)$ as $q$ goes to infinity. Let $(S,\zeta)$ be a pair satisfying Assumption \ref{assume}. There are three codimension $1$ strata in $\Flag(S)$, corresponding to the elements
\begin{equation}
    \{w_0s_{\alpha_1}, w_0s_{\alpha_2}, w_0s_{\alpha_3}\}=\{[3421],[4231],[4312]\}. 
\end{equation}
We show a vanishing result on the stratum $\overline{\Flag(S)}_{[3421]}$:

\begin{proposition}\label{prop-U31-inert-van}
Let $\lambda=(a_1,a_3,a_4,a_4)\in \ZZ^4$ such that $(q-1)(a_1-a_4)+(a_2-a_4)> 0$. Then $H^0(\overline{\Flag(S)}_{[3421]},\Vcal_{\flag}(\lambda))=0$.
\end{proposition}

Note that the expression appearing in Proposition \ref{prop-U31-inert-van} differs slightly from the one appearing in $C_{\lw}$ (Lemma \ref{lemma31inert}), namely $a_3$ is replaced by $a_2$. We prove Proposition \ref{prop-U31-inert-van} in \S \ref{sec-U31-proof}. The stratum $\overline{\Flag(S)}_{[3421]}$ is cut out (with multiplicity one) by the partial Hasse invariant $\Ha_{\alpha_1}$, which has weight $(q+1,q+1,1,q)$. We deduce immediately:

\begin{theorem}\label{thm-U31-inert-div}
Let $f\in H^0(S,\Vcal_I(\lambda))$ with $\lambda=(a_1,a_3,a_4,a_4)\in \ZZ^4$. If $(q-1)(a_1-a_4)+(a_2-a_4)> 0$, then $f$ is divisible by $\Ha_{\alpha_1}$.
\end{theorem}

Finally, we prove Conjecture \ref{conj-S} in the case $G=\U(4)_{\FF_q}$, $r=3$ and $s=1$. More specifically, we show the following result.

\begin{theorem}\label{thm-U31-conj}
Let $(S,\zeta)$ satisfy Assumption \ref{assume}. We have 
\begin{equation}
    \langle C_{S} \rangle = \langle C_{\zip} \rangle=C_{\rm lw}=\{(a_1,a_2,a_3,a_4)\in X_{+,I}^*(T) \mid (q-1)(a_1-a_4)+(a_3-a_4)\leq 0 \}.
\end{equation}
\end{theorem}

\begin{proof}
We already proved the lase equality. Since $C_{\rm lw}\subset \langle C_{\zip} \rangle \subset \langle C_{S} \rangle$, it suffices to show that $C_S$ also satisfies $(q-1)(a_1-a_4)+(a_3-a_4)\leq 0$. Let $f\in H^0(\Flag(S),\Vcal_{\flag}(\lambda))$ be a nonzero section. Since the partial Hasse invariant $\Ha_{\alpha_1}$ has multiplicity one, we can write $f=\Ha_{\alpha_1}^m g$ for some $m\geq 0$ and for some section $g$ which is not identically zero on the codimension one stratum $\overline{\Flag(S)}_{w_0s_{\alpha_1}}$. Write $\lambda_g$ for the weight of $g$ (hence $\lambda=m\lambda_{\alpha_1}+\lambda_g$). First, we note that $\lambda_{\alpha_1} \in C_{\rm lw}$. Indeed, $\lambda_{\alpha_1} = (q+1,q+1,1,q)$ satisfies $(q-1)(a_1-a_4)+(a_3-a_4)\leq 0$, hence $\lambda_{\alpha_1} \in C_{\rm lw}$. It remains to show that $\lambda_g=(b_1,b_2,b_3,b_4)$ lies in $C_{\rm lw}$. By Proposition \ref{prop-U31-inert-van}, we have $(q-1)(b_1-b_4)+(b_2-b_4)\leq 0$. Since $g$ is a nonzero global section over $\Flag(S)$, we also have $\lambda_g\in X_{+,I}^*(T)$, hence $b_1\geq b_2\geq b_3$. In particular $(q-1)(b_1-b_4)+(b_3-b_4)\leq (q-1)(b_1-b_4)+(b_2-b_4)\leq 0$, which terminates the proof.
\end{proof}

As in the case $G=\GL_{4,\FF_q}$, we represent a generic slice of the image of the cones by the map $\lambda\mapsto \overline{\lambda}$. By this map, the weight of the three partial Hasse invariants become $\overline{\lambda}_{\alpha_{1}}=(1,1,1-q)$, $\overline{\lambda}_{\alpha_{2}}=(1,1-q,-q)$ and
$\overline{\lambda}_{\alpha_{3}}=(1-q,-q,-q)$. The area colored in grey is the image by $\lambda \mapsto \overline{\lambda}$ of the set
\begin{equation}
    V\colonequals \{(a_1,a_2,a_2,a_4)\in \langle C_{\zip}\rangle \mid (q-1)(a_1-a_4)+(a_2-a_4)>0 \}.
\end{equation}
Any form whose weight lies in the grey area is divisible by $\Ha_{\alpha_1}$. There is another potential candidate for a divisibility theorem, namely the form $\Norm(f_{\eta,\low})$ with $\eta=(1,1-q,1-q)$, whose weight generates an extremal ray of $\langle C_{\GS} \rangle$ (see Figure \ref{fig:U31inert-cones} below). As in the case of $\Sp(6)$, it would be interesting to investigate divisibility by such forms, beyond the case of partial Hasse invariants.

\begin{figure}[H]
    \centering
    \includegraphics[width=14cm]{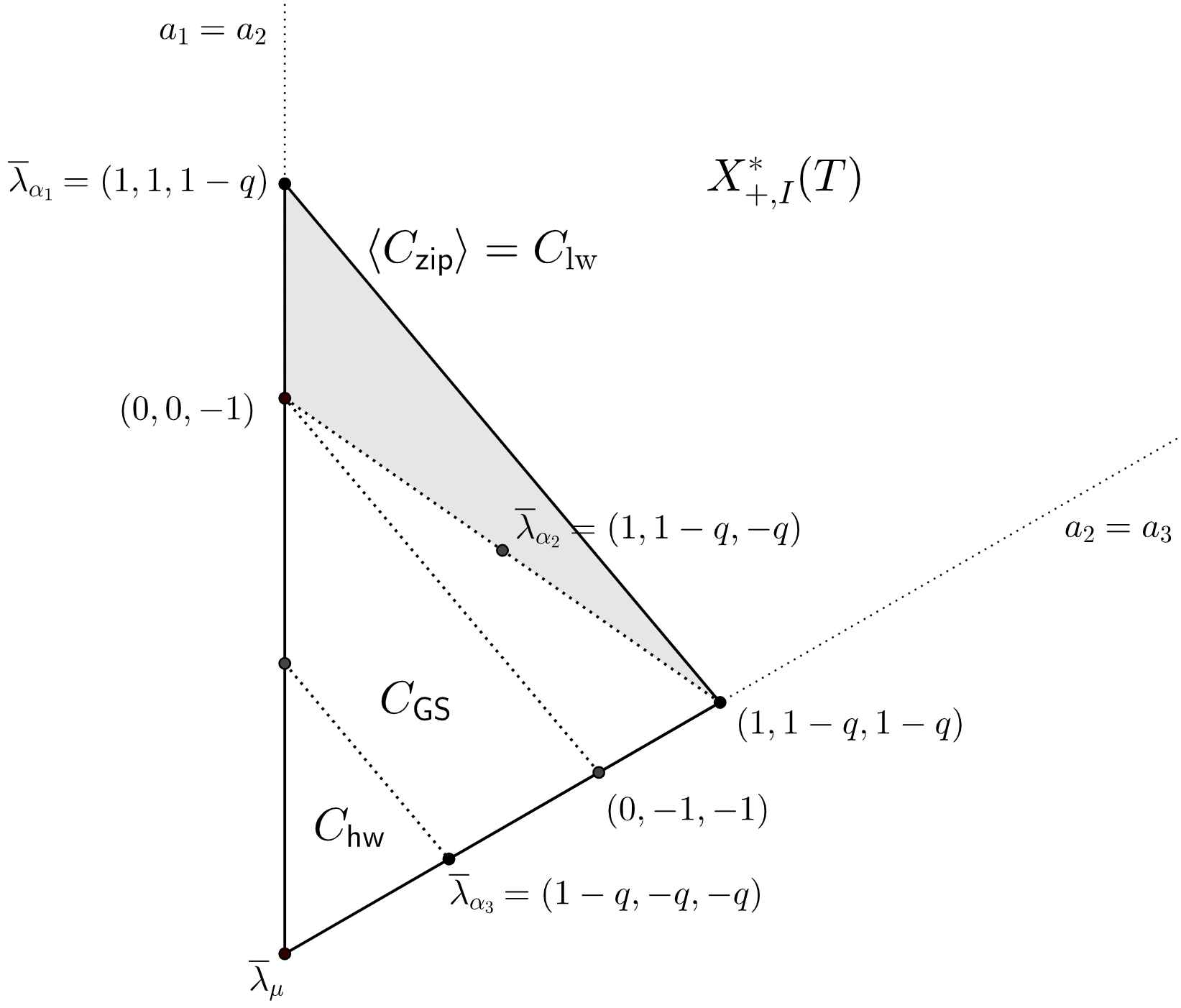}
    \caption{The case of $\U(4)_{\FF_q}$ for $r=3$, $s=1$}
    \label{fig:U31inert-cones}
\end{figure}

\subsubsection{The remaining cases}
As we noted in \S\ref{sec-remainGL}, the case $(r,s)=(1,3)$ is completely similar to the case $(r,s)=(3,1)$. Also, the cases $(r,s)=(4,0)$ and $(0,4)$ follow from Remark \ref{rmk-GeqP}. It remains to consider the case $(r,s)=(2,2)$, which seems to be the most difficult one. We were not able to determine the cone $\langle C_{\zip} \rangle$, let alone $\langle C_S \rangle$, and we could not prove Conjecture \ref{conj-S}. We only have a partial vanishing result for this case, we mention it below without proof. Assume now that $(r,s)=(2,2)$ and let $(S,\zeta)$ be a scheme satisfying Assumption \ref{assume}.

\begin{proposition}
Let $q>2$. Suppose that $\lambda=(a_1,a_2,a_3,a_4)\in \ZZ^4$ satisfies $q(a_1-a_4)+(a_2-a_3)>0$. Then
$H^0(S,\Vcal_{I}(\lambda))=0$.
\end{proposition}

We do not know whether this result also holds for $q=2$.

\section{Proofs}

\subsection{The case $G=\Sp(6)_{\FF_q}$}\label{Sp6-proof}
As explained in \S\ref{sec-mainSp6}, it suffices to show Lemma \ref{lem-564}. For this, we will consider a separating system $\EE=(\EE_w, \{\chi_\alpha \}_{\alpha\in \EE_w})_{w\in W}$. The element $w\in W$ that will be relevant for us were already mentioned in Figure \ref{fig:sepsyst}. For all other $w\in W$, we simply set $\EE_w=\emptyset$. The first step is to compute the Hasse cones $C_{\Hasse,w}$ for elements $w$ of length $1$. For each such element, we defined $C^{+}_w=C_{\Hasse,w}$ in Definition \ref{def-intersumcone} (to simplify the notation, we write $C_w^{+}$ instead of $C_w^{\EE,+}$). We record in the table below the equations for $C_{\Hasse,w}$ for the three elements of length one. They are an easy computation using \eqref{CHassew-def}.

\bigskip

\hspace{-0.7cm}
{\renewcommand{\arraystretch}{1.5}%
\begin{tabular}{ |c|c|c|l|c| } 
\hline
$w$  & $E_w $  &  $C_{\Hasse,w}$ & $C^{+}_w$  \\
\hline
$[132]$ & $e_2-e_3$ & $-q a_1+(q-1)a_2+a_3 \leq 0$ & same \\
\hline
$[124]$ & $2e_3$ & $q a_1-a_3 \leq 0$ & same \\
\hline
$[213]$ & $e_1-e_2$& $-a_1-(q-1)a_2+qa_3 \leq 0$ & same \\
\hline
\end{tabular}} 

\bigskip

For elements of length $\geq 2$, we will only consider certain elements $w$ (for the other ones, set $\EE_w=\emptyset$). In each case, we indicate the set $E_w$ and the subset $\EE_w\subset E_w$. For each $\alpha\in \EE_w$, we also indicate the corresponding lower neighbor $ws_{\alpha}$, and we define $\LL_w=\{ws_\alpha\}_{\alpha\in \EE_w}$. For each $w$ and each $\alpha\in \EE_w$, we give a character $\chi_\alpha$ satisfying Definition \ref{def-sep-syst}. More precisely, all elements $w$ that appear in the proof admit a full separating system of partial Hasse invariants (Definition \ref{def-Q-sep}). Hence, we do not need to worry about which $w'\in E_w$ is connected with an element of $\EE_w$ in the sense of Definition \ref{def-sep-syst}(c). Indeed, for all $\alpha\in \EE_w$, we will define characters $\{\chi_\alpha\}_{\alpha\in \EE_w}$ satisfying the stronger conditions:
\begin{equation}
    \begin{cases}
     \langle \chi_{\alpha},\alpha^\vee \rangle >0 & \\
\langle \chi_{\alpha},\beta^\vee \rangle =0 & \textrm{ for all } \beta\in E_{w}\setminus \{\alpha\}.
    \end{cases}
\end{equation}
We also indicate the image of $\chi_\alpha$ by the map $h_w$. Lastly, in the rightmost column, we give an upper bound for the cone $C^+_{w}$. We explain this upper bound below the table. Here is the table for the relevant elements of length $2$:

\bigskip

\hspace{-2cm}
{\renewcommand{\arraystretch}{1.2}%
\begin{tabular}{ |c|c|c|c|c|c|l| } 
\hline 
\multicolumn{7}{|l|}{Relevant elements of length $2$} \\
\hline
$w$ & $E_w$ & $\EE_w$ & $\LL_w$ &$\{\chi_\alpha\}_{\alpha\in \EE_w}$ & $\{h_w(\chi_\alpha)\}_{\alpha\in \EE_w}$   & Upper bound for $C^{+}_w$\\
\hline
\multirow{2}{2em}{$[135]$} & \multirow{2}{3em}{$e_2+e_3$ \\ $2e_3$} & \multirow{2}{3em}{$e_2+e_3$\\$2e_3$} & \multirow{2}{2em}{$[124]$\\$[132]$}  &\multirow{2}{3em}{$(0,1,0)$\\$(0,-1,1)$}& \multirow{2}{6em}{$(0,-q,-1)$\\$(-q,q+1,1)$} & $(q^2+q) a_1 + (q^2+1)a_2  - (q+1) a_3 \leq 0$ \\ 
&&&&&& $q^2 a_1 +a_2 - qa_3 \leq 0$  \\ 
\hline
\multirow{2}{2em}{$[142]$} &\multirow{2}{3em}{$e_2-e_3$ \\ $2e_2$}& \multirow{2}{3em}{$e_2-e_3$ \\ $2e_2$} & \multirow{2}{2em}{$[124]$ \\ $[132]$} &\multirow{2}{4em}{$(0,0,-1)$ \\ $(0,1,1)$}&\multirow{2}{7em}{$(q,1,0)$ \\ $(-q,-(q+1),1)$}& $-(q^2+q) a_1 + (q^2+1)a_2  + (q+1) a_3 \leq 0$ \\ 
&&&&&& $-q a_1 +q^2 a_2 + a_3 \leq 0$ \\ 
\hline
\multirow{2}{2em}{$[214]$} &\multirow{2}{3em}{$e_1-e_2$ \\ $2e_3$}& \multirow{2}{3em}{$e_1-e_2$ \\ $2e_3$} & \multirow{2}{2em}{$[124]$ \\ $[213]$} &\multirow{2}{4em}{$(1,0,0)$ \\ $(0,0,1)$}&\multirow{2}{7em}{$(0,-1,-q)$ \\ $(-q,0,1)$}& $(q+1)a_1 -(q^2+1)a_2 + (q^2+ q)a_3  \leq 0$ \\ 
&&&&&& $q^2 a_1 -q a_2 + a_3 \leq 0$ \\ 
\hline
\end{tabular}} 

\paragraph{Methodology:} To make the proof concise, we explain once and for all our methodology for checking the upper bound in the table above. Recall that $C^+_w$ is defined inductively as $C^{\EE}_{\Hasse,w}+ \bigcap_{\alpha\in \EE_w} C^{+}_{ws_\alpha}$ (Definition \ref{def-intersumcone}). Furthermore, the cone $C^{\EE}_{\Hasse,w}$ is generated by the weights $\{h_w(\chi_\alpha)\}_{\alpha\in \EE_w}$. Therefore, it suffices to check that
\begin{enumerate}[(a)]
    \item \label{pro1} each $h_{w}(\chi_\alpha)$ satisfies the corresponding upper bound,
    \item \label{pro2} the cone $\bigcap_{\alpha\in \EE_w} C^{+}_{ws_\alpha}$ also satisfies this upper bound.
\end{enumerate}
The verification of \eqref{pro1} is easy, so we leave it to the reader. We explain how we prove \eqref{pro2}. Inductively, any $\lambda=(a_1,a_2,a_3)\in C^+_w$ satisfies inequalities
\begin{equation}
I_{w,1}(\lambda)\leq 0, \quad I_{w,2}(\lambda)\leq 0, \quad  \dots\quad  I_{w,n_w}(\lambda)\leq 0   
\end{equation}
where $I_{w,i}(\lambda)$ is a certain linear expression in $\lambda$. By convention, we label the $I_{w,i}(\lambda)$ in the order that they appear from top to bottom in each table (for example, $I_{[135],1}(\lambda)$ is the top expression in the row corresponding to $w=[135]$). With this notation, for each $w$, the cone $\bigcap_{\alpha\in \EE_w} C^{+}_{ws_\alpha}$ satisfies (by induction) the inequalities
\begin{equation}
    I_{ws_\alpha,i}(\lambda)\leq 0 \quad \textrm{for all  $\alpha\in \EE_w$ and all $1\leq i\leq n_w$.}
\end{equation}
In order to show that $\bigcap_{\alpha\in \EE_w} C^{+}_{ws_\alpha}$ satisfies a certain inequality $I_{w,j}(\lambda)\leq 0$, it suffices to find non-negative coefficients $\{ A^{(j)}_{w,\alpha,i} \}_{\alpha\in \EE_w, i}$ such that
\begin{equation} \label{Iwj}
    I_{w,j}(\lambda) = \sum_{\alpha\in \EE_w}\sum_{i=1}^{n_w} A^{(j)}_{w,\alpha,i} I_{ws_\alpha,i}(\lambda).
\end{equation}
In each case, we will simply write down such non-negative coefficients $\{ A^{(j)}_{w,\alpha,i} \}_{\alpha\in \EE_w, i}$. Furthermore, we will only indicate the nonzero coefficients. With this convention, the upper bounds appearing in the above table can be checked by taking:
\begin{align*}
A^{(1)}_{[135],e_2+e_3,1}=\frac{2q^2}{q - 1}, \ A^{(1)}_{[135],2e_3,1}=\frac{q^2+1}{q - 1}, \ 
A^{(2)}_{[135],e_2+e_3,1}=\frac{q^2-q+1}{q - 1}, \  A^{(2)}_{[135],2e_3,1}=\frac{1}{q - 1} \\
A^{(1)}_{[142],e_2-e_3,1}=\frac{2}{q - 1}, \ A^{(1)}_{[142],2e_2,1}=\frac{q^2+1}{q - 1}, \ 
A^{(2)}_{[142],e_2-e_3,1}=\frac{q^2-q+1}{q - 1}, \  A^{(2)}_{[142],2e_2,1}=\frac{q^2}{q - 1}\\
A^{(1)}_{[214],e_1-e_2,1}=\frac{2q}{q - 1}, \ A^{(1)}_{[214],2e_3,1}=\frac{q^2+1}{q - 1}, \ 
A^{(2)}_{[214],e_1-e_2,1}=\frac{q^2-q+1}{q - 1}, \  A^{(2)}_{[214],2e_3,1}=\frac{q}{q - 1}.
\end{align*}

We now move on to elements of length $3$.
\bigskip

\hspace{-2cm}
{\renewcommand{\arraystretch}{1.2}%
\begin{tabular}{ |c|c|c|c|c|c|l| } 
\hline 
\multicolumn{7}{|l|}{Relevant elements of length $3$} \\
\hline
$w$ & $E_w$ & $\EE_w$ & $\LL_w$ &$\{\chi_\alpha\}_{\alpha\in \EE_w}$ & $\{h_w(\chi_\alpha)\}_{\alpha\in \EE_w}$   & Upper bound for $C^{+}_w$\\
\hline
\multirow{2}{2em}{$[145]$} & \multirow{2}{3em}{$2e_2$ \\ $2e_3$} & \multirow{2}{3em}{$2e_2$} & \multirow{2}{2em}{$[135]$}  &\multirow{2}{3em}{$(0,1,0)$}& \multirow{2}{4em}{$(0,-q,1)$} & $q^2 a_1 +a_2 - qa_3 \leq 0$ \\ 
&&&&&&   \\ 
\hline
\multirow{2}{2em}{$[153]$} &\multirow{2}{3em}{$e_2-e_3$ \\ $e_2+e_3$}& \multirow{2}{3em}{$e_2-e_3$ \\ $e_2+e_3$} & \multirow{2}{2em}{$[135]$ \\ $[142]$} &\multirow{2}{4em}{$(0,1,-1)$ \\ $(0,1,1)$}&\multirow{2}{7em}{$(q,1-q,1)$ \\ $(-q,1-q,-1)$}& $-q a_1 + (q+1)a_2  +  a_3 \leq 0$  \\ 
&&&&&& \\ 
\hline
\multirow{3}{2em}{$[236]$} &\multirow{3}{3em}{$e_1+e_3$ \\ $e_2+e_3$\\ $2e_3$}&\multirow{3}{3em}{$e_1+e_3$ \\ $e_2+e_3$}  & \multirow{3}{2em}{$[135]$ \\  $[214]$}&\multirow{3}{3em}{$(1,0,0)$\\$(0,1,0)$}& \multirow{3}{4em}{$(0,-1,-q)$\\$(0,-q,-1)$}& $ a_1\leq 0$ \\ 
&&&&&&   \\ 
&&&&&&   \\ 
\hline
\multirow{3}{2em}{$[315]$} &\multirow{3}{3em}{$e_1-e_2$ \\ $e_1+e_3$\\ $2e_3$}& \multirow{3}{3em}{$e_1+e_3$} & \multirow{3}{2em}{$[214]$} &\multirow{3}{3em}{$(1,1,0)$}& \multirow{3}{8em}{$(-1,-q,-(q+1))$} & $ (q+1)a_1 - (q^2+1)a_2 + (q^2+q)a_3\leq 0$ \\ 
&&&& && \\ 
&&&& &&   \\ 
\hline
\multirow{3}{2em}{$[412]$} &\multirow{3}{3em}{$e_1-e_2$\\$e_1-e_3$\\$2e_1$}& \multirow{3}{3em}{$e_1-e_2$\\$e_1-e_3$}  & \multirow{3}{2em}{$[142]$  \\ $[214]$} &\multirow{3}{4em}{$(0,-1,0)$\\$(0,0,-1)$}& \multirow{3}{3em}{$(1,q,0)$\\$(q,1,0)$}& $-qa_1 + a_2 + (q+1)a_3\leq 0$ \\ 
&&&& & & $a_1 - qa_2 + (q+1)a_3\leq 0$\\
&&&& & & \\
\hline

\end{tabular}} 

\paragraph{Proof :} Take
$A^{(1)}_{[145],2e_2,2}=1$,
$A^{(1)}_{[153],e_2-e_3,2}=\frac{2q}{q^3+2q+1}$,  $A^{(1)}_{[153],e_2+e_3,2}=\frac{q^2+q+1}{q^3+2q+1}$,
$A^{(1)}_{[236],e_1+e_3,1}=\frac{q^2}{q^2 + q + 1}$, $A^{(1)}_{[236],e_1+e_3,2}=\frac{(q^2 + 1)(q + 1)}{q^2 + q + 1}$,  $A^{(1)}_{[236],e_2+e_3,1}=1$,
$A^{(1)}_{[315],e_1+e_3,1}=1$, 
$A^{(1)}_{[412],e_1-e_2,1}=\frac{q^2+q+1}{q^3+q^2+q+1}$,  $A^{(1)}_{[412],e_1-e_3,1}=\frac{q^2}{q^3+q^2+q+1}$, 
$A^{(2)}_{[412],e_1-e_2,1}=\frac{1}{q^3+q^2+q+1}$,  $A^{(2)}_{[412],e_1-e_3,1}=\frac{q^2+q+1}{q^3+q^2+q+1}$.

\bigskip

\hspace{-2.5cm}
{\renewcommand{\arraystretch}{1.2}%
\begin{tabular}{ |c|c|c|c|c|c|l| } 
\hline 
\multicolumn{7}{|l|}{Relevant elements of length $4$} \\
\hline
$w$ & $E_w$ & $\EE_w$ & $\LL_w$ &$\{\chi_\alpha\}_{\alpha\in \EE_w}$ & $\{h_w(\chi_\alpha)\}_{\alpha\in \EE_w}$   & Upper bound for $C^{+}_w$\\
\hline
\multirow{2}{2em}{$[154]$} & \multirow{2}{3em}{$e_2-e_3$ \\ $2e_3$} & \multirow{2}{3em}{$e_2-e_3$ \\ $2e_3$} & \multirow{2}{2em}{$[145]$ \\ $[153]$}  &\multirow{2}{3em}{$(0,1,0)$\\$(0,1,1)$}& \multirow{2}{7em}{$(0,1-q,0)$\\$(-q,1-q,1)$} & $-(q^2-q)a_1 + (q^2+1)a_2 + (q-1)a_3\leq 0$ \\ 
&&&&&&$q^2a_1 + a_2 - qa_3 \leq 0$   \\ 
\hline
\multirow{3}{2em}{$[246]$} & \multirow{3}{3em}{$e_1+e_3$ \\ $2e_2$\\$2e_3$} & \multirow{3}{3em}{$2e_2$} & \multirow{3}{2em}{$[236]$}  &\multirow{3}{3em}{$(0,1,0)$}& \multirow{3}{4em}{$(0,-q,1)$} & $a_1 \leq 0$ \\ 
&&&&&&   \\ 
&&&&&&   \\ 
\hline
\multirow{3}{2em}{$[263]$} & \multirow{3}{3em}{$e_1+e_2$ \\ $e_2-e_3$\\$e_2+e_3$} & \multirow{3}{3em}{$e_1+e_2$} & \multirow{3}{2em}{$[153]$}  &\multirow{3}{3em}{$(1,0,0)$}& \multirow{3}{4em}{$(0,-1,-q)$} & $-qa_1+(q+1)a_2+a_3 \leq 0$ \\ 
&&&&&&   \\ 
&&&&&&   \\ 
\hline
\multirow{3}{2em}{$[326]$} & \multirow{3}{3em}{$e_1-e_2$ \\ $e_2+e_3$\\$2e_3$} & \multirow{3}{3em}{$e_2+e_3$} & \multirow{3}{2em}{$[315]$}  &\multirow{3}{3em}{$(1,1,0)$}& \multirow{3}{10em}{$(0,-(q+1),-(q+1))$} & $ (q+1)a_1 - (q^2+1)a_2 + (q^2+q)a_3\leq 0$\\ 
&&&&&&   \\ 
&&&&&&   \\ 
\hline
\multirow{3}{2em}{$[421]$} & \multirow{3}{3em}{$e_1-e_2$ \\ $2e_1$\\ $e_2-e_3$} & \multirow{3}{3em}{$e_2-e_3$} & \multirow{3}{2em}{$[412]$}  &\multirow{3}{4em}{$(0,0,-1)$}& \multirow{3}{7em}{$(q+1,0,0)$} & $-(q-1)a_2 +(q+1)a_3\leq 0$\\ 
&&&&&&   \\ 
&&&&&&   \\ 
\hline
\end{tabular}} 

\paragraph{Proof :} Take $A^{(1)}_{[154],e_2-e_3,1}=\frac{2}{q^2+q+1} , \ A^{(1)}_{[154],2e_3,1}=\frac{q^3+2q-1}{q^2+q+1}$, $A^{(2)}_{[154],e_2-e_3,1}=A^{(1)}_{[246],2e_2,1}=A^{(1)}_{[263],e_1+e_2,1}=A^{(1)}_{[326],e_2+e_3,1}=1$, $A^{(1)}_{[421],e_2-e_3,1}=\frac{1}{q+1}$ and $A^{(1)}_{[421],e_2-e_3,2}=\frac{q}{q+1}$.

\bigskip

\hspace{-2.6cm}
{\renewcommand{\arraystretch}{1.2}%
\begin{tabular}{ |c|c|c|c|c|c|l| } 
\hline 
\multicolumn{7}{|l|}{Relevant elements of length $5$} \\
\hline
$w$ & $E_w$ & $\EE_w$ & $\LL_w$ &$\{\chi_\alpha\}_{\alpha\in \EE_w}$ & $\{h_w(\chi_\alpha)\}_{\alpha\in \EE_w}$   & Upper bound for $C^{+}_w$\\
\hline
\multirow{3}{2em}{$[264]$} & \multirow{3}{3em}{$e_1+e_2$ \\ $e_2-e_3$\\$2e_3$} & \multirow{3}{3em}{$e_1+e_2$ \\ $e_2-e_3$} & \multirow{3}{2em}{$[154]$ \\ $[246]$}  &\multirow{3}{3em}{$(1,0,0)$\\$\!(-1,1,0)$}& \multirow{3}{7em}{$(0,-1,-q)$\\$(1,1-q,q)$} &  $-(q^2-q)a_1 + (q^2+1)a_2 + (q-1)a_3\leq 0$ \\ 
&&&&&&$q^2a_1 + a_2 - qa_3 \leq 0$   \\ 
&&&&&&   \\ 
\hline
\multirow{3}{2em}{$[362]$} & \multirow{3}{3em}{$e_1-e_3$ \\ $e_2-e_3$\\$e_2+e_3$} & \multirow{3}{3em}{$e_1-e_3$\\$e_2-e_3$} & \multirow{3}{2em}{$[263]$\\$[326]$}  &\multirow{3}{4em}{$(1,0,0)$\\$\! \!(-1,1,-1)$}& \multirow{3}{8em}{$(0,0,-(q+1))$\\$(q+1,1-q,q+1)$} &$-(q^4+q^2+q+1)a_1 + (2q^3+3q^2+2q+1)a_2$ \\ 
&&&&&& $ \qquad + (q^4+2q^3+q)a_3 \leq 0$     \\ 
&&&&&&   \\ 
\hline
\multirow{3}{2em}{$[531]$} & \multirow{3}{3em}{$e_1-e_2$ \\ $e_1+e_2$\\$e_2-e_3$} & \multirow{3}{3em}{$e_1+e_2$} & \multirow{3}{2em}{$[421]$}  &\multirow{3}{3em}{$(1,1,1)$}& \multirow{3}{7em}{$(-(q+1),1-q,$\\$-(q+1))$} & $ -(q-1)a_2 + (q+1)a_3\leq 0$ \\ 
&&&&&&   \\ 
&&&&&&   \\ 
\hline
\end{tabular}} 

\paragraph{Proof :} Take
$A^{(1)}_{[264],e_1+e_2,1}=1$, 
$A^{(2)}_{[264],e_1+e_2,1}=\frac{q^2-1}{q^3+2q-1}$,  $A^{(2)}_{[264],e_1+e_2,2}=\frac{2q^2-q+1}{q^3+2q-1}$,  $A^{(2)}_{[264],e_2-e_3,1}=q^2-q$,
$A^{(1)}_{[362],e_1-e_3,1}=q^3+q^2+2q$,  $A^{(1)}_{[362],e_2-e_3,1}=q^2-1$,
$A^{(1)}_{[531],e_1+e_2,1}=1$.

\bigskip

\hspace{-2.3cm}
{\renewcommand{\arraystretch}{1.2}%
\begin{tabular}{ |c|c|c|c|c|c|l| }
\hline 
\multicolumn{7}{|l|}{Relevant elements of length $6$} \\
\hline
$w$ & $E_w$ & $\EE_w$ & $\LL_w$ &$\{\chi_\alpha\}_{\alpha\in \EE_w}$ & $\{h_w(\chi_\alpha)\}_{\alpha\in \EE_w}$   & Upper bound for $C^{+}_w$\\
\hline
\multirow{3}{2em}{$[365]$} & \multirow{3}{3em}{$e_1+e_3$ \\ $e_2-e_3$\\$2e_3$} & \multirow{3}{3em}{$e_1+e_3$ \\ $2e_3$} & \multirow{3}{2em}{$[264]$ \\ $[362]$}  &\multirow{3}{4em}{$(1,0,0)$\\$(-1,1,1)$}& \multirow{3}{7em}{$(0,0,-(q+1))$\\$(1-q,1-q,q+1)$} &  $ -(q^2-q)a_1 + (q^2+1)a_2+ (q-1)a_3\leq 0$\\ 
&&&&&& $a_1+\theta(q)a_2\leq 0$ (see below)   \\ 
&&&&&&   \\ 
\hline
\multirow{3}{2em}{$[541]$} & \multirow{3}{3em}{$e_1-e_2$ \\ $e_2-e_3$\\$2e_2$} & \multirow{3}{3em}{$2e_2$} & \multirow{3}{2em}{$[531]$}  &\multirow{3}{4em}{$(1,1,1)$}& \multirow{3}{8em}{$(-(q+1),1-q, $\\ $1-q)$} &$-(q-1)a_2 + (q+1)a_3\leq 0$ \\ 
&&&&&&     \\ 
&&&&&&   \\ 
\hline

\end{tabular}} 

\bigskip
In the equation for the stratum $w=[365]$, the function $\theta(q)$ is defined as follows:
\begin{equation}\label{theta-fun}
    \theta(q)=\frac{q^5+2q^4+2q^3+4q^2+2q+1}{q^6+2q^5-q^4+q^3-q^2-q-1}.
\end{equation}

\paragraph{Proof :} $A^{(1)}_{[365],e_1+e_3,1}=1, \  A^{(2)}_{[365],e_1+e_3,2}=q^4+2q^3+q, \ A^{(2)}_{[365],2e_3,1}=1$ and $A^{(1)}_{[541],2e_2,1}=1$.

\bigskip

\hspace{-1.4cm}
{\renewcommand{\arraystretch}{1.2}%
\begin{tabular}{ |c|c|c|c|c|c|l| } 
\hline 
\multicolumn{7}{|l|}{Relevant elements of length $7$} \\
\hline
$w$ & $E_w$ & $\EE_w$ & $\LL_w$ &$\{\chi_\alpha\}_{\alpha\in \EE_w}$ & $\{h_w(\chi_\alpha)\}_{\alpha\in \EE_w}$   & Upper bound for $C^{+}_w$\\
\hline
\multirow{3}{2em}{$[465]$} & \multirow{3}{3em}{$2e_1$ \\ $e_2-e_3$\\$2e_3$} & \multirow{3}{2em}{$2e_1$} & \multirow{3}{2em}{$[365]$}  &\multirow{3}{3em}{$(1,0,0)$}& \multirow{3}{5em}{$(0,0,1-q)$} &  $ -(q^2-q)a_1 + (q^2+1)a_2 + (q-1)a_3\leq 0$\\ 
&&&&&& $a_1+\theta(q)a_2\leq 0$   \\ 
&&&&&&   \\ 
\hline
\multirow{3}{2em}{$[546]$} & \multirow{3}{3em}{$e_1-e_2$ \\ $2e_2$\\$2e_3$} & \multirow{3}{2em}{$2e_3$} & \multirow{3}{2em}{$[541]$}  &\multirow{3}{3em}{$(0,0,1)$}& \multirow{3}{5em}{$(1-q,0,0)$} &$-(q-1)a_2 + (q+1)a_3\leq 0$ \\ 
&&&&&&     \\ 
&&&&&&   \\ 
\hline
\end{tabular}} 

\paragraph{Proof :} Take $A^{(1)}_{[465],2e_1,1}=A^{(2)}_{[465],2e_1,2}=A^{(1)}_{[546],2e_3,1}=1$.

\bigskip

{\renewcommand{\arraystretch}{1.2}%
\noindent\begin{tabular}{ |c|c|c|c|c|c|l| } 
\hline 
\multicolumn{7}{|l|}{Relevant elements of length $8$} \\
\hline
$w$ & $E_w$ & $\EE_w$ & $\LL_w$ &$\{\chi_\alpha\}_{\alpha\in \EE_w}$ & $\{h_w(\chi_\alpha)\}_{\alpha\in \EE_w}$   & Upper bound for $C^{+}_w$\\
\hline
\multirow{3}{2em}{$[564]$} & \multirow{3}{3em}{$e_1-e_3$ \\ $e_2-e_3$\\$2e_3$} & \multirow{3}{3em}{$e_1-e_3$\\$e_2-e_3$} & \multirow{3}{2em}{$[465]$\\$[546]$}  &\multirow{3}{3em}{$(1,0,0)$\\$(0,1,0)$}& \multirow{3}{7em}{$(0,1,-q)$\\$(1,-q,0)$} &   $q^2a_1 + qa_2+a_3\leq 0$\\ 
&&&&&&   \\ 
&&&&&&   \\ 
\hline
\end{tabular}} 

\paragraph{Proof :} Define
\begin{equation}
    u=\frac{q^7 - 2q^6 - 9q^5 - 4q^4 - 7q^3 - 3q^2 - q + 1}{(q-1)(q^3 + 2q^2 + 1) (q^5 + 4q^4 + 2q^3 + 5q^2 + 4q + 2)}.
\end{equation}
For $q\geq 5$, one checks easily that $0\leq u\leq \frac{1}{q-1}$. We may take:
\begin{align*}
&A^{(1)}_{[564],e_1-e_3,1}=u, \ A^{(1)}_{[564],e_1-e_3,2}=q^2+q(q-1)u, \ A^{(1)}_{[564],e_2-e_3,1}= \frac{1-(q-1)u}{q+1}.
\end{align*}
For $q\geq 5$, these numbers are non-negative. This shows that $C^+_{[564]}$ is contained in the half-space $q^2a_1 + qa_2+a_3\leq 0$, which completes the proof of Lemma \ref{lem-564}.

\subsection{The case $G=\GL_{4,\FF_q}$ and $(r,s)=(3,1)$} \label{proof-GL31}
It suffices to show Proposition \ref{prop-GL31-van}. We use the same method as for the case $G=\Sp(6)$. Recall that the determinant $\GL_{4}\to \GG_{\textrm{m}}$ is an invertible section of the line bundle $\Vcal_I(\lambda_{\det})$ on $\GZip^\mu$, where $\lambda_{\det}=(1-q,1-q,1-q,1-q)$. All cones $\langle C^+_w \rangle$, $\langle C_{\zip}\rangle$, etc. contain $\ZZ \lambda_{\det}$. Moreover, there is a bijection between
\begin{equation}
\left\{ \textrm{saturated cones of $\ZZ^4$ containing $\ZZ \lambda_{\det}$}\right\} \longleftrightarrow \left\{ \textrm{saturated cones of $\ZZ^{3}$ }\right\} 
\end{equation}
defined by $C\mapsto \overline{C}$, with notation as in \eqref{Xbar}. We will implicitly make this simplification and consider the cones $\overline{C}\subset \ZZ^3$. This amounts to setting $a_4=0$ in all equations. We first explain the element of $W$ that appear in the proof.

\bigskip

\begin{figure}[H]
$\xymatrix@R=10pt{
&   *++[F]{[1243]} \ar@{-}[rd] &&&&\\
  *++[F]{[1234]}\ar@{-}[r]\ar@{-}[ru]\ar@{-}[rd] & *++[F]{[1324]}\ar@{-}[rd] &*++[F]{[2143]}\ar@{-}[r]&*++[F]{[4123]}\ar@{-}[r]&*++[F]{[4132]} \ar@{-}[r]&*++[F]{[4312]}\\
&   *++[F]{[2134]}\ar@{-}[r]\ar@{-}[ru]  &*++[F]{[3124]} \ar@{-}[ru]&&& 
}$
    \caption{The strata appearing in the proof}
    \label{fig:GU31}
\end{figure}
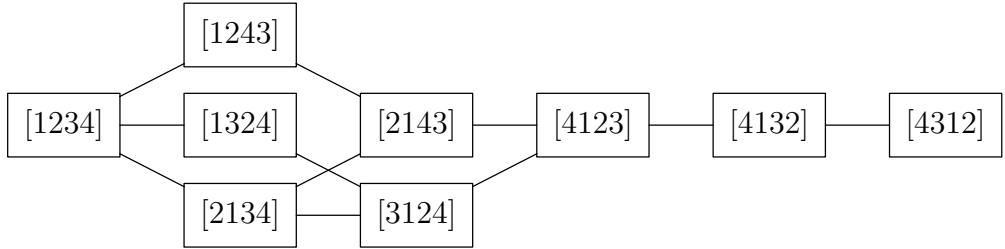

We start by giving the Hasse cones $\overline{C}_{\Hasse,w}$ for the three elements of $W$ of length $1$.

\bigskip

{\renewcommand{\arraystretch}{1.5}%
\noindent\begin{tabular}{ |c|c|c|l|c| } 
\hline
$w$  & $E_w $  &  $\overline{C}_{\Hasse,w}$ & $\overline{C}^{+}_w$  \\
\hline
$[1243]$ & $e_3-e_4$ & $-q a_1-q^2a_2+ (q^2+q+1)a_3 \leq 0$ & same \\
\hline
$[1324]$ & $e_2-e_3$ & $-q^2 a_1+(q^2+q+1)a_2-a_3 \leq 0$ & same \\
\hline
$[2134]$ & $e_1-e_2$& $(q^2+q+1)a_1-a_2-qa_3 \leq 0$ & same \\
\hline
\end{tabular}} 

\bigskip 

{\renewcommand{\arraystretch}{1.2}%
\noindent\begin{tabular}{ |c|c|c|c|c|c|l| } 
\hline 
\multicolumn{7}{|l|}{Relevant elements of length $2$} \\
\hline
$w$ & $E_w$ & $\EE_w$ & $\LL_w$ &$\{\chi_\alpha\}_{\alpha\in \EE_w}$ & $\{h_w(\chi_\alpha)\}_{\alpha\in \EE_w}$   & Upper bound for $\overline{C}^{+}_w$\\
\hline
\multirow{2}{3em}{$[2143]$} &\multirow{2}{3em}{$e_1-e_2$ \\ $e_3-e_4$}&\multirow{2}{3em}{$e_1-e_2$ \\$e_3-e_4$}  & \multirow{2}{3em}{$[1243]$\\$[2134]$}&\multirow{2}{3em}{$(1,0,0)$\\$(0,0,1)$}& \multirow{2}{8em}{$(-q,-(q+1),-q)$\\$(1,q+1,1)$}& $-qa_2+(q+1)a_3\leq 0$ \\ 
&&&&&& $(q+1)a_1-a_2\leq 0$  \\ 
\hline
\multirow{2}{3em}{$[3124]$} &\multirow{2}{3em}{$e_1-e_2$\\$e_1-e_3$}& \multirow{2}{3em}{$e_1-e_2$\\$e_1-e_3$}  & \multirow{2}{3em}{$[1324]$  \\ $[2134]$} &\multirow{2}{4em}{$(0,-1,0)$\\$(1,1,0)$}& \multirow{2}{8em}{$(1-q,0,0)$\\$(-1,-q,-(q+1))$}& $(q+1)a_2-a_3\leq 0$ \\ 
&&&& & &  \\
\hline

\end{tabular}}

\paragraph{Proof :} Take
$A^{(1)}_{[2143],e_1-e_2,1}=\frac{q^2+q+1}{q^3+q^2+q+1}$,  $A^{(1)}_{[2143],e_3-e_4,1}=\frac{q}{q^3+q^2+q+1}$, 
$A^{(2)}_{[2143],e_1-e_2,1}=\frac{q}{q^3+q^2+q+1}$,   $A^{(2)}_{[2143],e_3-e_4,1}=\frac{q^2+q+1}{q^3+q^2+q+1}$,
$A^{(1)}_{[3124],e_1-e_2,1}=\frac{q^2+q+1}{q^3+q^2+q+1}$,   $A^{(1)}_{[3124],e_1-e_3,1}=\frac{q^2}{q^3+q^2+q+1}$.

\bigskip 

{\renewcommand{\arraystretch}{1.2}%
\noindent\begin{tabular}{ |c|c|c|c|c|c|l| } 
\hline 
\multicolumn{7}{|l|}{Relevant elements of length $3$} \\
\hline
$w$ & $E_w$ & $\EE_w$ & $\LL_w$ &$\{\chi_\alpha\}_{\alpha\in \EE_w}$ & $\{h_w(\chi_\alpha)\}_{\alpha\in \EE_w}$   & Upper bound for $\overline{C}^{+}_w$\\
\hline
\multirow{3}{3em}{$[4123]$} & \multirow{3}{3em}{$e_1-e_2$\\$e_1-e_3$\\$e_1-e_4$} & \multirow{3}{3em}{$e_1-e_3$\\$e_1-e_4$} & \multirow{3}{3em}{$[2143]$\\$[3124]$}  &\multirow{3}{3em}{$(0,0,-1)$\\$(1,1,1)$}& \multirow{3}{7em}{$(0,1-q,0)$\\$(0,0,1-q)$} & $a_1\leq 0$ \\ &&&&&&  $a_2\leq 0$  \\ 
&&&&&&   $a_3\leq 0$ \\ 
\hline
\end{tabular}} 

\paragraph{Proof :} Take
$A^{(1)}_{[4123],e_1-e_3,1}=\frac{1}{q^3+2q^2+2q+1}$, 
$A^{(1)}_{[4123],e_1-e_3,2}=\frac{1}{q+1}$, 
$A^{(1)}_{[4123],e_1-e_4,1}=\frac{1}{q^2+q+1}$,  
$A^{(2)}_{[4123],e_1-e_3,1}=\frac{1}{q^2+q+1}$,  
$A^{(2)}_{[4123],e_1-e_4,1}=\frac{q+1}{q^2+q+1}$, 
$A^{(3)}_{[4123],e_1-e_3,1}=\frac{q+1}{q^2+q+1}$,  
$A^{(3)}_{[4123],e_1-e_4,1}=\frac{q}{q^2+q+1}$.

\bigskip

{\renewcommand{\arraystretch}{1.2}%
\noindent\begin{tabular}{ |c|c|c|c|c|c|l| } 
\hline 
\multicolumn{7}{|l|}{Relevant elements of length $4$} \\
\hline
$w$ & $E_w$ & $\EE_w$ & $\LL_w$ &$\{\chi_\alpha\}_{\alpha\in \EE_w}$ & $\{h_w(\chi_\alpha)\}_{\alpha\in \EE_w}$   & Upper bound for $\overline{C}^{+}_w$\\
\hline
\multirow{3}{3em}{$[4132]$} & \multirow{3}{3em}{$e_1-e_2$\\$e_1-e_3$\\$e_3-e_4$} & \multirow{3}{3em}{$e_3-e_4$} & \multirow{3}{3em}{$[4123]$}  &\multirow{3}{3em}{$(1,1,1)$}& \multirow{3}{5em}{$(0,1,-q)$} & $qa_2+a_3\leq 0$ \\ 
&&&&&& $a_1\leq 0$  \\ 
&&&&&&   \\ 
\hline
\end{tabular}} 

\paragraph{Proof :} Take $A^{(1)}_{[4132],e_3-e_4,2}=q, \ A^{(1)}_{[4132],e_3-e_4,3}=1$ and $A^{(2)}_{[4132],e_3-e_4,1}=1$.

\bigskip

{\renewcommand{\arraystretch}{1.2}%
\noindent\begin{tabular}{ |c|c|c|c|c|c|l| }
\hline 
\multicolumn{7}{|l|}{Relevant elements of length $5$} \\
\hline
$w$ & $E_w$ & $\EE_w$ & $\LL_w$ &$\{\chi_\alpha\}_{\alpha\in \EE_w}$ & $\{h_w(\chi_\alpha)\}_{\alpha\in \EE_w}$   & Upper bound for $\overline{C}^{+}_w$\\
\hline
\multirow{3}{3em}{$[4312]$} & \multirow{3}{3em}{$e_1-e_2$ \\ $e_2-e_3$\\$e_2-e_4$} & \multirow{3}{3em}{$e_2-e_3$} & \multirow{3}{3em}{$[4132]$}  &\multirow{3}{4em}{$(0,0,-1)$}& \multirow{3}{5em}{$(1,-q,0)$} & $q^2a_1+qa_2+a_3 \leq 0$ \\ 
&&&&&&  \\ 
&&&&&&  \\ 
\hline
\end{tabular}} 

\paragraph{Proof :} Take $A^{(1)}_{[4312],e_2-e_3,1}=1, \ A^{(1)}_{[4312],e_2-e_3,2}=q^2$.

\bigskip

\subsection{The case $G=\GL_{4,\FF_q}$ and $(r,s)=(2,2)$} \label{proof-GL22}
We retain the same conventions as explained in \S\ref{proof-GL31}. In particular, we consider the image of each cone by the map $\lambda\mapsto \overline{\lambda}$. Here are the elements of $W$ which appear in the proof.

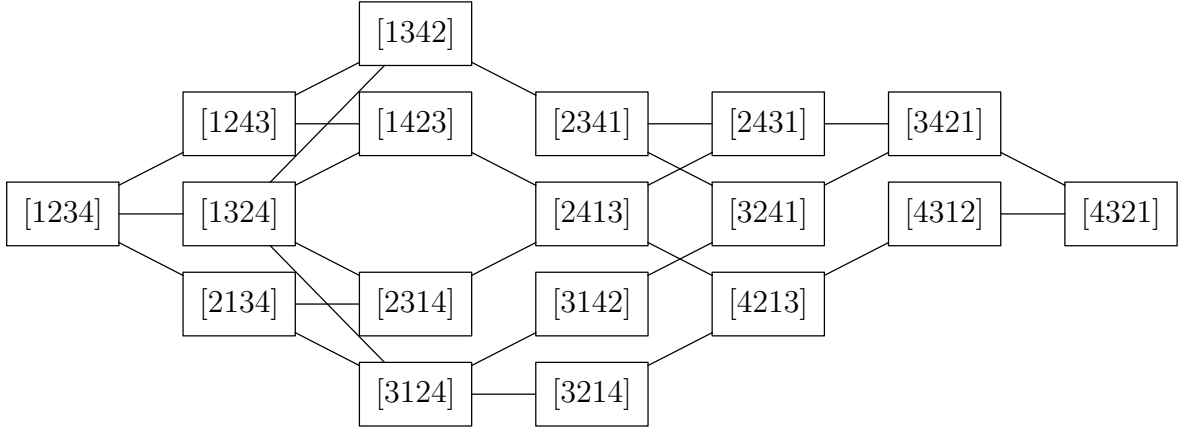
\begin{figure}[H]

$\xymatrix@R=10pt{
&& *++[F]{[1342]}\ar@{-}[rd] &&&& \\
&   *++[F]{[1243]} \ar@{-}[ru] \ar@{-}[r] &*++[F]{[1423]} \ar@{-}[rd]&*++[F]{[2341]}\ar@{-}[rd] \ar@{-}[r]&*++[F]{[2431]}\ar@{-}[r]&*++[F]{[3421]}\ar@{-}[rd]&\\
  *++[F]{[1234]}\ar@{-}[r]\ar@{-}[ru]\ar@{-}[rd] & *++[F]{[1324]}\ar@{-}[rd]\ar@{-}[rdd]\ar@{-}[ru]\ar@{-}[ruu] & &*++[F]{[2413]}\ar@{-}[ru]\ar@{-}[rd]&*++[F]{[3241]} \ar@{-}[ru]&*++[F]{[4312]}\ar@{-}[r]&*++[F]{[4321]}\\
&   *++[F]{[2134]}\ar@{-}[r] \ar@{-}[rd] &*++[F]{[2314]} \ar@{-}[ru]&*++[F]{[3142]}\ar@{-}[ru]&*++[F]{[4213]}\ar@{-}[ru]&& \\
&&*++[F]{[3124]}\ar@{-}[ru]\ar@{-}[r]&*++[F]{[3214]} \ar@{-}[ru]&&& \\
}$
    \caption{The strata appearing in the proof}
    \label{fig:GU22}
\end{figure}

We start with the Hasse cones $\overline{C}_{\Hasse,w}$ for the three elements of length $1$.

\bigskip

{\renewcommand{\arraystretch}{1.5}%
\noindent\begin{tabular}{ |c|c|c|l|c| } 
\hline
$w$  & $E_w $  &  $\overline{C}_{\Hasse,w}$ & $\overline{C}^+_w$  \\
\hline
$[1243]$ & $e_3-e_4$ & $-q a_1+qa_2-a_3 \leq 0$ & same \\
\hline
$[1324]$ & $e_2-e_3$ & $q a_1-a_2+a_3 \leq 0$ & same \\
\hline
$[2134]$ & $e_1-e_2$& $-a_1+a_2-qa_3 \leq 0$ & same \\
\hline
\end{tabular}} 

\bigskip

\hspace{-1.7cm}
{\renewcommand{\arraystretch}{1.2}%
\begin{tabular}{ |c|c|c|c|c|c|l| } 
\hline 
\multicolumn{7}{|l|}{Relevant elements of length $2$} \\
\hline
$w$ & $E_w$ & $\EE_w$ & $\LL_w$ &$\{\chi_\alpha\}_{\alpha\in \EE_w}$ & $\{h_w(\chi_\alpha)\}_{\alpha\in \EE_w}$   & Upper bound for $\overline{C}^{+}_w$\\
\hline
\multirow{2}{3em}{$[1342]$} & \multirow{2}{3em}{$e_2-e_4$ \\ $e_3-e_4$} & \multirow{2}{3em}{$e_2-e_4$ \\ $e_3-e_4$} & \multirow{2}{3em}{$[1243]$\\$[1324]$}  &\multirow{2}{3em}{$(0,1,0)$\\$(0,0,1)$}& \multirow{2}{8em}{$(-q,-q,-(q+1))$\\$(q+1,1,1)$} & $-qa_1+(q^2+q+1)a_2-a_3 \leq 0$ \\ 
&&&&&&   \\ 
\hline
\multirow{2}{3em}{$[1423]$} &\multirow{2}{3em}{$e_2-e_3$ \\ $e_2-e_4$}& \multirow{2}{3em}{$e_2-e_3$ \\ $e_2-e_4$} & \multirow{2}{3em}{$[1243]$ \\ $[1324]$} &\multirow{2}{4em}{$(0,0,-1)$ \\ $(0,1,1)$}&\multirow{2}{7em}{$(-q,1,0)$ \\ $(1,-q,1-q)$}& $q^2 a_1 + a_2  + q a_3 \leq 0$  \\ 
&&&&&&$q a_1 + q^2 a_2  +  a_3 \leq 0$ \\ 
\hline
\multirow{2}{3em}{$[2314]$} &\multirow{2}{3em}{$e_1-e_3$ \\ $e_2-e_3$}& \multirow{2}{3em}{$e_1-e_3$ \\ $e_2-e_3$} & \multirow{2}{3em}{$[1324]$\\$[2134]$} &\multirow{2}{3em}{$(1,0,0)$\\$(0,1,0)$}& \multirow{2}{8em}{$(0,-1,q)$\\$(-q,-q,-(q+1))$} & $ (q^2+q+1)a_1-qa_2-a_3\leq 0$ \\ 
&&&& && $(q^2+q+1)a_1-a_2-q^2a_3\leq 0$ \\  
\hline
\multirow{2}{3em}{$[3124]$} &\multirow{2}{3em}{$e_1-e_2$\\$e_1-e_3$}& \multirow{2}{3em}{$e_1-e_2$\\$e_1-e_3$}  & \multirow{2}{3em}{$[1324]$  \\ $[2134]$} &\multirow{2}{4em}{$(0,-1,0)$\\$(1,1,0)$}& \multirow{2}{8em}{$(q+1,q,q)$\\$(-(q+1),-q,-1)$}& $q^2a_1+a_2-(q^2+q+1)a_3\leq 0$ \\ 
&&&& & & $a_1+ q a_2-(q^2+q+1)a_3\leq 0$ \\
\hline

\end{tabular}} 

\paragraph{Proof :} Take
$A^{(1)}_{[1342],e_2-e_4,1}=\frac{q(q+1)}{q-1}$, $A^{(1)}_{[1342],e_3-e_4,1}=\frac{q^2+1}{q-1}$,
$A^{(1)}_{[1423],e_2-e_3,1}=\frac{q+1}{q-1}$,  $A^{(1)}_{[1423],e_2-e_4,1}=\frac{q^2+1}{q-1}$,
$A^{(2)}_{[1423],e_2-e_3,1}=\frac{q^2+1)}{q-1}$, $A^{(2)}_{[1423],e_2-e_4,1}=\frac{q(q+1)}{q-1}$,
$A^{(1)}_{[2314],e_1-e_3,1}=\frac{q^2+1}{q-1}$,  $A^{(1)}_{[2314],e_2-e_3,1}=\frac{q+1}{q-1}$,
$A^{(2)}_{[2314],e_1-e_3,1}=\frac{q(q+1)}{q-1}$,  $A^{(2)}_{[2314],e_2-e_3,1}=\frac{q^2+1}{q-1}$,
$A^{(1)}_{[3124],e_1-e_2,1}=\frac{q^2+1}{q-1}$,  $A^{(1)}_{[3124],e_1-e_3,1}=\frac{q(q+1)}{q-1}$,
$A^{(2)}_{[3124],e_1-e_2,1}=\frac{q+1}{q-1}$,  $A^{(2)}_{[3124],e_1-e_3,1}=\frac{q^2+1}{q-1}$.

\bigskip

\hspace{-2cm}
{\renewcommand{\arraystretch}{1.2}%
\begin{tabular}{ |c|c|c|c|c|c|l| } 
\hline 
\multicolumn{7}{|l|}{Relevant elements of length $3$} \\
\hline
$w$ & $E_w$ & $\EE_w$ & $\LL_w$ &$\{\chi_\alpha\}_{\alpha\in \EE_w}$ & $\{h_w(\chi_\alpha)\}_{\alpha\in \EE_w}$   & Upper bound for $\overline{C}^{+}_w$\\
\hline
\multirow{3}{3em}{$[2341]$} & \multirow{3}{3em}{$e_1-e_4$\\$e_2-e_4$\\$e_3-e_4$} & \multirow{3}{3em}{$e_1-e_4$} & \multirow{3}{3em}{$[1342]$}  &\multirow{3}{3em}{$(1,0,0)$}& \multirow{3}{7em}{$(0,-1,q)$} & $-qa_1 + (q^2+q+1)a_2 -a_3\leq 0$ \\ 
&&&&&&   \\ 
&&&&&&   \\ 
\hline
\multirow{3}{3em}{$[2413]$} & \multirow{3}{3em}{$e_1-e_3$ \\ $e_2-e_3$\\$e_2-e_4$} & \multirow{3}{3em}{$e_1-e_3$\\$e_2-e_4$} & \multirow{3}{3em}{$[1423]$\\$[2314]$}  &\multirow{3}{3em}{$(1,0,0)$\\$(1,1,1)$}& \multirow{3}{4em}{$(0,-1,q)$\\$(0,-q,1)$} & $(2q+1)a_1-a_2-qa_3 \leq 0$ \\ 
&&&&&& $a_1\leq 0$  \\ 
&&&&&& $qa_1+qa_2+a_3\leq 0$  \\ 
\hline
\multirow{3}{3em}{$[3142]$} & \multirow{3}{3em}{$e_1-e_2$\\$e_1-e_4$\\$e_3-e_4$} & \multirow{3}{3em}{$e_3-e_4$} & \multirow{3}{3em}{$[3124]$} &\multirow{3}{3em}{$(0,0,1)$}& \multirow{3}{8em}{$(q+1,1,1)$} & $a_1+a_2-(q+2)a_3 \leq 0$\\ &&&&&&   \\ 
&&&&&&   \\ 
\hline
\multirow{2}{3em}{$[3214]$} & \multirow{2}{3em}{$e_1-e_2$ \\ $e_2-e_3$} & \multirow{2}{3em}{$e_2-e_3$} & \multirow{2}{3em}{$[3124]$}  &\multirow{2}{3em}{$(1,1,0)$}& \multirow{2}{10em}{$(-q,-(q+1),-1)$} & $ a_1+qa_2-(q^2+q+1)a_3\leq 0$\\ 
&&&&&&   \\ 
\hline
\end{tabular}}

\paragraph{Proof :} Take $A^{(1)}_{[2341],e_1-e_4,1}=1$,  
$A^{(1)}_{[2413],e_2-e_4,1}=\frac{q}{q^2+q+1}$,  $A^{(1)}_{[2413],e_2-e_4,2}=\frac{q+1}{q^2+q+1}$,  
$A^{(2)}_{[2413],e_1-e_3,1}=\frac{1}{(q+1)^2}$,  
$A^{(2)}_{[2413],e_2-e_4,1}=\frac{q}{(q+1)(q^3+2q^2+2q+1)}$,  $A^{(2)}_{[2413],e_2-e_4,2}=\frac{1}{q^3+2q^2+2q+1}$,
$A^{(3)}_{[2413],e_1-e_3,1}=\frac{q}{q^2+q+1}$,  $A^{(3)}_{[2413],e_1-e_3,2}=\frac{q+1}{q^2+q+1}$, 
$A^{(1)}_{[3142],e_3-e_4,1}=\frac{1}{q^2+q+1}$,  $A^{(1)}_{[3142],e_3-e_4,2}=\frac{q+1}{q^2+q+1}$,
$A^{(1)}_{[3214],e_2-e_3,2}=1$.

\bigskip

\hspace{-1.5cm}
{\renewcommand{\arraystretch}{1.2}%
\begin{tabular}{ |c|c|c|c|c|c|l| } 
\hline 
\multicolumn{7}{|l|}{Relevant elements of length $4$} \\
\hline
$w$ & $E_w$ & $\EE_w$ & $\LL_w$ &$\{\chi_\alpha\}_{\alpha\in \EE_w}$ & $\{h_w(\chi_\alpha)\}_{\alpha\in \EE_w}$   & Upper bound for $\overline{C}^{+}_w$\\
\hline
\multirow{3}{3em}{$[2431]$} & \multirow{3}{3em}{$e_1-e_4$\\$e_2-e_3$\\$e_3-e_4$} & \multirow{3}{3em}{$e_3-e_4$} & \multirow{3}{3em}{$[2413]$}  &\multirow{3}{3em}{$(0,1,1)$}& \multirow{3}{7em}{$(1,1-q,-q)$} & $q^2a_1 + qa_2 +a_3\leq 0$ \\ 
&&&&&& $qa_1 + qa_2 +a_3\leq 0$  \\ 
&&&&&&   \\ 
\hline
\multirow{3}{3em}{$[3241]$} & \multirow{3}{3em}{$e_1-e_2$ \\ $e_2-e_4$\\$e_3-e_4$} & \multirow{3}{3em}{$e_1-e_2$\\$e_2-e_4$} & \multirow{3}{3em}{$[2341]$\\$[3142]$}  &\multirow{3}{3em}{$(1,0,0)$\\$(1,1,0)$}& \multirow{3}{8em}{$(0,0,q-1)$\\$(-q,-(q+1),-1)$} & $-qa_1+(q^2+q+1)a_2-a_3 \leq 0$ \\ 
&&&&&& $a_1+a_2-(q+2)a_3\leq 0$  \\ 
&&&&&&  \\ 
\hline
\multirow{3}{3em}{$[4213]$} & \multirow{3}{3em}{$e_1-e_2$\\$e_1-e_4$\\$e_2-e_3$} & \multirow{3}{3em}{$e_1-e_2$\\$e_1-e_4$} & \multirow{3}{3em}{$[2413]$\\$[3214]$} &\multirow{3}{4em}{\!\!$(0,-1,-1)$\\$(1,1,1)$}& \multirow{3}{8em}{$(1,q+1,q)$\\$(0,-q,1)$} & $(2q+1)a_1-a_2-qa_3 \leq 0$\\ &&&&&&$a_1+qa_2-(q^2+q+1)a_3\leq 0$   \\ 
&&&&&&   \\ 
\hline
\end{tabular}}

\paragraph{Proof :} 
Take $A^{(1)}_{[2431],e_3-e_4,2}=q^2-q$ and $A^{(1)}_{[2431],e_3-e_4,3}=A^{(2)}_{[2431],e_3-e_4,3}=A^{(1)}_{[3241],e_1-e_2,1}=A^{(2)}_{[3241],e_2-e_4,1}=A^{(1)}_{[4213],e_1-e_2,1}=A^{(2)}_{[4213],e_1-e_4,1}=1$.
\bigskip 

\hspace{-1cm}
{\renewcommand{\arraystretch}{1.2}%
\begin{tabular}{ |c|c|c|c|c|c|l| } 
\hline 
\multicolumn{7}{|l|}{Relevant elements of length $5$} \\
\hline
$w$ & $E_w$ & $\EE_w$ & $\LL_w$ &$\{\chi_\alpha\}_{\alpha\in \EE_w}$ & $\{h_w(\chi_\alpha)\}_{\alpha\in \EE_w}$   & Upper bound for $\overline{C}^{+}_w$\\
\hline
\multirow{3}{3em}{$[3421]$} & \multirow{3}{3em}{$e_1-e_3$\\$e_2-e_3$\\$e_3-e_4$} & \multirow{3}{3em}{$e_1-e_3$\\$e_2-e_3$} & \multirow{3}{3em}{$[2431]$\\$[3241]$}  &\multirow{3}{3em}{$(1,0,0)$\\$(0,1,0)$}& \multirow{3}{7em}{$(0,0,q-1)$\\$(1-q,1-q,1-q)$} & $a_2\leq 0$ \\ 
&&&&&& $a_1 + \epsilon(q)a_2\leq 0$ (see below) \\ 
&&&&&&   \\ 
\hline
\multirow{3}{3em}{$[4312]$} & \multirow{3}{3em}{$e_1-e_2$ \\ $e_2-e_3$\\$e_2-e_4$} & \multirow{3}{3em}{$e_2-e_4$} & \multirow{3}{3em}{$[4213]$}  &\multirow{3}{3em}{$(1,1,1)$}& \multirow{3}{6em}{$(0,1-q,0)$} & $a_1-a_3 \leq 0$ \\ 
&&&&&&  \\ 
&&&&&&  \\ 
\hline
\end{tabular}}

\bigskip

In the equation for $w=[3421]$, the function $\epsilon(q)$ is defined by $\epsilon(q)\colonequals\frac{q^2+2q+1}{q^3+2q^2+1}$.

\paragraph{Proof :} Take
$A^{(1)}_{[3421],e_1-e_3,2}=\frac{1}{(q+1)^2}$,  $A^{(1)}_{[3421],e_2-e_3,1}=\frac{1}{(q+1)^2}$, 
$A^{(2)}_{[3421],e_1-e_3,1}=\frac{q+2}{q^3+2q^2+1}$,  $A^{(2)}_{[3421],e_2-e_3,2}=\frac{1}{q^3+2q^2+1}$,
$A^{(1)}_{[4312],e_2-e_4,1}=\frac{q}{2q^2+q+1}$,  $A^{(1)}_{[4312],e_2-e_4,2}=\frac{1}{2q^2+q+1}$.

\bigskip

This shows that $C_{S,[3421]}$ satisfies the equations $a_2-a_4\leq 0$ and $a_1-a_4 + \epsilon(q)(a_2-a_4)\leq 0$. Similarly, $C_{S,[4312]}$ satisfies $a_1-a_3\leq 0$. This finishes the proof of Proposition \ref{propGL22}. Finally, we consider the longest element $w_0=[4321]$ and prove Theorem \ref{thmGL22-conj}.

\bigskip 

\hspace{-0.7cm}
{\renewcommand{\arraystretch}{1.2}%
\begin{tabular}{ |c|c|c|c|c|c|l| } 
\hline 
\multicolumn{7}{|l|}{Relevant elements of length $6$} \\
\hline
$w$ & $E_w$ & $\EE_w$ & $\LL_w$ &$\{\chi_\alpha\}_{\alpha\in \EE_w}$ & $\{h_w(\chi_\alpha)\}_{\alpha\in \EE_w}$   & Upper bound for $\overline{C}^{+}_w$\\
\hline
\multirow{3}{3em}{$[4321]$} & \multirow{3}{3em}{$e_1-e_2$\\$e_2-e_3$\\$e_3-e_4$} & \multirow{3}{3em}{$e_1-e_2$\\$e_3-e_4$} & \multirow{3}{3em}{$[3421]$\\$[4312]$}  &\multirow{3}{3em}{$(1,0,0)$\\$(1,1,1)$}& \multirow{3}{7em}{$(1,1,q+1)$\\$(1,-q,0)$} & $ qa_1+a_2-a_3\leq 0$ \\ 
&&&&&&  \\ 
&&&&&&   \\ 
\hline
\end{tabular}} 

\paragraph{Proof :} Take $A^{(1)}_{[4321],e_1-e_2,1}=\frac{q^2+q+1}{q^3+2q^2+1}, \ A^{(1)}_{[4321],e_1-e_2,2}=q-1, \ A^{(1)}_{[4321],e_3-e_4,1}=1$.

\bigskip

This finishes the proof of Theorem \ref{thmGL22-conj}.

\subsection{The case $G=\U(3)_{\FF_q}$ and $(r,s)=(2,1)$} \label{sec-U21-proof}

It suffices to prove Proposition \ref{propU21inert}. Here are the elements of $W$ appearing in the proof.

\bigskip

\begin{figure}[H]
$\xymatrix@R=10pt{
&   *++[F]{[213]} \ar@{-}[rd] &\\
  *++[F]{[123]}\ar@{-}[ru]\ar@{-}[rd] & &*++[F]{[231]}\\
&   *++[F]{[132]}\ar@{-}[ru]  &
}$
    \caption{The strata appearing in the proof}
    \label{fig:U21inert}
\end{figure}

We start with the Hasse cones $\overline{C}_{\Hasse,w}$ for the three elements of length $1$. 

\bigskip

{\renewcommand{\arraystretch}{1.5}%
\noindent\begin{tabular}{ |c|c|c|l|c| } 
\hline
$w$  & $E_w $  &  $C_{\Hasse,w}$ & $\overline{C}^{+}_w$  \\
\hline
$[132]$ & $e_2-e_3$ & $q a_1-a_2 \leq 0$ & same \\
\hline
$[213]$ & $e_1-e_2$ & $- a_1+a_2 \leq 0$ & same \\
\hline
\end{tabular}} 

\bigskip

{\renewcommand{\arraystretch}{1.2}%
\noindent\begin{tabular}{ |c|c|c|c|c|c|l| } 
\hline 
\multicolumn{7}{|l|}{Relevant elements of length $2$} \\
\hline
$w$ & $E_w$ & $\EE_w$ & $\LL_w$ &$\{\chi_\alpha\}_{\alpha\in \EE_w}$ & $\{h_w(\chi_\alpha)\}_{\alpha\in \EE_w}$   & Upper bound for $\overline{C}^{+}_w$\\
\hline
\multirow{2}{3em}{$[231]$} & \multirow{2}{3em}{$e_1-e_3$\\$e_2-e_3$} & \multirow{2}{3em}{$e_1-e_3$\\$e_2-e_3$} & \multirow{2}{3em}{$[132]$\\$[213]$}  &\multirow{2}{3em}{$(1,0)$\\$(0,1)$}& \multirow{2}{7em}{$(0,-(q+1))$\\$(1-q,1)$} & $a_1\leq 0$ \\ 
&&&&&&  \\ 
\hline
\end{tabular}}

\paragraph{Proof :} Take $A^{(1)}_{[231],e_1-e_3,1}=\frac{1}{q-1}, \ A^{(1)}_{[231],e_2-e_3,1}=\frac{1}{q-1}$.

\subsection{The case $G=\U(4)_{\FF_q}$ and $(r,s)=(3,1)$} \label{sec-U31-proof}
It suffices to prove Proposition \ref{prop-U31-inert-van}. Here are the elements of $W$ relevant for the proof:

\begin{figure}[H]
$\xymatrix@R=10pt{
&   *++[F]{[1243]} \ar@{-}[rd] &&&&\\
  *++[F]{[1234]}\ar@{-}[r]\ar@{-}[ru]\ar@{-}[rd] & *++[F]{[1324]}\ar@{-}[rd] &*++[F]{[1342]}\ar@{-}[r]&*++[F]{[2341]}\ar@{-}[r]&*++[F]{[3241]} \ar@{-}[r]&*++[F]{[3421]}\\
&   *++[F]{[2134]}\ar@{-}[r]\ar@{-}[ru]  &*++[F]{[2314]} \ar@{-}[ru]&&& 
}$
    \caption{The strata appearing in the proof}
    \label{fig:U31inert}
\end{figure}

We start with the Hasse cones $\overline{C}_{\Hasse,w}$ for the three elements of length $1$. 

\bigskip

{\renewcommand{\arraystretch}{1.5}%
\noindent\begin{tabular}{ |c|c|c|l|c| } 
\hline
$w$  & $E_w $  &  $C_{\Hasse,w}$ & $\overline{C}^{+}_w$  \\
\hline
$[1243]$ & $e_3-e_4$ & $q a_1-a_3 \leq 0$ & same \\
\hline
$[1324]$ & $e_2-e_3$ & $-q a_1+(q-1)a_2+a_3 \leq 0$ & same \\
\hline
$[2134]$ & $e_1-e_2$& $-a_1-(q-1)a_2+qa_3 \leq 0$ & same \\
\hline
\end{tabular}} 

\bigskip

{\renewcommand{\arraystretch}{1.2}%
\noindent\begin{tabular}{ |c|c|c|c|c|c|l| } 
\hline 
\multicolumn{7}{|l|}{Relevant elements of length $2$} \\
\hline
\multirow{2}{3em}{$[1342]$} &\multirow{2}{3em}{$e_2-e_4$ \\ $e_3-e_4$}&\multirow{2}{3em}{$e_2-e_4$ \\$e_3-e_4$}  & \multirow{2}{3em}{$[1243]$\\$[1324]$}&\multirow{2}{3em}{$(0,1,0)$\\$(0,0,1)$}& \multirow{2}{8em}{$(0,-q,-1)$\\$(1-q,1,1)$}& $q^2a_1+a_2-qa_3\leq 0$ \\ 
&&&&&&   \\ 
\hline
\multirow{2}{3em}{$[2314]$} &\multirow{2}{3em}{$e_1-e_3$\\$e_2-e_3$}& \multirow{2}{3em}{$e_1-e_3$\\$e_2-e_3$}  & \multirow{2}{3em}{$[1324]$  \\ $[2134]$} &\multirow{2}{4em}{$(1,0,0)$\\$(0,1,0)$}& \multirow{2}{8em}{$(0,-1,-q)$\\$(0,-q,-1)$}& $-(q-1)a_1-a_2+qa_3\leq 0$ \\ 
&&&& & & $-qa_1+(q-1)a_2+a_3\leq 0$ \\
\hline
\end{tabular}}

\paragraph{Proof :} Take
$A^{(1)}_{[1342],e_2-e_4,1}=\frac{q^2-q+1}{q-1}$,  $A^{(1)}_{[1342],e_3-e_4,1}=\frac{1}{q-1}$,
$A^{(1)}_{[2314],e_1-e_3,1}=\frac{q(q-2)}{q^2-1}$,  $A^{(1)}_{[2314],e_2-e_3,1}=\frac{q^2-q+1}{q^2-1}$,
$A^{(2)}_{[2314],e_1-e_3,1}=1$.

\bigskip

{\renewcommand{\arraystretch}{1.2}%
\noindent\begin{tabular}{ |c|c|c|c|c|c|l| } 
\hline 
\multicolumn{7}{|l|}{Relevant elements of length $3$} \\
\hline
$w$ & $E_w$ & $\EE_w$ & $\LL_w$ &$\{\chi_\alpha\}_{\alpha\in \EE_w}$ & $\{h_w(\chi_\alpha)\}_{\alpha\in \EE_w}$   & Upper bound for $\overline{C}^{+}_w$\\
\hline
\multirow{3}{3em}{$[2341]$} & \multirow{3}{3em}{$e_1-e_4$\\$e_2-e_3$\\$e_3-e_4$} & \multirow{3}{3em}{$e_1-e_4$\\$e_3-e_4$} & \multirow{3}{3em}{$[1342]$\\$[2314]$}  &\multirow{3}{3em}{$(1,0,0)$\\$(0,0,1)$}& \multirow{3}{7em}{$(0,-1,-q)$\\$(1-q,1,1)$} & $a_1\leq 0$ \\ &&&&&&  $a_1-a_2+qa_3\leq 0$  \\ 
&&&&&&   $qa_1+(q^2-q+1)a_2-a_3\leq 0$ \\ 
\hline
\end{tabular}}

\paragraph{Proof :} Take
$A^{(1)}_{[2341],e_1-e_4,1}=\frac{1}{q^2-q+1}$, 
$A^{(1)}_{[2341],e_3-e_4,1}=\frac{1}{q^2-q+1}$, 
$A^{(2)}_{[2341],e_1-e_4,1}=\frac{q}{q^2-q+1}$, 
$A^{(2)}_{[2341],e_3-e_4,1}=\frac{q^2+1}{q^2-q+1}$, 
$A^{(3)}_{[2341],e_1-e_4,1}=\frac{q^2}{q^2-q+1}$, 
$A^{(3)}_{[2341],e_3-e_4,2}=\frac{q^3-q^2+q-1}{q^2-q+1}$.

\bigskip

{\renewcommand{\arraystretch}{1.2}%
\noindent\begin{tabular}{ |c|c|c|c|c|c|l| } 
\hline 
\multicolumn{7}{|l|}{Relevant elements of length $4$} \\
\hline
$w$ & $E_w$ & $\EE_w$ & $\LL_w$ &$\{\chi_\alpha\}_{\alpha\in \EE_w}$ & $\{h_w(\chi_\alpha)\}_{\alpha\in \EE_w}$   & Upper bound for $\overline{C}^{+}_w$\\
\hline
\multirow{3}{3em}{$[3241]$} & \multirow{3}{3em}{$e_1-e_2$\\$e_2-e_4$\\$e_3-e_4$} & \multirow{3}{3em}{$e_1-e_2$} & \multirow{3}{3em}{$[2341]$}  &\multirow{3}{3em}{$(1,0,0)$}& \multirow{3}{7em}{$(0,0,-(q+1))$} & $a_1\leq 0$ \\ 
&&&&&& $a_1+(q-1)a_2\leq 0$  \\ 
&&&&&&   \\ 
\hline
\end{tabular}} 

\paragraph{Proof :} Take $A^{(1)}_{[3241],e_1-e_1,1}=1, A^{(2)}_{[3241],e_1-e_2,2}=\frac{1}{q^2+1}, \ 
A^{(2)}_{[3241],e_1-e_2,3}=\frac{q}{q^2+1}$.

\bigskip

{\renewcommand{\arraystretch}{1.2}%
\noindent\begin{tabular}{ |c|c|c|c|c|c|l| } 
\hline 
\multicolumn{7}{|l|}{Relevant elements of length $5$} \\
\hline
$w$ & $E_w$ & $\EE_w$ & $\LL_w$ &$\{\chi_\alpha\}_{\alpha\in \EE_w}$ & $\{h_w(\chi_\alpha)\}_{\alpha\in \EE_w}$   & Upper bound for $\overline{C}^{+}_w$\\
\hline
\multirow{3}{3em}{$[3421]$} & \multirow{3}{3em}{$e_1-e_3$ \\ $e_2-e_3$\\$e_3-e_4$} & \multirow{3}{3em}{$e_2-e_3$} & \multirow{3}{3em}{$[3241]$}  &\multirow{3}{4em}{$(0,1,0)$}& \multirow{3}{5em}{$(1,1-q,1)$} & $(q-1)a_1+a_2 \leq 0$ \\ 
&&&&&&  \\ 
&&&&&&  \\ 
\hline
\end{tabular}}

\paragraph{Proof :} Take $A^{(1)}_{[3421],e_2-e_3,1}=\frac{q(q-2)}{q-1}, A^{(1)}_{[3421],e_2-e_3,2}=\frac{1}{q-1}$. This terminates the proof of Proposition \ref{prop-U31-inert-van}.

\bibliographystyle{test}
\bibliography{biblio_overleaf}

\noindent
Wushi Goldring\\
Department of Mathematics, Stockholm University, SE-10691, Sweden\\
wushijig@gmail.com\\ 

\noindent
Jean-Stefan Koskivirta\\
Department of Mathematics, Faculty of Science, Saitama University, 
255 Shimo-Okubo, Sakura-ku, Saitama City, Saitama 338-8570, Japan \\
jeanstefan.koskivirta@gmail.com

\end{document}